\documentclass[11pt,reqno,english]{amsart}
\usepackage{amsfonts,amsmath,latexsym,verbatim,amscd,mathrsfs,color,array}
\usepackage[colorlinks=true]{hyperref}
\usepackage[hmargin=3cm, vmargin=3cm]{geometry}

\usepackage{amsmath,amssymb,amsthm,amsfonts,graphicx,color}
\usepackage{amssymb}
\usepackage{pdfsync}
\usepackage{epstopdf}

\usepackage{amsfonts,amsmath,latexsym,verbatim,amscd,mathrsfs,color,array}

\makeatletter
\def\@currentlabel{2.1}\label{e:dispaa}
\def\@currentlabel{2.21}\label{e:dispau}
\def\@currentlabel{2.22}\label{e:dispav}
\def\@currentlabel{2.23}\label{e:dispaw}
\def\@currentlabel{2.24}\label{e:dispax}
\def\theequation{\thesection.\@arabic\c@equation}
\makeatother

\newcommand{\inn}{{\quad\hbox{in } }}
\newcommand{\onn}{{\quad\hbox{on } }}
\newcommand{\ttt}{\tilde }

\newcommand{\nn}{ {\nabla}  }

\newcommand{\pp}{ {\partial} }

\newcommand{\vp}{\varphi}

\newcommand{\R} {\mathbb R}
\newcommand{\cuad}{{\sqcap\kern-.68em\sqcup}}

\newcommand{\foral}{\quad\mbox{for all}\quad}
\newcommand{\ass}{\quad\mbox{as}\quad}

\newcommand{\ve}{\varepsilon}

\newcommand{\be}{\begin{equation}}
\newcommand{\ee}{\end{equation}}

\newcommand{\la}{\lambda}

\newcommand{\equ}[1]{(\ref{#1})}

\renewcommand{\theequation}{\thesection.\arabic{equation}}
 \newtheorem{lema}{Lemma}[section]
 \newtheorem{lemma}{Lemma}[section]

\newtheorem{teo}{Theorem}

\newtheorem{prop}{Proposition}[section]
\newtheorem{corollary}{Corollary}[section]
\newtheorem{remark}{Remark}[section]
\newcommand{\bremark}{\begin{remark} \em}
\newcommand{\eremark}{\end{remark} }

\long\def\red#1{{\color{black}#1}}

\title [infinite-time bubbling in the critical nonlinear heat equation]
{Green's function and infinite-time bubbling in the critical nonlinear heat equation}

\author[C. Cort\'azar]{Carmen Cort\'azar}
\address{\noindent   Departmento of Matem\'aticas, Pontificia Universidad Cat\'olica de Chile, Santiago, Chile.}
\email{\red{ccortaza@mat.puc.cl }}

\author[M. del Pino]{Manuel del Pino}
\address{\noindent   Departamento de
Ingenier\'{\i}a  Matem\'atica and Centro de Modelamiento
 Matem\'atico (UMI 2807 CNRS), Universidad de Chile,
Casilla 170 Correo 3, Santiago,
Chile.}
\email{delpino@dim.uchile.cl}

\author[M. Musso]{Monica Musso}
\address{\noindent   Departmento of Matem\'aticas, Pontificia Universidad Cat\'olica de Chile, Santiago, Chile.}
\email{\red{mmusso@mat.puc.cl }}

\begin{document}

\date{}\maketitle


\begin{abstract}
Let $\Omega$ be  a smooth bounded domain in $\R^n$, $n\ge 5$.
We consider the semilinear heat equation at the critical Sobolev exponent
$$
u_t = \Delta u + u^{\frac{n+2}{n-2}} \inn \Omega\times (0,\infty), \quad u =0 \onn   \pp\Omega\times (0,\infty).
$$
 Let $G(x,y)$ be the Dirichlet Green's function of
$-\Delta$ in $\Omega$ and $H(x,y)$ its regular part.
Let  $q_j\in \Omega$, $j=1,\ldots,k$, be points such that the matrix
$$
 \left [ \begin{matrix} H(q_1, q_1) &  -G(q_1,q_2) &\cdots & -G(q_1, q_k) \\   -G(q_1,q_2) &  H(q_2,q_2) &  -G(q_2,q_3) \cdots & -G(q_3,q_k) \\ \vdots &  & \ddots& \vdots   \\  -G(q_1,q_k) &\cdots&  -G(q_{k-1}, q_k) &  H(q_k,q_k)            \end{matrix}    \right ]
$$
is positive definite. For any $k\ge 1$ such points indeed exist. We prove the existence of a positive smooth solution $u(x,t)$ which blows-up by bubbling in infinite time
 near those points. More precisely,
for large time $t$, $u$ takes the approximate form
$$
u(x,t)   \approx  \sum_{j=1}^k \alpha_n \left ( \frac {   \mu_j(t)}  { \mu_j(t)^2  +  |x-\xi_j(t)|^2 } \right )^{\frac {n-2}2}  .
$$
Here $\xi_j(t) \to q_j$ and $0<\mu_j(t) \to 0$, as $t \to \infty$.
We find that $\mu_j(t) \sim t^{-\frac 1{n-4}} $ as $t\to +\infty$, when $n\geq 5$.

\end{abstract}










\date{}\maketitle

\setcounter{equation}{0}
\section{Introduction and statement of the main result}

Let $\Omega$ be a bounded, smooth domain in $\R^n$, $n\ge3$.  This paper deals with the asymptotic behavior of positive, classical and globally defined in time  solutions $u(x,t)$ of  the critical
semilinear heat equation
\begin{align}\label{P}
u_t   &= \Delta u+ u^{\frac {n+2}{n-2}}   \inn \Omega\times (0, \infty), \\
u &=0 \onn  \partial \Omega\times (0, \infty )  \nonumber\\
u(\cdot ,0) &=u_0\inn  \Omega .\nonumber
\end{align}
where $u_0$ is a positive, smooth  initial datum.

\medskip
The aim of this paper is to construct solutions which exhibit {\em infinite time blow-up}. Problem \equ{P} is a model of parabolic gradient flow for an energy
that loses compactness in the form of {\em bubbling}. The energy
\be\label{E}
E(u) = \frac 1{2}\int  |\nn u|^2 dx\,  -\,   \frac{n-2}{2n}\int  |u|^{\frac {2n}{n-2}}  dx
\ee
is a Lyapunov functional for \equ{P}. In fact, for solution $u(x,t)$ of \equ{P}
we have
$$
\frac{d}{dt} E(u(\cdot ,t)) =  -  \int |u_t|^2 dx
$$
and hence $E$ is  strictly decreasing along trajectories of \equ{P}. Alternatively, \equ{P} corresponds to the formal negative $L^2$-gradient flow for $E$ in
$H_0^1(\Omega)$.

\medskip
A special characteristic of the exponent $\frac{n+2}{n-2}$ in the functional $E$ is
the presence of an asymptotically singular continuum of entire, energy invariant critical points:
let us denote
\be\label{bubble}
U_0(y) = \alpha_n \left ( \frac 1 {1+|y|^2} \right )^{\frac {n-2}2}  , \quad \alpha_n = (n(n-2))^{\frac 1{n-2}} .
\ee
$U_0$ is a positive solution of
$$
\Delta U  +  U^{\frac{n+2}{n-2}} = 0 \inn \R^n
$$
 In fact all positive entire solutions of this equation are given by the {\em Talenti bubbles}
 \be
 U_{\mu, \xi} (x) = \mu^{-\frac{n-2}2} U_0\left ( \frac{x-\xi} {\mu}   \right ) = \alpha_n \left ( \frac {\mu}  { \mu^2 +|x|^2} \right )^{\frac {n-2}2}.
 \label{bubbles}\ee
 which correspond to the extremals of Sobolev's embedding \cite{t}.
This family is energy invariant: for all $\mu,\xi$,
 \be\label{Sn}
 E(U_{\mu, \xi}) = E(U_0) =: S_n >0 .
 \ee
 These functions reflect exactly the loss of compactness of the critical Sobolev embedding
  and thus of the functional $E$: in fact, a Palais-Smale sequence $u_n\in H_0^1(\Omega)$, namely one along which $E$ is bounded and $\nn E$ goes to zero
  must asymptotically be, passing to a subsequence, of the form
 \be\label{bub}
 u_n = u_\infty +  \sum_{i\in I} U_{\mu_n^i, \xi^i_n} + o(1) , \quad
 \ee
 for finite set $I$, a critical point $u_0\in H_0^1(\Omega)$ of $E$, and $\mu_n^i \to 0$, \cite{struwe}. From Pohozaev's identity we must necessarily have $u_\infty=0$ if the domain is star-shaped. This type of behavior, sometimes called bubbling is precisely the one expected along sequences $t=t_n\to+\infty$ for a solution $u(x,t)$ of \equ{P} globally defined in time,
 in particular for the so-called {\em threshold solutions} which we describe next.

 Let $\vp$ be a positive, smooth function defined in $\Omega$, and $u_\alpha (x,t)$
the solution  of Problem \equ{P} with initial datum  $u(x,0) = \alpha \vp(x)$. For small $\alpha>0$, the effect of the nonlinearity is negligible, and $u_\alpha$ goes uniformly to zero as time goes to infinity. Letting
$$ \alpha_* = \sup \{\alpha >0 \ /\  \lim_{t\to \infty} \| u_\alpha(\cdot, t)\|_\infty = 0  \},  $$
then $0<\alpha_*<+\infty$.
Ni, Sacks and Tavantzis, \cite{nst} found that $u_{\alpha_*} (x,t)$    is a well-defined $L^1$-weak solution of \equ{P}.  Since for large values of $\alpha$, finite
 time blow-up is known to happen, $u_{\alpha_*}$ is a type of solution which loosely speaking lies in the dynamic threshold between solutions globally defined in time and those that blow-up in finite time. It is known, see Du \cite{du} and Suzuki \cite{suzuki} that along sequences $t_n\to +\infty$,  $u_n = u_{\alpha_*} (x, t_n)$  does have (up subsequences) a {\em bubble resolution} of the type \equ{bub}.

\medskip
Galaktionov and V\'azquez \cite{gv} found that in the case that $\Omega=B(0,1)$ and the threshold solution $u_{\alpha*}$ is radially symmetric, then no finite time singularities for   $u_{\alpha*}(r,t)$ occur and it must become unbounded as $t\to +\infty$, thus exhibiting infinite-time blow up.  Galaktionov and King discovered in \cite{gk2} that this radial blow-up solution does have a bubbling asymptotic profile as $t\to +\infty$  of the form

\be\label{bubble}
u_{\alpha_*}(r,t)  \ \approx \  \alpha_n \left ( \frac {   \mu(t)}  { \mu(t)^2  +  r^2 } \right )^{\frac {n-2}2} , \quad r=|x|.
\ee
where for $n\ge 5$, $\mu(t) \sim t^{-\frac 1{n-4}}\to 0$. For $\alpha> \alpha_*$ blow-up in finite time of $u_{\alpha_*}(r,t)$ occurs while, it goes to zero when $\alpha< \alpha_*$. Understanding this threshold phenomenon is central in capturing the global dynamics of Problem \equ{P}. These solutions are unstable, while intuitively codimension-one stable in the space of initial conditions containing $\alpha_*\vp$.

 \medskip
Nothing seems to be  known however on existence of infinite-time bubbling solutions in the nonradial case, or about their degree of stability. Our main goal in this paper is to build solutions with single or multiple blow-up points as $t\to +\infty$ in problem \equ{P} when $\Omega$ is arbitrary and $n\ge 5$, providing precise account of their asymptotic form and investigate their stability.

\medskip
Our construction unveils the interesting role played by the (elliptic) Green function of the domain $\Omega$.
In what follows we denote by
 $G(x,y)$  Green's function for the boundary value problem
$$
- \Delta_x G(x,y)   = c_n\delta(x-y) \inn \Omega, \quad G(\cdot ,y ) = 0 \quad\hbox{ on }  \pp\Omega, $$
where $\delta(x)$ is the Dirac mass at the origin and $c_n$ is the number such that
\be\label{fund}
- \Delta_x \Gamma(x)   = c_n\delta(x), \quad   \Gamma(x) =   \frac {\alpha_n}{  |x|^{n-2}}, \quad
\ee
namely $c_n = (n-2) \omega_n\alpha_n$ with $\omega_n$ the surface area of the unit sphere in $\R^n$ and $\alpha_n$ the number in \equ{bubble}.
We let $H(x,y)$ be the regular part of $G(x,y)$ namely the solution of the problem
\be
- \Delta_x H(x,y)   =  0 \inn \Omega, \quad H(\cdot,y ) =    \Gamma (\cdot -y)  \inn \pp\Omega. \label{H}\ee
The diagonal $H(x,x)$ is called the Robin function of $\Omega$. It is well known that it satisfies
\be
 H(x,x) \to  +\infty \quad\hbox{as } {\rm dist} \, (x,\pp\Omega) \to 0. \label{ojo}\ee
Let $q= (q_1, \ldots, q_k$) be  an array of $k$  distinct points in $\Omega$.
\be\label{gxi}
 \mathcal G (q) = \left [ \begin{matrix} H(q_1,q_1) &  -G(q_1,q_2) &\cdots & -G(q_1,q_k) \\   -G(q_1,q_2) &  H(q_2,q_2) &  -G(q_2,q_3) \cdots & -G(q_3,q_k) \\ \vdots &  & \ddots& \vdots   \\  -G(q_1,q_k) &\cdots&  -G(q_{k-1},q_k) &  H(q_k,q_k)            \end{matrix}    \right ]
\ee

Our main result states
that  a global solution to \equ{P}  which blows-up at exactly $k$ given points  $q_j$  exists if  $q$ lies in the open region of $\Omega^k$  where the matrix $\mathcal G (q)$ is positive definite.

\begin{teo}\label{teo1}
Assume $n\ge 5$.
Let $q_1, \cdots, q_k$ be  distinct points in $\Omega$  such that the matrix $\mathcal G (q) $ is positive definite.
Then  there exist an initial datum $u_0$ and smooth functions $\xi_j(t)\to q_j$ and $0<\mu_j(t)\to 0$, as $t\to+\infty$, $j=1,\ldots, k$, such that the solution $u_q$ of Problem  $\equ{P}$ has the form
\be\label{forma1}
u_q(x,t)   =  \sum_{j=1}^k \alpha_n \left ( \frac {   \mu_j(t)}  { \mu_j(t)^2  +  |x-\xi_j(t)|^2 } \right )^{\frac {n-2}2}  -  \mu_j(t)^{\frac {n-2}2} H( x, q_j) + \mu_j(t)^{\frac {n-2}2}\theta(x,t)  ,
\ee
where $\theta(x,t)$ is bounded, and  $\theta(x, t)\to 0$ as $t\to +\infty$, uniformly away from the points $q_j$.
In addition, for certain positive constants $\beta_j$ depending on $q$.
$$
\mu_j(t)  =       \beta_j t^{-\frac 1{n-4}} (1+o(1)) , \quad  \xi_j (t)-q_j = O(t^{-\frac 2{n-4}})  \quad\hbox{as } t\to +\infty
$$

\end{teo}

\medskip
Our construction of the solution $u_q(x,t)$ in Theorem \ref{teo1}  yields the codimension $k$-stability of its bubbling phenomenon in the following sense.

\begin{corollary} \label{corol1}
Let $u_q$ be the solution found in Theorem \ref{teo1}.
There exists a codimension $k$ manifold in $C^1(\bar\Omega)$  that contains $u_q(x,0)$ such that if $u_0$ lies in that manifold and it is sufficiently close to $u_q(x,0)$,
then the solution $u(x,t)$ of problem $\equ{P}$ has exactly $k$ bubbling points $\ttt q_j$, $j=1, \ldots ,k$ which lie close to the $q_j$.  It has the form $\equ{forma1}$ with $q$ replaced by $\ttt q$.
\end{corollary}

\medskip
Positive definiteness of $\mathcal G (q)$ trivially holds if $k=1$. For $k=2$ this condition holds  if and only if $$H(q_1,q_1)H(q_2,q_2) -G(q_1,q_2)^2 >0, $$ in particular it does not hold if
 both points $q_1$ and $q_2$ are too close to a given point in $\Omega$.  Given $k>1$, using \equ{ojo},   we can always find $k$ points where $\mathcal G (q)$ is positive definite: it suffices to take points located at a uniformly positive distance one to each other, and then let them lie sufficiently close to the boundary.

The role of the matrix $\mathcal G (q)$ in elliptic bubbling phenomena for perturbations of the stationary version of \equ{P} has been known for a long time, after the works
\cite{bc,blr}. In particular in \cite{blr} A. Bahri, Y.-Y. Li and O. Rey analyze bubbling solutions for slightly subcritical perturbations, and positivity of the matrix already appears there. To illustrate why in the parabolic case such a condition is needed, we invoke a computation in \cite{dfm}, that
shows that if
$$
u_0(x) = \sum_{j=1}^k U_{\mu_j, \xi_j}(x) -\mu_j^{\frac {n-2}2} H(x ,\xi_j)
$$
with all $\mu_j$'s small and of comparable order, then
$$
E(u_0) =  kS_n +  \alpha \big (\, \sum_{j=1}^k \mu_j^{n-2} H(\xi_i,\xi_i)  -\sum_{i\ne j} G(\xi_i,\xi_j) \mu_i^{\frac{n-2}2}  \mu_i^{\frac{n-2}2} \big ) +   \hbox{ smaller terms}.
$$
Here $E$ is the energy \equ{E},  $S_n$ is given by \equ{Sn} and  $\alpha$ is a positive constant.
Since $E$ is decreasing along the solution of \equ{P} then we may only end at the $k$-bubble energy $kS_n$ as $t\to\infty$ if $E(u_0) >   kS_n$. If the matrix $\mathcal G (q)$ is positive definite that fact is guaranteed. A formal consideration of balancing needed for the functions $\mu_j(t)$ that we carry out in  \S \ref{sec3} will also lead us to the necessity of that condition to have a solution of the form
\equ{forma1}.

\medskip
The proof consists of building a first approximation to the solution, then solving for a small remainder by means of linearization and fixed point arguments.
In   \S \ref{sec3} we shall construct the first approximation of the form \equ{forma1}. We shall compute the error and  will see that in order to improve the approximation we need solvability conditions  for the elliptic linearized operator around the bubble.   These relations yield a system of ODEs for the scaling parameters, of which we find a suitable solution.
 After this has been achieved, we solve the full problem as a small perturbation by an {\em inner-outer gluing scheme} which is described in \S \ref{sec3new}. This method consists of decomposing the perturbation in the form $\sum_j \eta_j \ttt\phi_j +\psi$ where $\eta_j$ is a smooth cut-off function that vanishes away from the concentration point $q_j$. The tuple $(\ttt \phi, \psi)$ will satisfy a coupled nonlinear system where the operator for $\psi$ is just a small perturbation of the standard heat operator, and the equation for the
$\ttt \phi_j$ involves the parabolic linearized equation at the scale of the bubble. We solve the outer problem for $\psi$, for a given suitable decaying $\ttt \phi$ and derive estimates for it in \S \ref{outer}.

\medskip
A delicate matter is the construction of an inverse of the operator at $\ttt \phi_j$  (provided that certain solvability conditions for the errors hold) so that the solution has a sufficiently fast decay. We do this in  \S \ref{seclineartheory}.  The solvability conditions are achieved after adjusting lower order terms of the parameters in \S \ref{sec5}.  After this is achieved the coupling gets sufficiently weak so that contraction mapping principle applies to finally solve the problem. Arbitrary small initial data for $\psi$ can be imposed, while the inverse built for the linear operator at $\phi_j$ requires one linear constraint in the initial condition.  This is ultimately the reason for the codimension $k$-stability of the bubbling solution found. We prove the main result in \S \ref{final}. The inner-outer gluing technique has been a very useful tool in finding solutions for singular perturbation elliptic problems with higher dimensional concentration set, as developed in \cite{dkw,dkw1,dmp}. We believe it applies to the construction of bubbling solutions in various parabolic flows where this critical loss of compactness arises, and in the construction of type II blow up in various problems.
The machinery employed is parabolic in nature, but we believe the inner-outer gluing principle may be applicable to other evolution problems.

\medskip
Before proceeding into the proof, we make some further bibliographic comments.
Large literature has been devoted to the more general equation with a power nonlinearity $$u_t= \Delta u + |u|^{p-1}u, $$ sometimes called the Fujita equation, after \cite{fujita}.
   on asymptotic behavior and finite-time blow-up.
 We refer the reader to the book by Quittner and Souplet \cite{qs} for
a comprehensive survey of results until 2007.
In particular, the role of  the Sobolev critical exponent $p=\frac{n+2}{n-2}$ in existence and characteristics of blow-up phenomena has been broadly considered, see also Matano and Merle \cite{matanomerle} for a more recent classification of radial blow-up and supercritical powers and Schweyer \cite{schweyer} for a construction of radial type-II blow-up in the critical case.  In most results available of construction of solutions radial symmetry is a central feature.
Our main result shares the flavor of that by Merle and Zaag \cite{mz}, where multiple-point, finite time  type I blow-up and its stability was found in the subcritical case.
We point out that the corresponding energy-critical wave equation,
$$
u_{tt} = \Delta u + |u|^{\frac 4{n-2}}u
$$
and the role of the Talenti bubbles \equ{bubbles} in its dynamics, in particular bubbling blow-up has been the subject of many studies. See for instance
\cite{dkm,kenigmerle,kst}.

\medskip
A natural question is whether or not there exist solutions with {\em multiple bubbling}, namely with the form of a superposition of
bubbles of different rates around a given point.
From a result by Schoen, this is not possible in the slightly subcritical stationary version of \equ{P}, see \cite{blr}.
We believe this is also the case in \equ{P} as $t\to +\infty$. On the other hand
when $t\to -\infty$, ancient solutions with this pattern may exist. This is indeed the case for the conformally invariant variation of \equ{P},  the  {\em Yamabe flow in $\R^n$},
$$
	(u^{\frac{n+2}{n-2}})_t   = \Delta u  + u^{\frac{n+2}{n-2} }\inn \R^n\times (-\infty, 0]
	.$$
corresponding to the conformal evolution of metrics by scalar curvature. It has been proven in \cite{dd} that  there exist radially symmetric ``ancient towers of bubbles''. In the elliptic case they appear at slightly supercritical powers, see \cite{ddm}.

\setcounter{equation}{0}
\section{Construction of the approximate solution and error computations}\label{sec3}

We consider the Talenti bubbles \equ{bubbles} which we recall are given by
\begin{equation}
	\label{defbubble}
	U(y) = \alpha_n \left( {1 \over 1+ |y|^2 } \right)^{n-2 \over 2}, \quad \alpha_n= (n(n-2))^ {\frac {n-2}4},\quad
\end{equation}
and
$$
 U_{\mu , \xi } (x) = \mu^{-{n-2 \over 2}} U\left ({x-\xi  \over \mu }\right ), \quad \mu>0, \quad \xi\in \R^n.
$$
Given $k$ points $q_1, \ldots, q_k\in \R^n$, we want to find a solution $u(x,t)$ of equation \equ{P} with the approximate form
\be\label{oo}
u(x,t)  \approx \sum_{j=1}^k  U_{\mu_j(t) , \xi_j(t) } (x)
\ee
where $\xi_j(t) \to q_j$ and $\mu_j(t)\to 0$ as $t\to \infty$ for each $j=1,\ldots, k$.
The functions $\xi_j$ and $\mu_j$ cannot of course be arbitrary.  We make an ansatz for these parameters. To begin with, we assume that the vanishing speed of all functions
$\mu_j(t)$ is the same. More precisely, we assume that for a certain fixed positive function $\mu_0(t)\to 0 $ and positive constants $b_1,\ldots, b_k$ we have that
$$
\mu_j(t) = b_j \mu_0(t) + O(\mu_0^2(t)) \ass t\to \infty.
$$
At the same time we assume that
 $$
 \xi_j (t) - q_j =  O(\mu_0^2(t)) \ass t\to \infty.
 $$
Since the right hand side of expression \equ{oo} does not satisfy the zero boundary condition, we shall first identify a correction term which is consistent with this ansatz. Then we will determine what the functions $\mu_0(t)$ and the constants $b_j$ should be. Since, away from the concentration points $q_j$ the functions $U_{\mu_j , \xi_j}$ are uniformly small, we see that
$u_t$ should approximately satisfy
$$
u_t \approx \Delta u  +   \sum_{j=1}^k U_{\mu_j,q}(x) ^p
$$
Besides, we see that
$$
\int_\Omega U_{\mu_j,q}(x)^pdx \approx   \mu_j^{\frac{n-2}2 }  a_n, \quad    a_n := \int_{\R^n} U(y)^pdy ,
$$
and hence away from the points $q_j$ the equation should be well approximated by
$$
u_t \approx \Delta u  +    c_n \mu_0^{\frac{n-2}2 }    \sum_{j=1}^k   b_j^{\frac{n-2}2 }  \delta_{q_j}  \inn \Omega \times (0,\infty) .
$$
where $\delta_q$ is the Dirac mass at the point $q$.
Letting $u = \mu_0^{\frac{n-2}2 } v$ we get
$$
v_t  \approx  \Delta v  - \frac {n-2} 2 \mu_0^{-1} {\dot \mu_0}v  + c_n   \sum_{j=1}^k   b_j^{\frac{n-2}2 }  \delta_{q_j}  \inn \Omega \times (0,\infty) .
$$
where $\dot {\{ \}} = \frac d{dt} $.  We assume, as it will be a priori satisfied that  $\mu_0^{-1} {\dot \mu_0}\to 0$, which is the case for instance if $\mu_0 \sim t^{-a} $.
Hence
$$
v_t  \approx  \Delta v   +   a_n\sum_{j=1}^k   b_j^{\frac{n-2}2 }  \delta_{q_j}  \inn \Omega \times (0,\infty) ,
$$
$$
v = 0 \onn \pp \Omega \times (0,\infty) .
$$
This tells us that away from the points $q_j$ we should have
$$v(x,t) \approx a_n\sum_{j=1}^k   b_j^{\frac{n-2}2 } G(x,q_j)   , $$ in other
words
$$
u(x,t) \approx  \sum_{j=1}^ k    \frac {\alpha_n \mu_j^{\frac{n-2}2} }{|x-q_j|^{n-2}} -\mu_j^{\frac{n-2}2} H(x,q_j).
$$
Observing that for $x$ away from the point $q_j$,  we precisely have $$U_{\mu_j,\xi_j} (x) \approx  \frac {\alpha_n \mu_j^{\frac{n-2}2} }{|x-q_j|^{n-2}} $$
we see that a better global approximation to a solution $u(x,t)$ to our problem is given by the corrected $k$-bubble
\be\label{ansatz0}
u_{\xi , \mu} ( x,t) : =   \sum_{j=1}^ k u_j(x,t), \quad   u_j(x,t):=  U_{\mu_j,\xi_j} (x)    -\mu_j^{\frac{n-2}2} H(x,q_j).
\ee
We have obtained this correction term out of a rough analysis to what is happening away from the blow-up points. Let us now analyze the region near them. That will allow us
to identify the function $\mu_0(t)$ and the constants $b_j$.
It is convenient to write
$$
S(u) :=  -u_t + \Delta_x u  +   u^p .
$$
We consider the error of approximation $S(u_0)$. We have
$$
S( u_{\mu , \xi}  )  =  -  \sum_{i=1}^k   \pp_t u_i  + \left ( \sum_{i=1}^k  u_i \right ) ^p  - \sum_{i=1}^k  U_{\mu_i,\xi_i}^p .
$$
We obtain the following estimate near a given concentration point $q_j$, from where the formal asymptotic derivation of the unknown parameters will be a rather direct consequence.

\begin{lemma}\label{errore}
Let us fix an index $j$ and consider $x$ with $|x-q_j| \le \frac 12 \min_{i\ne l} |q_i-q_l|$.
Then setting $x= \xi_j + \mu_j y_j $, we have that for all $t$ large, the error of approximation
 $S(u_{ \mu , \xi})$ can be estimated as
$$
S(u_{ \mu , \xi}) =   \mu_j^{-\frac{n+2}2} \big ( \mu_j E_{0j}[\mu,\dot \mu_j] +  \mu_j E_{1j}[\mu,\dot \xi_j]  \ + \ \mathcal R_j \, \big)
$$
where
$$
 E_{0j}[\mu,\dot\mu_j]= pU(y_j )^{p-1}\big [ -\mu_j^{n-3} H(q_j ,q_j)    \, + \,   \sum_{i\ne j}  \mu_j ^{\frac{n-4}2} \mu_i^{\frac {n-2}2}  G( q_j,q_i ) \big  ]  \  +\    \dot \mu_j Z_{n+1}(y_j ) ,
$$
$$
 E_{1j}[\mu,\dot\xi_j] =   pU(y_j)^{p-1} \big [-\mu_j^{n-2} \nn_x H(q_j ,q_j)   \,  + \,  \sum_{i\ne j}  \mu_j^{\frac {n-2}2} \mu_i^{\frac {n-2}2}  \nn_x G( q_j,q_i )\, \big ] \cdot  y_j \ +\       \dot \xi_j \cdot \nn U(y_j) ,
$$
and
\be\label{Rj}
\mathcal R_j 
 \ =\   \frac { \mu_0^ng } {1+|y_j|^2} +
  \frac { \mu_0^{n-2} \vec g } {1+|y_j|^4} \cdot (\xi_j -q_j) +
  \mu_0^{n+2}  f\, +\,  \mu_0^{n-1} \sum_{i=1}^k   \dot \mu_i  f_{i} \,  +\,  \mu_0^{n} \sum_{i=1}^k   \dot \xi_i  \cdot \vec f_{i}
\ee
where the functions $f,f_i,\vec f_i$ are smooth, bounded functions of the tuple  $( y, \mu_0^{-1} \mu, \xi, \mu_jy_j )$, and   $g, \vec g$ of $( y, \mu_0^{-1} \mu, \xi )$.
Furthermore,
$$
Z_{n+1} (y) := {n-2 \over 2} U(y) + \nabla U(y) \cdot y.
$$

\end{lemma}

\begin{proof}
We  write
$$
u_{\mu,\xi}  (x,t) =  \sum_{i=1}^k  \mu_i^{-\frac{n-2}2} U(y_i)  -   \mu_i^{\frac{n-2}2} H(x,q_i), \quad y_i =\frac{x-\xi_i}{\mu_i} .
$$
and
$$
S(u_{\mu,\xi} ) =  S_1 + S_2
$$
where
\be \label{form1}
S_1 \ :=\   \sum_{i=1}^ k  \mu_i^{-\frac n2} \dot \xi_i \cdot \nn U(y_i)  +  \mu_i^{-\frac n2} \dot \mu_i Z_{n+1}(y_i)  + {n-2 \over 2}   \mu_i^{\frac{n-4}2} \dot \mu_i  H(x,q_i),
\ee
and
\be \label{form2}
S_2\ :=\    \left ( \sum_{i=1}^k  \mu_i^{-\frac {n-2}2}   U(y_i)  -  \mu_i^{\frac{n-2}2} H(x,q_i) \right )^p - \sum_{i=1}^k  \mu_i^{-\frac {n+2}2}   U(y_i)^p .
\ee

We further write
$$
S_2 = S_{21} + S_{22}
$$
where
$$
S_{21} :=  \mu_j^{-\frac {n+2}2} \left [\,  \left (  U(y_j)  + \Theta_j  \right )^p  -  U(y_j)^p \right ], \quad S_{22} := - \sum_{i\ne j }^k  \mu_i^{-\frac {n+2}2}   U(y_i)^p,
$$
with
\be\label{thetaj}
\Theta_j  =  -\mu_j^{n-2} H(x,q_j) +  \sum_{i\ne j}   \left (  \mu_j \mu_i^{-1} \right )^{\frac{n-2}2}  U(y_i)  - (\mu_j\mu_i)^{\frac{n-2}2} H(x,q_i).
\ee
We observe that $|\Theta_j| \lesssim \mu_0^{n-2}$ uniformly in small $\delta$. Hence in particular we may assume $U(y_j)^{-1}|\Theta_j| < \frac 12 $ in the considered region.
We can therefore Taylor expand
$$
S_{21} = \mu_j^{-\frac {n+2}2}   \left [\,  pU(y_j) ^{p-1} \Theta_j  +   p(p-1)\int_0^1 (1-s)(U(y_j) + s\Theta_j) ^{p-2} ds\, \Theta_j^2 \right ]
$$

We make some further expansion.
We have, for $i\ne j$,
\begin{align*}
U(y_i) = &\  U \left (   \frac { \mu_j y_j +  \xi_j - \xi_i} {\mu_i} \right )  =
\frac{\alpha_n \mu_i^{n-2}} { (|\mu_j y_j +  \xi_j - \xi_i|^2 + \mu_i^2 )^{\frac{n-2}2} } \ \\
= &\
 \frac{\alpha_n \mu_i^{n-2}} {| \mu_j y_j + x_j - x_i|^{{n-2}} }  + \mu_i^{n}f(\xi, \mu, \mu_jy_j)
\end{align*}
where $f$ is smooth in its arguments and  $f(q, 0, 0) = 0$. Then we can write
$$
\Theta_j =       -\mu_j^{n-2} H(x_j + \mu_j y_j ,q_j)     +   \sum_{i\ne j} (\mu_i\mu_j )^{\frac {n-2}2}  G(x_j + \mu_j y_j ,q_i )  + \mu_i^{n}f(\xi, \mu, \mu_jy_j).
$$
Further expanding, we get
\begin{align*}
\Theta_j\ =&\  -\mu_j^{n-2} H(q_j ,q_j)     +   \sum_{i\ne j} (\mu_i\mu_j )^{\frac {n-2}2}  G( q_j,q_i )   +   \mu_i^{n}f(\xi, \mu, \mu_jy_j) \quad \\
+ &\
\Big [-\mu_j^{n-2} \nn_x H(q_j ,q_j)     +   \sum_{i\ne j} (\mu_i\mu_j )^{\frac {n-2}2}  \nn_x G( q_j,q_i )\, \Big ] \cdot (\mu_j y_j + \xi_j -q_j )\quad
\\
+ &\  \
  \int_0^1     \Big \{ -\mu_j^{n-2} D^2_x H(q_j+ s(\xi_j -q_j +\mu_j y_j ) ,q_j)\quad
\\
+ &\ \
  \sum_{i\ne j} (\mu_i\mu_j )^{\frac {n-2}2}  D^2_x G( q_j+ s(\xi_j -q_j +\mu_j y_j) ,q_i )\, \Big\} \,  [\xi_j -q_j +\mu_j y_j ]^2(1-s)\,ds\, .
\end{align*}
We conclude that
\begin{align*}
\Theta_j        \ =&\  -\mu_j^{n-2} H(q_j ,q_j)     +   \sum_{i\ne j} (\mu_i\mu_j )^{\frac {n-2}2}  G( q_j,q_i )\  \\
&+
\Big [-\mu_j^{n-1} \nn_x H(q_j ,q_j)     +   \sum_{i\ne j}  \mu_j^{\frac n2} \mu_i^{\frac {n-2}2}  \nn_x G( q_j,q_i )\, \Big ] \cdot  y_j \\
 +&\
\mu_0^{n-2} (\xi_j -q_j)\cdot f(\xi, \mu_j y_j , \mu_0^{-1}\mu )  +  \mu_0^n  F(\xi, \mu_j y_j , \mu_0^{-1}\mu )[y_j]^2
\end{align*}
where $f$ anf $F$ are smooth in its arguments, and bounded.
On the other hand,
$$
S_{22} := - \sum_{i\ne j }^k  \mu_i^{-\frac {n+2}2}   U(y_i)^p=
 - \sum_{i\ne j }^k
 \frac{\alpha_n^p \mu_i^{\frac {n+2}2}} {| q_j - q_i|^{{n+2}} }  + \mu_i^{\frac{n+2} 2}f_i(\xi, \mu, \mu_jy_j)
$$
so that
$$
S_{22} :=    \mu_0^{\frac{n+2} 2}  f(\xi, \mu_0^{-1} \mu, \mu_jy_j)
$$
where $f$ is smooth in its arguments and  $f_i(q, 0, 0) = 0$.

\medskip
We also decompose $S_1 = S_{11} + S_{12}$ where
$$
S_{11} =  \mu_j^{-\frac n2} \dot \xi_j \cdot \nn U(y_j)     +    \mu_j^{-\frac n2} \dot \mu_j Z_{n+1}(y_j) ,
$$
$$
S_{12} =\sum_{i\ne j}   \mu_i^{-\frac n2} \dot \xi_i \cdot \nn U(y_i)  +  \mu_i^{-\frac n2} \dot \mu_i Z_{n+1}(y_i)  +  \sum_{i=1}^k {n-2 \over 2}  \mu_i^{\frac{n-4}2} \dot \mu_i  H(x,q_i)
$$
We can write
$$
S_{12} = \sum_{i\ne j}  \alpha_n  \mu_i^{\frac n2 -1} \dot \xi_i \cdot \left [ \frac {q_i -q_j } {| q_i -q_j|^n }  + \vec f_i(\xi, \mu, \mu_jy)  \right ]   \, +
\mu_i^{\frac {n-4} 2} \dot \mu_i \left [  \frac {c_n} { | q_i-q_j|^{n-2}} + f_i(\xi, \mu, \mu_jy) \right ]
$$
$$
 + \sum_{i=1}^k  \mu_i^{\frac{n-4}2} \dot \mu_i [ H(q_j,q_i) +  f_i (\mu_j y , \xi )]
$$
where $\vec f_i$ are smooth in their arguments vanishing in the limit. In total we can write
$$
S_{12} =     \mu_0^{\frac{n-4}2} \sum_{i=1}^k   \dot \mu_i  f_{i0} ( \mu_0^{-1} \mu, \xi, \mu_jy )  +  \mu_0^{\frac{n-2}2} \sum_{i=1}^k   \dot \xi_i  \cdot\vec f_{i1} ( \mu_0^{-1} \mu, \xi, \mu_jy )
 $$
for functions $f_{i0}$, $\vec f_{i1}$ smooth in their arguments. This concludes the proof of the Lemma.
\end{proof}

\medskip
On the other hand, in the region $|x-q_i | > \delta$ for any $i=1, \ldots , k$,  we can describe the error function $S (u_{\mu , \xi } )$ as follows
\be \label{form3}
S (u_{\mu , \xi } ) (x,t) = \mu_0^{n-4 \over 2} \sum_{j=1}^k \dot \lambda_j f_j +
\mu_0^{n-2 \over 2} \sum_{j=1}^k \dot \xi_j \cdot \vec f_j + \mu_0^{n+2 \over 2} f
\ee
where $f_j$, $\vec f_j$, $f$ are smooth and bounded functions of $( x, \mu_0^{-1} \mu, \xi )$. This fact is a direct consequence of \eqref{form1} and \eqref{form2}.

\bigskip
\subsection*{The choice of the parameters at main order and improvement of the approximation }
We are looking for a solution of our equation of the form
$$
u(x,t) = u_{\mu,\xi}(x,t)  + \ttt  \phi(x,t)
$$
where $\ttt \phi$ is globally small compared with $u_{\mu,\xi}$.
In terms of $\ttt \phi$ the equation reads
$$
0 = S( u_{\mu,\xi}  + \ttt  \phi )  =  - \pp_t \ttt  \phi + \Delta_x \ttt \phi +  pu_{\mu,\xi}^{p-1} \ttt\phi + S( u_{\mu,\xi}) + \ttt N_{\mu,\xi}(\ttt \phi)
$$
where
$$
\ttt N_{\mu,\xi}(\ttt \phi)=    (u_{\mu,\xi}+  \ttt\phi)^p - u_{\mu,\xi}^p -  pu_{\mu,\xi}^{p-1} \ttt\phi .
$$
It is reasonable to read locally around $q_j$ the function $\ttt \phi (x,t)$ in terms of the same scaling as that of $u_{\mu,\xi}$. We write
$$
\ttt \phi (x,t) =  \mu_j^{-\frac{n-2}2} \phi\left ( \frac {x-\xi_j} {\mu_j} , t \right )
$$
In this terms the equation for $\phi(y,t)$ becomes
\be\label{SS}
0 = \mu_j^{\frac{n+2}2} S( u_{\mu,\xi}  + \ttt  \phi   ) \ = \   \Delta_y \phi   + pU(y)^{p-1} \phi +   \mu_j^{\frac{n+2}2} S( u_{\mu,\xi})    + A[\phi]
\ee
where
\begin{eqnarray}\label{A}
A[\phi] =&&  - \mu_j^2 \pp_t \phi +  \mu_j\dot \mu_j [  \frac {n-2}2\phi  + y\cdot \nn_y\phi ]
+ \nabla \phi \cdot \dot \xi_j  \, +\,  [p(U(y) + \Theta_j ) ^{p-1} -  pU(y)^{p-1}]\,\phi \  \nonumber \\
+&&
(U(y) + \Theta_j + \phi  ) ^{p}- (U(y) + \Theta_j ) ^{p}-  p(U(y) + \Theta_j ) ^{p-1}\phi
\end{eqnarray}
and $\Theta_j$ is the function in \equ{thetaj}.  It is reasonable to assume that $\phi(y,t)$ decays in the $y$ variable and that for large $t$ the terms in $A(\phi)$ are comparatively small. Considering the largest term $E_0$ in the expansion of the error $\mu_j^{\frac{n+2}2} S( u_{\mu,\xi}) $ in the previous lemma we find
that $\phi(y,t)$ should equal at main order a solution $\phi_{0j}(y,t)$ of the elliptic equation
 \be\label{pico0}
 \Delta_y \phi_{0j} + p U^{p-1} \phi_{0j} =- \mu_{0j} E_{0j}[\mu,\dot\mu_j]  \inn \R^n , \quad \phi_0(y,t) \to 0 \ass |y|\to \infty .
\ee
At this point we recall some standard facts on a linear equation of the form
\be\label{pb}
L_0(\psi) := \Delta_y \psi + p U^{p-1} \psi = h(y) \inn \R^n , \quad \psi(y) \to 0 \ass |y|\to \infty .
\ee
It is well known that all bounded
solutions of the equation $L_0(\psi)=0$ in $\R^n$
consist of linear combinations of the functions $Z_1,\ldots, Z_{n+1}$ defined as
\begin{equation}
	\label{Zj}
	Z_i(y)  :=     {\partial U \over \partial y_i} (y), \quad i=1, \ldots , n, \quad
	Z_{n+1}(y)  :=   {n-2 \over 2} U(y) + y\cdot \nabla U(y) ,
\end{equation}
and that problem \equ{pb} is solvable for a function $h(y) = O( |y|^{-m})$, $m>2$,  if and only if
$$
\int_{\R^n} h(y)Z_i(y) \, dy\, =\, 0 \foral i=1,\ldots, n+1 .
$$
Since $n\ge 5$,  the necessary and sufficient condition for the existence of $\phi_{0j}$ solving \equ{pico0} is
 \be\label{egg}
\int_{\R^n}   E_{0j}[\mu,\dot\mu_j] (y,t)\,Z_{n+1}(y) \, dy\, =\, 0 .
\ee
We compute
\be\label{egg1}
\int_{\R^n}  E_{0j}[\mu,\dot\mu_j] (y,t)\,Z_{n+1}(y) \, dy\, =\, c_1\Big  [ \mu_j^{{n-3}} H(q_j,q_j)   -   \sum_{i\ne j}   \mu_j^{\frac {n-4}2}\mu_i^{\frac{n-2}2} G(q_i,q_j) \Big ]
 \, + \,
c_2\,  \dot \mu_j\, ,
\ee
where $c_1$ and $c_2$  are the positive constants given by
\be\label{defc1c2}
c_1 =   - p\int_{\R^n}  U^{p-1} Z_{n+1}   =  {n - 2 \over 2}  \int_{\R^n}  U^p , \quad c_2 = \int_{\R^n} |Z_{n+1}|^2  .
\ee
We observe that $c_2< +\infty$ thanks to the assumed fact  $n\ge 5$.
Relations \equ{egg} for $j=1,\ldots,k$ define a nonlinear system of ODEs for which a solution can be found as follows:
we write $$\mu_j(t) = b_j\mu_0(t)$$ and arrive at the relations

$$
  b_j^{{n-2}} H(q_j,q_j)   -   \sum_{i\ne j}  (b_i b_j)^{\frac{n-2}2} G(q_i,q_j)
 \, + \,
c_2 c_1^{-1}\,  b_j^2 \mu_0^{3-n} \dot \mu_0(t)  = 0 \foral j=1,\ldots, k,
$$
so that $\mu_0^{3-n} \dot \mu_0(t) $ should  equal a  constant, which is necessarily negative since $\mu_0$ decays to zero. This constant can be scaled out, hence it can be chosen arbitrarily to the expense of changing accordingly
the values $b_i$. We impose
\be \label{defdef}
 \dot \mu_0  =  -   \frac{  2c_1 c_2^{-1} }{n-2}\mu_0^{n-3} , \quad
\ee
which yields after a suitable translation of time,
\be\label{mu0}
\mu_0(t)\ = \   \gamma_n t^{-\frac 1{n-4} }, \quad \gamma_n = (2^{-1}(n-4)^{-1}(n-2)c_1^{-1} c_2)^{\frac 1{n-4} }
\ee
and therefore the positive constants $b_j$ (in case they exist) must solve the nonlinear system of equations
\be
  b_j^{{n-3}} H(q_j,q_j)   -   \sum_{i\ne j} b_i^{\frac{n-2}2}b_j^{\frac{n-2}2 -1} G(q_i,q_j)
= \frac { 2b_j }{n-2} \foral j=1,\ldots, k.
\label{ee}\ee
We make the following claim: System \equ{ee} has a solution (which is unique) if and only if the matrix
\be\label{gxi1}
 \mathcal G (q) = \left [ \begin{matrix} H(q_1,q_1) &  -G(q_1,q_2) &\cdots & -G(q_1,q_k) \\   -G(q_1,q_2) &  H(q_2,q_2) &  -G(q_2,q_3) \cdots & -G(q_3,q_k) \\ \vdots &  & \ddots& \vdots   \\  -G(q_1,q_k) &\cdots&  -G(q_{k-1},q_k) &  H(q_k,q_k)            \end{matrix}    \right ]
\ee
is positive definite.

This system \eqref{ee} is variational. Indeed, it is equivalent to $\nn_b I(b) =0$ where
$$
I(b) := {1\over n-2} \left[    \sum_{j=1}^k   b_j^{{n-2}}H(q_j,q_j)   -        \sum_{i\ne j} b_i^{\frac{n-2}2}b_j^{\frac{n-2}2} G(q_i,q_j) - \sum_{j=1}^k   b_j^2 \right]
$$
Writing $\Lambda_j = b_j^{\frac {n-2}2}$ the functional becomes
$$
(n-2) \, I(b) = \ttt I(\Lambda) = \sum_{j=1}^k H(q_j, q_j) \Lambda_j^2 -   \sum_{i \ne j}  G(q_i,q_j) \Lambda_i \Lambda_j  -   \sum_{j=1}^k \Lambda_j^{4 \over n-2} .
$$
Let us assume that the matrix $ \mathcal G (q)$ is positive definite. Then the functional $\ttt I(\Lambda)$ is strictly convex. It clearly has a global minimizer with all components positive.
This yields the existence of a unique critical point $b$ of $I(b)$ with positive components. Reciprocally, let us assume that $\ttt I$ has a critical point $\Lambda$.
Indeed, if $\gamma_1$ is the least eigenvalue of $\mathcal G(q)$, with eigenvector $v = (v_1 , \ldots , v_k )$, then its variational characterization implies that it can be chosen with all its components  non-negative. Since  $\pp_t \ttt I( \Lambda +t v) |_{t=0} =0  $, we obtain
$$  {4\over n-2}  \sum_{i=1}^k\Lambda_i^{{4\over n-2} -1} v_i  =  \gamma_1 \sum_{i=1}^k\ v_i \Lambda_i, $$
and hence $\gamma_1 >0$, so that $\mathcal G(q)$ is positive definite.

Since  the matrix
$
D^2 I(b)
$
is positive definite, we denote by ${2\over n-2} \bar \sigma_j $, $j=1, \ldots , k$, its positive eigenvalues and $\bar w_1 , \ldots , \bar w_k$ the corresponding eigenvectors. Thus, there exists $\bar \sigma >0$ so that
\begin{equation}
\label{D2Ib}
D^2 I (b) = {2\over n-2} \left( P^T {\mbox {diag}} (\bar \sigma_1 , \ldots , \bar \sigma_k )
P \right) , \quad {\mbox {with}} \quad \bar \sigma_j  \geq \bar \sigma,
\end{equation}
where $P$ is the $k\times k$ matrix defined by $P = \left[ \bar w_1 \, | \ldots | \, \bar w_k \right]$. This fact will be useful in the sequel.

\bigskip
In what follows we fix the function $\mu_0(t)$ and the constants $b_j$ as in \equ{mu0} and \equ{ee}. We also write
\be \label{defdef2}
\bar \mu_0 = (\mu_{01},\ldots, \mu_{0k}) = (b_1\mu_0,\ldots,b_k\mu_0)
\ee
where $\mu_0 = \mu_0 (t)$ is defined in \eqref{defdef} and $b_j$ are the positive constants defined in \eqref{ee}.
Let $\Phi_{j}$ be the unique solution of \equ{pico0} for $\mu = \bar\mu_0$.
Thus $\Phi_{j}$ solves

\be\label{aa}
  \Delta \Phi_{j} + pU(y)^{p-1} \Phi_{j} =  - \mu_{0j} E_{0j}[\bar\mu_0, \dot \mu_{0j}],\quad \Phi_{j}(y,t) \to 0 \ass y\to \infty.
\ee
From
the choice of the parameters $\mu_0$, $b_j$  we have
\begin{equation}\label{newerror}
\mu_{0j} E_{0j}[\bar\mu_0, \dot \mu_{0j}]  =  -    \gamma_j \mu_0(t)^{n-2}  \, q_0 (y)
\end{equation}
where $\gamma_j $ is a positive constant  and
$$
q_0( y ):=   p\,U^{p-1} (y)  c_2  +    c_1Z_{n+1}  (y) , \quad
$$ so that
$ \int_{\R^n}  q_0(y) \, Z_{n+1}(y) \, dy \, = \, 0 \  .$

\medskip
Problem \equ{aa} has a radially symmetric solution which we can describe from the variation of parameters formula as follows.
Let $\tilde Z_{n+1} (r) $ so that $L_0(\ttt Z_{n+1}) = 0$ with
$$ \ttt Z_{n+1}(r) \sim r^{2-n} \ass r\to 0, \quad \ttt Z_{n+1}(r) \sim 1 \ass r\to \infty, $$
and the radial solution $p_0= p_0(|y|)$ of $L_0(p_0) = q_0$ described as
$$
p_0 (r)  = cZ_{n+1}(r) \int_0^{r} \ttt Z_{n+1}(s) q_0(s) s^{n-1}\, ds -  c\ttt Z_{n+1}(r) \int_0^{r} Z_{n+1}(s) q_0(s) s^{n-1}\, ds.
$$
Observe that $p_0$ satisfies
\be \label{marshall}
p_0(|y| ) = O( |y|^{-2}  )\ass |y|\to \infty.
\ee
Then a solution $\Phi_{j}(y,t) $ to \equ{aa}  is simply given by the function
$$\Phi_{j}(y,t)  = \gamma_j \mu_0(t)^{n-2} p_0(y). $$

\medskip
\noindent
This leads us to the following corrected approximation,
 \begin{equation}
\label{bb1}
u_{\mu,\xi}^*(x,t)  := u_{\mu,\xi}(x,t)  + \ttt \Phi (x,t) ,
\end{equation}
where
$$
\ttt \Phi (x,t) :=    \sum_{j=1}^k  \mu_j^{-\frac{n-2}2}  \Phi_{j}\left ( \frac {x-\xi_j} {\mu_j}, t\right ) 
$$
The space decay of the functions $\Phi_{j}$   makes the size of $\ttt \Phi$  $\mu_0^2$-times smaller than that of $u_{\mu , \xi} $ when measured far away from the $q_j$'s. Therefore its addition to $u_{\mu,\xi}$ does not modify the size of the error there. A direct consequence of \eqref{form3} and \eqref{marshall} is that, in the region $|x-q_i | > \delta$ for any $i=1, \ldots , k$, we can describe the error function $S (u_{\mu , \xi }^* )$ as follows
\be \label{form4}
S (u_{\mu , \xi }^* ) (x,t) = \mu_0^{n-4 \over 2} \sum_{j=1}^k \dot \lambda_j f_j +
\mu_0^{n-2 \over 2} \sum_{j=1}^k \dot \xi_j \cdot \vec f_j + \mu_0^{n+2 \over 2} f
\ee
where $f_j$, $\vec f_j$, $f$ are smooth and bounded functions of $( x, \mu_0^{-1} \mu, \xi )$.

Instead, near each of the points $q_j$, the leading term   $\Phi_{j}(y,t)$ of $\ttt \Phi$ eliminates the term of size comparable to $\mu_0^{n-2}$ at the expense of creating new smaller order terms.
Using the notation $y_{j} = {x-\xi_j \over \mu_{j}}$, the new error of approximation $S(u_{\mu,\xi}^*)$ is
\begin{align}\label{nonsaprei}
S(u_{\mu,\xi}^*) &=  S( u_{\mu,\xi}) -  \sum_{j=1}^k \mu_j^{- {n+2 \over 2} } \mu_{0j} E_{0j}[\bar \mu_0 , \dot \mu_{0j}]  \\
&+\sum_{j=1}^k \mu_j^{-{n+2 \over 2}} \left\{ \mu_j^2 \partial_t  \Phi_{j} (y_j , t) -
\mu_j \dot \mu_j [ {n-2 \over 2} \Phi_{j} (y_j , t ) +\nabla \Phi_{j} (y_j , t ) \cdot y_j ]
+ \nabla \Phi_{j} (y_j , t ) \cdot \mu_j \dot \xi_j \right\} \nonumber \\
&+ (u_{\mu , \xi} + \ttt \Phi )^p - u_{\mu , \xi }^p - p \sum_{j=1}^k  \mu_j^{-{n+2 \over 2}} U(y_j )^{p-1} \Phi_{j} (y_j , t) .\nonumber
\end{align}
\medskip

Using formula \equ{SS},
we find that for a given fixed $j$ and $|x-q_j| \le  \delta  $,
 the error $S(u_{\mu,\xi}^*)$ has the form
\be\label{SSnew}
 \mu_j^{\frac{n+2}2} S( u_{\mu,\xi}^*  ) \ = \      \mu_j^{\frac{n+2}2} S( u_{\mu,\xi}) - \mu_{0j} E_{0j}[\bar \mu_0 , \dot \mu_{0j}]   + A_j (y) ,
\ee
and,
after some computation we estimate for $f$ and $g$ smooth and bounded,
\be\label{SS1}
 A_j =    {\mu_0^{n+4}}  f( \mu_0^{-1}\mu, \xi, \mu_j y)  +  \frac{\mu_0^{2n-4}}{1+ |y_j|^2} g( \mu_0^{-1}\mu, \xi,  \mu_j y), \quad y_j = {x-\xi_j \over \mu_j}.
\ee
In what follows
we set
\begin{equation}\label{gigi}
\mu(t) = \bar \mu_0(t) + \lambda(t), \quad {\mbox {with}} \quad       \lambda(t) =(\la_1(t), \ldots, \la_k(t)),
\end{equation}
and $\bar \mu_0$ defined in \eqref{defdef2}.
 In the notation of Lemma \ref{errore},  we then get from \equ{SS} and \equ{SS1},
\begin{align*}
S(u_{\xi, \mu}^*) =   \mu_j^{-\frac{n+2}2}     & \big \{  \mu_{0j} ( E_{0j}[\mu,\dot \mu_j] - E_{0j}[\bar\mu_0,\dot \mu_{0j}] ) +
 \la_j  E_{0j}[\mu,\dot \mu_j] \\
&
+  \mu_j E_{1j}[\mu,\dot \xi_j]  \ + \ \mathcal R_j \, + A_j \big \} .
\end{align*}
Let us estimate
$
E_{0j}[\mu,\dot \mu_j] - E_{0j}[\bar\mu_0,\dot \mu_{0j}]$.
Recall that
$$
 E_{0j} [\mu, \dot \mu_j] \ = \  \dot \mu_j\, Z_{n+1}(y) \ +\   pU(y)^{p-1}  \Big [ -\mu_j^{{n-3}} H(q_j,q_j)   +   \sum_{i\ne j}  \mu_j ^{\frac{n-2}2-1} \mu_i^{\frac{n-2}2} G(q_i,q_j) \Big ].
$$
We find
\begin{align*}
&E_{0j}(\bar \mu_0 + \la, b_j\dot \mu_0 + \dot \la_j  )  -    E_{0j}(\bar \mu_0, b_j\dot \mu_0 )\  \\
&=  \dot \la _j\, Z_{n+1}(y)
-  \mu_0^{n-4}  pU(y)^{p-1}\Big [ \,   \sum_{i=1}^k M_{ij}\la_i +   \sum_{i,l=1}^k f_{ijl}( \mu_0^{-1}\la)   \la_i \la_l\, \Big ] .
\end{align*}
for smooth functions $f_{ijl}$, where the $M_{ij}$ are the positive constants
$$
M_{jj} =    (n-3) b_j^{{n-4}} H(q_j,q_j)   -  \frac{n-4}2 \sum_{i\ne j}  b_j ^{\frac{n-6}2} b_i^{\frac{n-2}2} G(q_i,q_j)
$$
and for $i \ne j$
$$
M_{ij} =  - \frac {n-2}2 b_j ^{\frac{n-4}2} b_i^{\frac{n-4}2} G(q_i,q_j).
$$
We claim that  the $k\times k$ symmetric matrix $M= M_{ij}$ is positive definite.
In fact,
by definition $M = D^2 I_0 (b )$, where
$$
I_0 (b) :={1\over n-2} \left[    \sum_{j=1}^k   b_j^{{n-2}}H(q_j,q_j)   -        \sum_{i\ne j} b_i^{\frac{n-2}2}b_j^{\frac{n-2}2} G(q_i,q_j)  \right].
$$
From the definition
of the numbers $b_j$ in \eqref{ee}  we have
\begin{equation}\label{gajardo6}
M= D^2 I_0 (b) = D^2 I(b ) + {2\over n-2} I_k = {2\over n-2} \,  \left( P^T
\, {\mbox {diag}} (1+ \bar \sigma_1 , \ldots , 1+ \bar \sigma_k ) \, P \right)
\end{equation}
where we use the notation introduced in \eqref{D2Ib}. Therefore, $M$ is positive definite and there exists $\bar \sigma >0$ so that its eigenvalues
\begin{equation}\label{defsigmabar}
{\mbox {eigenvalues of }} \, M  \geq {2\over n-2} \, (1+ \bar \sigma ),
\quad {\mbox {for all}} \quad j=1, \ldots , k.
\end{equation}

\medskip
Let us now estimate $ \la_j \, E_{0j}[\mu,\dot \mu_j] $.
We have
\begin{align*}
\la_j  E_{0j}[\mu,\dot \mu_j] &=
\lambda_j  \, b_j \left[  \dot \mu_0 Z_{n+1} + p U^{p-1} \mu_0^{n-3} \left( - b_j^{n-2} H(q_j , q_j ) + \sum_{i\not= j} (b_i b_j)^{n-2 \over 2} G(q_i , q_j ) \right) \right] \\
+&\la_j  \dot \la _j\, Z_{n+1}(y)
-  \mu_0^{n-4}  pU(y)^{p-1}   \sum_{i,l=1}^k f_{ijl}( \mu_0^{-1}\la , q)   \la_i \la_l\,
\end{align*}
for smooth functions $f_{ijl}$.
\medskip
\subsection*{Total expansion of the error }
We have the validity of the following expansion for the error of approximation $S (u_{\mu , \xi }^* )$.

\begin{lemma}\label{lemaerror} Let us fix an index $j$ and consider $x$ with $|x-q_j| \le \frac 12 \min_{i\ne l} |q_i-q_l|$. Let $\mu = \bar\mu_0 + \la$ defined as in \eqref{gigi} and
assume that $|\la(t)| \le \mu_0(t)^{1+\sigma} $ for some $0<\sigma<\bar \sigma$, with $\bar \sigma$ defined in \eqref{defsigmabar}.
Then setting $x= \xi_j + \mu_j y_j $, we have that for all $t$ large, the error of approximation
 $S(u^*_{\mu , \xi})$ can be estimated as
\begin{align*}
S( u^*_{\mu, \xi} )\ & = \
    \sum_{j=1}^k \mu_j^{-\frac{n+2}2 } \Big \{ \,  \mu_{0j}\dot \la _j\, Z_{n+1}(y_j) \, - \,   \mu_{0j}\mu_0^{n-4}  pU(y_j)^{p-1} \,   \sum_{i=1}^k M_{ij}\la_i\, \, \\
&+
 \mu_j \dot \xi_j \cdot \nn U(y_j)  +  pU(y_j)^{p-1} \big [-\mu_j^{n-2} \nn_x H(q_j ,q_j)   \,  + \,  \sum_{i\ne j}  \mu_j^{\frac {n-2}2} \mu_i^{\frac {n-2}2}  \nn_x G( q_j,q_i )\, \big ] \cdot  y_j    \Big \} \,
\\
&+ \sum_{j=1}^k \mu_j^{-{n+2 \over 2}}
\lambda_j  \, b_j \left[  \dot \mu_0 Z_{n+1} + p U^{p-1} \mu_0^{n-3} \left( - b_j^{n-2} H(q_j , q_j ) + \sum_{i\not= j} (b_i b_j)^{n-2 \over 2} G(q_i , q_j ) \right) \right]\\
&+
\mu_0^{-\frac{n+2}2 } \Big [ \,   \sum_{j=1}^k   \frac{\mu_0^{n}g_j}{1+ |y_j|^2} +   \sum_{j=1}^k  \frac{\mu_0^{n-2}\vec g_j}{1+ |y_j|^4}\cdot (\xi_j-q_j)
 + \sum_{j=1}^k   \frac{\mu_0^{n-2} g_j}{1+ |y_j|^4}\lambda_j  \\
&
+       \mu_0^{n -4}   \, \sum_{i,j,l=1}^kU(y_j )^{p-1} f_{ijl}  \la_i \la_l\,
+         \, \sum_{i,j,l=1}^k {  f_{ijl}\over 1+ |y_j|^{n-2}}   \la_i \dot \la_l\,     \\ &+
\mu_0^{n+2}  f +  \mu_0^{n-1 } \sum_{i=1}^k   \dot \la_i  f_{i}   +  \mu_0^{n} \sum_{i=1}^k   \dot \xi_i  \cdot \vec f_{i} \Big ],
\end{align*}
where $\vec f_{i},\ f_{i},\ f,\  f_{ijl}$ are smooth, bounded functions of the argument $(\mu_0^{-1}\mu, \xi, x)$, and  $g_j, \vec g_j$ of $(\mu_0^{-1}\mu, \xi, y_j)$.
\end{lemma}

\setcounter{equation}{0}
\section{The inner-outer gluing procedure}\label{sec3new}

Let $t_0 >0$,  and consider the Problem
\be \label{equ1}
\pp_t u = \Delta u + u^p \inn \Omega \times [t_0,\infty), \qquad u = 0 \quad \quad\onn \pp\Omega \times [t_0,\infty).
\ee
Our purpose is to construct a global unbounded solution to \equ{equ1}, provided that $t_0 $ is sufficiently large. This produces a solution $u(x,t) = u(x, t-t_0) $ to Problem \equ{P}.

We solve Problem \eqref{equ1}
with $u$ of the form
\be\label{solll}
u= u_{\mu,\xi}^* + \ttt \phi,
\ee
where $\ttt \phi (x,t )$ is  a smaller term.

We construct the function $\tilde \phi$ by means of what we call
the {\em inner-outer gluing} procedure, which has been applied in other contexts, but all related to elliptic problems. This time the procedure is used to find solutions of a parabolic problem.

This procedure consists in writing
\begin{equation}\label{deftildephi}
\ttt \phi(x,t) =  \psi (x,t) +    \phi^{in} (x,t)
\quad
{\mbox {where}} \quad
 \phi^{in}(x,t) : =  \sum_{j=1}^k \eta_{j,R} (x,t) \ttt \phi_j (x,t)
 \end{equation}
with
\begin{equation}\label{martin1}
\ttt \phi_j (x,t) := \mu_{0j}^{-\frac{n-2} 2} \phi_j\left (\frac {x- \xi_j} {\mu_{0j}} , t     \right ), \quad \mu_{0j} (t) = b_j \mu_0 (t)
\end{equation}
and
\begin{equation}\label{we4}
\eta_{j,R} (x,t) = \eta \left (\frac {x- \xi_j} {R \mu_{0j}}    \right ) .
\end{equation}
Here $\eta(s)$ is a smooth cut-off function with $\eta(s) =1$ for $s<1$ and $=0$ for $s>2$. The number $R$ is a sufficiently large number, independent of $t$,  that for convenience we take as
\begin{equation}\label{defR}
R= t_0^{\rho} , \quad {\mbox {with}} \quad  0< \rho <{n-2 \over 2(n-4)}.
\end{equation}
In terms of $\ttt \phi$, Problem \equ{equ1} reads as
\be\label{equ2prima}
\pp_t \ttt \phi =  \Delta \ttt \phi + p(u_{\mu,\xi}^*)^{p-1}\ttt \phi + \ttt N(\ttt \phi ) + S(u_{\mu,\xi}^*) \inn \Omega \times [t_0,\infty) , \quad \ttt \phi = - u_{\mu,\xi}^*   \onn \pp\Omega \times [t_0,\infty),
\ee
where
$$
\ttt N_{\mu,\xi}(\ttt \phi ) = (u_{\mu,\xi}^* + \ttt \phi)^p - (u_{\mu,\xi}^*)^p - p( u_{\mu,\xi}^*)^{p-1}\ttt\phi, \quad  S(u_{\mu,\xi}^*) = -\pp_t u_{\mu,\xi}^* + \Delta u_{\mu,\xi}^* + (u_{\mu,\xi}^* ) ^p.
$$
In the notation of Lemma \ref{lemaerror} we decompose
\begin{equation} \label{we1}
S( u^*_{\mu, \xi} )\ = \   \sum_{j=1}^k S_{\mu, \xi, j}   +   S_{\mu, \xi}^{(2)}
\end{equation}
where, for $y_j = {x-\xi_j \over \mu_j}$,
\begin{align} \label{we2}
 S_{\mu, \xi,j}  &=    \mu_j^{-\frac{n+2}2 } \Big \{ \,  \mu_{0j} \left[ \dot \la _j\, Z_{n+1}(y_j) \, - \,   \mu_0^{n-4}  pU(y_j)^{p-1} \,   \sum_{i=1}^k M_{ij}\la_i\, \right]  \\
&+\, \lambda_j  \, b_j \left[  \dot \mu_0 Z_{n+1}  (y_j )+ p U^{p-1} (y_j)  \mu_0^{n-3} \left( - b_j^{n-2} H(q_j , q_j ) + \sum_{i\not= j} (b_i b_j)^{n-2 \over 2} G(q_i , q_j ) \right) \right] \nonumber \\
&+
 \mu_j \left[ \dot \xi_j \cdot \nn U(y_j)  +  pU(y_j)^{p-1} \big [-\mu_j^{n-2} \nn_x H(q_j ,q_j)   \,  + \,  \sum_{i\ne j}  \mu_j^{\frac {n-2}2} \mu_i^{\frac {n-2}2}  \nn_x G( q_j,q_i )\, \big ] \cdot  y_j    \right] \Big \}. \, \nonumber
\end{align}
\medskip
Let
\begin{equation}
\label{defVmonica}
V_{\mu,\xi} =     p  \sum_{j=1}^k ( (u_{\mu,\xi}^*)^{p-1}  - (\mu_j^{-\frac{n-2}2} U({x-\xi_j \over \mu_j}))^{p-1}) \eta_{j,R}  +  p\big (1 - \sum_{j=1}^k  \eta_{j,R} \big ) (u_{\mu,\xi}^*)^{p-1}    .
\end{equation}
A main observation we make is that $\ttt \phi $ solves Problem \equ{equ2prima} if the tuple $(\psi, \phi)$ where $\phi= (\phi_1,\ldots, \phi_k)$ solves the following system:
\begin{align} \label{equpsi}
\pp_t \psi &=  \Delta \psi +  V_{\mu,\xi} \psi  +
\sum_{j=1}^k  [2\nn\eta_{j,R} \nn_x \ttt \phi_j +  \ttt \phi_j(\Delta_x-\pp_t)\eta_{j,R} ]+   \ttt N_{\mu,\xi}( \ttt\phi  ) +  S_{\mu,\xi}^{o}   \inn \Omega \times [t_0,\infty),\nonumber\\
\psi &= -u_{\mu,\xi}^*  \onn  \pp\Omega \times [t_0,\infty),
\end{align}
where
\begin{equation}\label{we3}  S_{\mu,\xi}^{o}  = S_{\mu,\xi}^{(2)}   +  \sum_{j=1}^k (1- \eta_{j,R} ) S_{\mu,\xi,j},
\end{equation}
and for all $j=1,\ldots, k$,
\be\label{equ2}
\pp_t \ttt \phi_j =  \Delta \ttt \phi_j + pU_j^{p-1}\ttt \phi_j + pU_j^{p-1}\psi +   S_{\mu,\xi,j} \inn  B_{2R\mu_{0j}} (\xi_j)  \times [t_0,\infty).
\ee
Let us rewrite \equ{equ2} in terms of $\phi_j(y,t)$, $y\in B_{2R}(0)$, see \eqref{martin1}. The equations become, for $j=1,\ldots,k$
\begin{align} \label{equ3}
\mu_{0j}^2 \pp_t  \phi_j  =&  \Delta_y \phi_j +  pU(y)^{p-1} \phi_j
+
 \mu_{0j}^{\frac{n+2}2}  S_{\mu,\xi,j}  (\xi_j +\mu_{0j} y,t)
\\
&+  p \mu_{0j}^{\frac{n-2}2}\,  {\mu_{0j}^2 \over \mu_{j}^2 } \, U^{p-1} ( {\mu_{0j} \over \mu_j}  y) \psi(\xi_j +\mu_{0j} y,t) +  B_j[\phi_j] + B_j^0 [\phi_j]    \nonumber
\end{align}
where
\begin{equation}\label{gajardo1}
 B_j[\phi_j]:= \mu_{0j} \dot \mu_{0j} (\frac {n-2}2 \phi_j  + y\cdot \nn_y \phi_j)  + \mu_{0j} \nabla \phi_j \cdot \dot \xi_j
 \end{equation}
and
\begin{equation}\label{gajardo2}
B_j^0 [\phi_j ] := p \left[ U^{p-1} \left( {\mu_{0j} \over \mu_j} y \right) -
U^{p-1} (y) \right] \, \phi_j +p \, [ {\mu_{0j}^2 \over \mu_{j}^2 } -1 ] \, U^{p-1} \left( {\mu_{0j} \over \mu_j}  y \right)
\phi_j
\end{equation}
We call \equ{equpsi} the {\em outer problem} and  \equ{equ3} the {\em inner problem(s) }.

\medskip
We  proceed as follows. For given parameters $\la, \xi , \dot \la , \dot \xi $ and functions $\phi_j$ fixed  in a suitable range,  we solve for $\psi$ Problem \equ{equpsi}. Indeed, we can solve \eqref{equpsi} for any small and smooth initial condition $\psi_0 (x)$, in the form of a (nonlocal) operator $\psi =  \Psi(\la,\xi, \dot \la , \dot \xi, \phi)$. This is done in full details in Section \ref{outer}.

We then replace this $\psi$ in equations \equ{equ3}.
Let us explain formally how we solve \eqref{equ3}.

Recall  that the ellitpic linear operator $  L_0 (\phi ) := \Delta \phi + p U^{p-1} (y) \phi$
has an $n+1$ dimensional kernel generated by the bounded functions
$$
Z_i (y) = {\partial U \over \partial y_i}, \quad i=1, \ldots , n, \quad
Z_{n+1} (y ) = {n-2 \over 2} U (y) + \nabla U(y) \cdot y.
$$
If we consider the model problem for \equ{equ3}, in which now we do not neglect the term corresponding to time derivative, and we consider it on the whole $\R^n$
\be \label{modprob}
\mu_{0j}^2 \pp_t \phi= L_0(\phi) + E(y,t),
\ee
we observe that  $\mu_{0j}^2 \pp_t \phi= L_0(\phi) $ when $\phi$ is any linear combination of the functions
$Z_i (y)$, $i=1, \ldots , n , n+1$. This fact suggests that solvability of \eqref{modprob}
depends on whether the right hand side $E(y,t)$ does have component in the directions {\it spanned}  by the $Z_i (y)$'s. In other words,  one expects solvability for \eqref{modprob} provided that
some orthogonality conditions like
$$
\int E(y,t) Z_i (y) \, dy = 0 , \quad i=1, \ldots , n+1, \quad \forall t > t_0
$$
are fullfilled.  Since we have $k$ of these conditions, for any $j=1, \ldots , k$ in \eqref{equ3}, they take the form of a nonlinear, nonlocal system of $(n+1) k$ ODEs in the $(n+1) k$ parameter functions $\la_1 , \ldots , \la_k $ and $\xi_1 , \ldots , \xi_k$. It is at this point that
we choose the parameters $\la$ and $\xi$ (as functions of the given $\phi$) in such a way that these orthogonality (or solvability)  conditions are satisfied.  This is done in Section \ref{sec5}.

Another well known fact about the ellitpic linear operator $  L_0 (\phi ) $ is that
 $L_0$  has a positive radially symmetric bounded eigenfunction  $Z_0$
associated to the only negative eigenvalue $\la_0$ to the problem
\begin{equation}
\label{eigen0}
L_0 (\phi ) + \lambda \phi = 0 , \quad  \phi \in L^\infty(\R^n).
\end{equation}
Furthermore,  $\la_0 $ is simple and  $Z_0$
 decays like
$$Z_0 (y) \sim  |y|^{-\frac{n-1}2} e^{-\sqrt{|\la_0 |}\,  |y|}  \ass |y| \to \infty. $$
Let $e(t) := \int_{\R^n} \phi(y,t)Z_0 (y)\, dy $, the projection of $\phi(y,t)$ in the direction $Z_0(y)$.  Then, integrating equation \eqref{modprob} in $\R^n$, using that  $\mu_{0j}(t)^2 =b_j^2  t^{-\frac 2{n-4}} $ we get
\be\label{kkk}
b_j t^{-\frac 2{n-2}} \dot e(t)  - \la_0 e(t) = f(t):=  (\int_{\R^n} Z_0 (y)^2 \, dy )^{-1}\,  \int_{\R^n} E(y,t)Z_0(y)\, dy .
\ee
Hence, for some $a>0$,
$$
e(t)  =    \exp( at^{\frac {n-2} {n-4}} ) \, \Big (   e(t_0) +   \int_{t_0}^t  s^{\frac 2{n-4}}f(s) \exp(- as^{\frac {n-2} {n-4}} ) ds  \Big ).
$$
The only way in which $e(t)$ does not grow exponentially in time (and hence $\phi(y,t)$ does not growth exponentially in time) is for the specific value of initial condition
$$e(t_0)=   \int_{\R^n} \phi(y,t_0)Z_0(y)\, dy =    -\int_{t_0}^\infty  s^{\frac 2{n-4}}f(s) \exp(- as^{\frac {n-2} {n-4}} ) ds.$$
This formal argument suggests that the (small) initial condition required for  $\phi$ should lie on a certain manifold locally described as a translation of the hyperplane orthogonal to $Z_0(y)$. Since we have $k$ of these hyperplanes, for any $j=1, \ldots , k $ in \eqref{equ3}, these constraints define a {\em codimension $k$ manifold}  of initial conditions which describes those for which the expected asymptotic bubbling behavior is possible. We discuss this part in Sections \ref{seclineartheory} and \ref{final}.

\medskip
In summary, the inner-outer gluing procedure allows us to show that: for any small and smooth initial condition $\psi_0$ for Problem \eqref{equpsi}, we find a solution $\psi$  to \eqref{equpsi}, $\la $, $\xi$ and $\phi$ solutions to \eqref{equ3}, with initial condition $\phi (y, t_0)$ belonging to a $k$-codimensional space, so that
$u_{\mu , \xi} + \tilde \phi$ defined in \eqref{solll}, \eqref{martin1}, \eqref{we4} is a solution to \eqref{equ1} with the expected asymptotic bubbling behavior.

\medskip
The rest of the paper is devoted to prove rigorously what we have described so far.

\medskip\setcounter{equation}{0}
\section{Solving the outer problem }\label{outer}
Our aim is to solve first the {\em outer problem} \equ{equpsi} for a given small function $\phi$ and the considered range of parameters $\la, \xi , \dot \la , \dot \xi $, in the form of a nonlinear operator
$$
\psi(x,t)  = \Psi ( \la, \xi , \dot \la , \dot \xi , \phi )\, (x,t).
$$
Thus we consider the initial-boundary value problem
\begin{align}\label{equpsi1}
\pp_t \psi =&  \Delta \psi +  V_{\mu,\xi} \psi \ +
\sum_{j=1}^k  [2\nn\eta_{j,R} \nn_x \ttt \phi_j +  \ttt \phi_j(\Delta_x-\pp_t)\eta_{j,R} ]\
\nonumber \\
& +   \ttt N_{\mu,\xi}\Big(  \psi + \phi^{in}
  \Big) +  S_{\mu,\xi}^{o}   \inn \Omega \times [t_0,\infty), \\
\psi =& -u_{\mu,\xi}^*  \onn  \pp\Omega \times [t_0,\infty), \quad \psi(t_0,\cdot) = \psi_0 \inn \Omega, \nonumber
\end{align}
for an initial condition $\psi_0$ which is a small and smooth function  whose size will be chosen sufficiently small.

Fix $\sigma$ with
\be\label{defsigma}
0<\sigma < \bar \sigma, \quad {\mbox {where}} \quad \bar \sigma \leq \bar \sigma_j , \quad j=1, \ldots , k
\ee
where $\bar \sigma$, $\bar \sigma_j$ are defined in \eqref{D2Ib} and \eqref{defsigmabar}.
Given $h (t) : [t_0 , \infty ) \to \R^k $, and $\delta >0$, we introduce the weighted $L^\infty$ norm
\be\label{normfinal}
\| h \|_{\delta} := \| \mu_0 (t)^{- \delta} \, h (t)  \|_{L^\infty (t_0 , \infty)} .
\ee
In what follows we assume that the parameters $\la $, $\xi $, $\dot \la $, $\dot \xi$ satisfy the constraints
\be \label{rest}
\|\dot \la(t) \|_{n-3+\sigma}  + \|\dot \xi(t) \|_{n-3 +\sigma}  \le c ,
\ee
and
\be \label{rest1}
\| \la(t) \|_{1+\sigma}   +  \|\xi(t) -q \|_{1+\sigma}  \le  c ,
\ee
for some constant $c$.

We recall that $\phi^{in} (x,t) = \sum_{j=1}^k \eta_{j, R} (x,t) \ttt\phi_j (x,t) $, with
$
\ttt\phi_j(x,t)
$ defined in \eqref{martin1}. Let
\be \label{starstar}
\| \phi \|_{n-2+ \sigma , a} = \max_{j=1, \ldots , k} \|\phi_j \|_{n-2+\sigma,a}.
\ee
where  $\|\phi_j \|_{n-2+ \sigma,a}$ is the least number  $M$ for which
$$
  (1+ |y|) | \nn \phi_j(y,t) | + | \phi_j (y,t) |   \ \le \  M   \frac {\mu_0(t)^{n-2+\sigma}}{ 1+ |y|^a}  , \quad j=1,\ldots,k ,
$$
for a given $0<a<1$, and $\sigma$ fixed in \eqref{defsigma}.
We assume that $\phi=(\phi_1,\ldots , \phi_k)$ satisfies the constraints
\be \label{rest13}
\| \phi \|_{ n-2+ \sigma , a} \leq c\, t_0^{-\ve}
\ee
for some $\ve >0$, small.

\medskip
Next Proposition states the existence of a solution $\psi$ to \eqref{equpsi1}, for any initial condition $\psi_0$ small. Furthermore, it clarifies the dependence of
 $\Psi (\la , \xi , \dot \la, \dot \xi , \phi )$ in the parameters involved. This is done by  estimating, for instance,
$$\pp_\phi \Psi  (\la , \xi , \dot \la , \dot \xi , \phi) [\bar \phi] = \pp_s \Psi [\phi + s\bar \phi,\la,\xi ]\Big|_{s=0} $$
as linear operator between Banach spaces.
 For notational convenience we denote the above operator by $\partial_\phi \Psi [\bar \phi]$. Similarly we denote by  $\partial_\la \Psi [\bar \la ] $, $\partial_\xi \Psi [\bar \xi ] $, $\partial_{\dot \la} \Psi [\dot {\bar \la} ] $ , $\partial_{\dot \xi } \Psi [\dot {\bar \xi} ] $ the corresponding linear operators.

We have the validity of the following

\begin{prop}\label{probesterno}
Assume that the parameters $ \lambda$, $\xi$, $\dot \la $, $\dot \xi$ satisfy \eqref{rest}-\eqref{rest1},  and the
function $\phi=(\phi_1,\ldots , \phi_k)$ satisfies the constraint
\eqref{rest13}.  Assume furthermore that $\psi_0 \in C^2 (\bar \Omega )$ and
\be \label{psi0rest}
\| \psi_0 \|_{L^\infty (\bar \Omega) } + \| \nabla \psi_0 \|_{L^\infty (\bar \Omega) }
\leq  t_0^{-\ve}.
\ee
 Assume that the radius $R $ is given in \eqref{defR}.
Then there exists  $t_0$ large so that Problem \eqref{equpsi1}
has a unique solution $\psi= \Psi(\la , \xi , \dot \la , \dot \xi , \phi)$, and, given $\alpha $ with $0< \alpha \leq a$, there exist $\sigma $ satisfying \eqref{defsigma}, $\rho$ in \eqref{defR}  and  $\ve >0$ small so that, for $y_j = {x-\xi_j \over \mu_{0j} }$,
\be \label{leuco1}
|\psi(x,t)| \ \lesssim \   t_0^{-\ve}  \sum_{j=1}^k \frac{\mu_0^{\frac{n-2}2 +\sigma} (t) }{|y_j|^\alpha + 1 }
+  e^{-\delta (t-t_0) } \|\psi_0 \|_{L^\infty (\Omega ) },
\ee
and
\be\label{leuco11}
|\nabla_x \psi(x,t)| \ \lesssim \    t_0^{-\ve}  \,  \sum_{j=1}^k \frac{\mu_0^{\frac{n-2}2  +\sigma} (t)  \mu_j^{-1} }{|y_j|^{\alpha +1} + 1 }, \quad {\mbox {for}} \quad |y_j | <R.
\ee
Here, and in what follows, we use the symbol $\lesssim$ to say $\leq C$, for a positive constant $C$, whose value may change from line to line, and which is independent of $t$ and $t_0$.

\end{prop}

\medskip
\noindent
To prove this result,  we shall estimate, for  given functions $f(x,t)$, $g(x,t)$, $h(x)$ the unique solution  of the linear problem
\be\label{hh1}
\pp_t \psi =  \Delta \psi + V_{\mu , \xi} \psi \ + f(x,t)  \inn \Omega \times [t_0,\infty),
\ee
$$
\qquad \psi = g   \onn  \pp\Omega \times [t_0,\infty), \quad \psi(\cdot, t_0) = h,
$$
where the function $V_{\mu , \xi} $ is defined in \eqref{defVmonica}.
We assume that for $\alpha, \beta>0$ we have
\be
|f(x,t) | \ \leq \     M  \sum_{j=1}^k   \frac {\mu_j^{-2} t ^{-\beta}} { 1+ |y_j|^{2+\alpha} }   , \quad y_j =\frac{|x-\xi_j|}{\mu_j}
\label{gg}\ee
and denote by $\|f\|_{*,\beta,2+\alpha}$ the least $M$ for which \equ{gg} holds.

\medskip
\noindent
We have the following result.

\begin{lemma}\label{psistar1}
Assume that $\|f\|_{*,\beta,2+\alpha}<+\infty$ for some $\beta,\alpha>0$, $0<\alpha <1$. Assume also that $\| h \|_{L^\infty (\Omega ) } < +\infty$ and
$\| s^\beta g(x,s) \|_{L^\infty (\partial \Omega \times (t_0 , \infty) )} <+\infty$.
Let $\phi = \psi[f,g,h] $ be the unique solution of Problem $\equ{hh1}$.
There exists  $\delta = \delta(\Omega) >0$ small so that,  for all $(x,t)$,
\be \label{potoo0}
  |\psi(x,t)| \ \lesssim\    \|f\|_{*,\beta,2+\alpha } \left( \sum_{j=1}^k \frac { t^{-\beta}  } { 1+ \left |y_j\right |^{\alpha}   } \right)  +  e^{-\delta (t-t_0) } \|h\|_{L^\infty (\Omega ) }  +    t^{-\beta} \|s^\beta g\|_{L^\infty (\pp \Omega \times (t_0,\infty )) } ,
\ee
where $y_j = {x-\xi_j \over \mu_j} $. Moreover, we have the following local estimate on the gradient
\be \label{potoo1}
  |\nn_x\psi(x,t)| \ \lesssim\   \|f\|_{*,\beta,2+\alpha}  \sum_{j=1}^k \frac {\mu_j^{-1} t^{-\beta}  } { 1+ \left |y_j\right |^{\alpha+1}   } , \quad {\mbox {for}} \quad |y_j |\leq R.
\ee
\end{lemma}

\proof
We consider first the solution $\psi_0[g,h]$  to the heat equation
\be\label{hh11}
\pp_t \psi_0 =  \Delta \psi_0   \inn \Omega \times [t_0,\infty), \quad
\psi_0 = g   \onn  \pp\Omega \times [t_0,\infty), \quad \psi_0(\cdot, t_0) = h.
\ee
We let $ v(x)$ be the bounded solution of $\Delta v + 1 =0$ in $\Omega$  with $v =1$ on $\Omega$. Then $v\ge 1$ in $\Omega$ and the function
$$
 \bar \psi (x,t) =    \left (  e^{-\delta (t-t_0) } \|h\|_\infty  +    t^{-\beta} \|s^\beta g\|_{L^\infty (\pp \Omega \times (t_0,\infty )) } \right ) v(x)
$$
is a supersolution of \equ{hh11} provided that $\delta = \delta(\Omega) >0$ is fixed sufficiently small.  Then we have $|\psi_0| \le \bar \psi$.
The proof of \eqref{potoo0} and \eqref{potoo1} is thus reduced to the the case $g=0$, $h=0$.

\medskip
Let $q(|z|) = \frac 1{1+ |z|^{2+\alpha}} $ and let $p(|z|)$ be the radial positive solution of
$$
\Delta p   + 4q = 0\inn \R^n
, \quad
{\mbox {given by}} \quad
p(r) =   4\int_r^\infty \frac {d\rho}{ \rho^{n-1}} \int_0^\rho   q(s)\rho^{n-1} d\rho.
$$
Then
$
p(z) \sim  \frac 1{1+ |z|^{\alpha } }\inn \R^n.
$
If $\delta$ is sufficiently small we  have
$$
\Delta p   +  \frac {\delta} { 1 + |z|^2}p +   2q  \le  0\inn \R^n.
$$
It follows that
$
\bar p  (x) := \sum_{j=1}^k p\left (\frac{x-\xi_j}{\mu_j}\right )
$
satisfies, for a possibly smaller $\delta$,
$$
\Delta_x \bar p   +  \left( \sum_{j=1}^k  \,  \mu_j^{-2 }  \,  \frac {\delta} { 1 +  \left |\frac{x-\xi_j}{\mu_j}\right |^2} \  \right)  \bar p \, + \  \frac 32\, \bar q \ \le\  0\inn \R^n, \quad
{\mbox {where}} \quad
\bar q := \sum_{j=1}^k  \frac 1{\mu_j^2 }  q\left (\frac{x-\xi_j}{\mu_j}\right ).
$$
Observe now that
$$
|V_{\mu,\xi}|  \lesssim \sum_{j=1}^k  \mu_j^{-2 }  \,  \frac { R^{-2}}{ 1 +  \left |y_j\right |^2},
$$
as a direct consequence of the definition of the function $V_{\mu , \xi}$ given in \eqref{defVmonica}.
From the above estimates, we obtain that for a given number $\beta >0$ we have that
the function
$\bar \psi (x, t) =  2t^\beta \bar p $
is a positive supersolution of
$$
\pp_t \bar \psi \ge \Delta \bar \psi + V_{\mu , \xi}  \bar
\psi +  {t^\beta} \bar q
$$
for $t>t_0$, where $t_0$ is fixed large enough.
By parabolic comparison we then get
$$
|\psi(x,t)| \ \lesssim\   t^{-\beta}\|f\|_{*,\beta,2+\alpha } \sum_{j=1}^k \frac { 1  } { 1+ \left |y_j\right |^{ \alpha  } }.
$$
Hence,  \equ{potoo0} has been established.
To get the gradient estimate in \equ{potoo1} we scale around each point $\xi_j$
letting
$$
 \psi(x, t ) :=   \ttt \psi\left ( \frac {x- \xi_j}{\mu_j} , \tau(t) \right )
$$
where
$\dot \tau (t) =  \mu_j(t)^{-2}$,
namely
$
\tau(t) \sim     t^{ \frac {n-2}{n-4}} ,
$ as $t \to \infty$.
We choose and fix $t_0$ so that $\tau(t_0) \ge 2$.

Then $\ttt \psi$ satisfies  for  $|z| \le \delta \mu_0^{-1}$, with sufficiently small $\delta$,
$$
\pp_\tau \ttt \psi =         \Delta_z \ttt \psi  +   a(z,t) \cdot \nn_z \ttt \psi   +  b(z,t) \ttt\psi + \ttt f(z,\tau)
$$
where
$$
\ttt f(z,\tau)  =   \mu_j^2 f( \xi_j + \mu_j  z , t(\tau) ) ,
$$
and the uniformly small coefficients $a(z,t)$ and $b(z,t)$ are given by
$$
a(z,t) := [ \mu_j {\dot \mu_j}  z  + {\dot \xi_j } {\mu_j} ], \quad  b(z,t) =  V_{\mu,\xi} ( \xi_j + \mu_j z )   =  O( R^{-4}) (1+|z|)^{-4}   .
$$
Our assumption in $f$ implies  that in this region
$$
|\ttt f(z,\tau)|\ \lesssim \   t(\tau)^{-\beta} \frac { \|f\|_{*,\beta,2+\alpha }}{ 1 + |z|^{2+\alpha} }
$$
while we have already  established that
$$
|\ttt \psi (z,\tau)| \lesssim  t(\tau)^{-\beta} \frac { \|f\|_{*,\beta,2+\alpha }}{ 1 + |z|^{\alpha} } .
$$
Let us now fix $0<\eta <1$. By standard parabolic estimates we get that for $\tau_1 \ge \tau(t_0) + 2$,
\begin{align}
 [ \nn_z  \ttt \psi( \tau_1  , \cdot ) ]_{\eta, B_{10} (0)}   +    \|  \nn_z  \ttt \psi( \tau_1 , \cdot )  \|_ {L^\infty ( B_{10} (0))}\ & \lesssim\
  \|  \ttt \psi  \|_ {L^\infty ( B_{20} (0)) \times (\tau_1-1 , \tau_1 )}  + \|  \ttt f  \|_ {L^\infty ( B_{20} (0)) \times (\tau_1-1 , \tau_1 )} \nonumber\\
& \lesssim
t(\tau_1 -1)^{-\beta} \|f \|_{*,\beta,2+\alpha} \lesssim  t(\tau_1)^{-\beta} \|f\|_{*,\beta,2+\alpha} .
\nonumber
\end{align}
provided that $\tau_1 \ge 2$.  Translating this estimate to the original variables $(x,t)$ we find  that for any $t \ge   c_n t_0$, for a suitable constant $c_n$,
 \be
 ( R\mu_j )^{1+\eta} [ \nn_x \psi( t  , \cdot ) ]_{\eta, B_{ 10 R\la_j } (\xi_j)}   +     R\mu_j \|  \nn_x\psi( t , \cdot )  \|_ {L^\infty (B_{ 10 R\la_j } (\xi_j))}\ \lesssim\
 t^{-\beta} \|f\|_{*,\beta,2+\alpha} . \label{kk}\ee
The use of a similar parabolic estimate up to the initial condition $0$ at $t_0$ for $\psi$ yields the validity of  estimate \equ{kk}  and hence of \equ{potoo1} for any $t\ge t_0$. The proof  is complete. \qed

\medskip
\noindent
We have now the elements to give the

\medskip
\noindent
\begin{proof}[Proof of Proposition \ref{probesterno}]. \ \  We prove the existence of $\psi$ and the validity of \eqref{leuco1}.

Lemma \ref{psistar1} defines a linear operator $T$ that to any set of functions $f ( x,t) $, $g (x,t)$, and $h  (x)$ associates the solution $\psi = T(f,g,h)$ to  Problem \eqref{hh1}.

Define $\psi_1 (x,t) = T(0, -u_{\mu , \xi }^* , \psi_0)$. Using \eqref{bb1}, \eqref{ansatz0} and \eqref{marshall}, we get that, for any $x \in \partial \Omega$
$$
| u_{\mu , \xi}^* (x,t) | \lesssim \mu_0^{n+2 \over 2} (t).
$$
Thus Lemma \ref{psistar1} gives that
$$
|\psi_1 (x,t) |  \lesssim e^{-\delta (t-t_0 ) } \, \| \psi_0 \|_{L^\infty (\Omega ) } +
t^{-\beta} \mu_0 (t_0)^{2-\sigma}, \quad \beta=  {n-2 \over 2(n-4)} + {\sigma \over n-4}.
$$
The function $\psi + \psi^1$ is thus a solution to \eqref{equpsi1} if $\psi$ is a fixed point for the operator
\be\label{primacontra}
{\mathcal A} (\psi ) :=
 T (f(\psi) , 0 , 0), \quad
\ee
where $T$ is again the operator defined by Lemma \ref{psistar1} and
$$
 f(\psi ) =\sum_{j=1}^k  [2\nn\eta_{j,R} \nn_x \ttt \phi_j +  \ttt \phi_j(\Delta_x-\pp_t)\eta_{j,R} ]\
 +   \ttt N_{\mu,\xi}\Big(  \psi + \psi_1 + \sum_{j=1}^k \ttt\phi_j
  \Big)  +  S_{\mu,\xi}^{o}.
$$
We shall show the existence of a fixed point $\psi$ for ${\mathcal A}$ using the Contraction Mapping Theorem, for functions $\psi$  so that
$$
\| \psi \|_{* , \beta , \alpha} \quad {\mbox {is bounded}}, \quad {\mbox {with}} \quad
\beta=  {n-2 \over 2(n-4)} +{\sigma \over n-4}
$$
(see \eqref{gg}).
To do so, we establish first the following estimates: given $\alpha \in (0 , 1)$, there exists $\ve >0$ so that
\be \label{uno}
\left|  S_{\mu,\xi}^{o}  (x,t) \right| \lesssim  t_0^{-\ve}  \,  \sum_{j=1}^k {\mu_j^{-2} \mu_0^{{n-2 \over 2} +\sigma } (t) \over (1+ |y_j |^{2+\alpha} )},
\ee
\be \label{due}
\left| \sum_{j=1}^k  [2\nn\eta_{j,R} \nn_x \ttt \phi_j +  \ttt \phi_j(\Delta_x-\pp_t)\eta_{j,R} ]\
\right| \lesssim  \, \| \phi \|_{ n-2+ \sigma , a}  \sum_{j=1}^k {\mu_j^{-2} \mu_0^{{n-2 \over 2}+\sigma }  (t) \over (1+ |y_j |^{2+\alpha} )},
\ee
and
\be \label{tre}
\left|  \ttt N_{\mu,\xi}\Big(  \psi + \psi_1 + \phi^{in}
  \Big)  \right| \lesssim  t_0^{-\ve} \,  \left\{
\begin{array}{r}
\left(  \| \phi \|_{n-2+ \sigma , a}^2  + \| \psi \|^2_{* , \beta ,\alpha} \right) \sum_{j=1}^k {\mu_j^{-2} \mu_0^{{n-2 \over 2}+\sigma}  (t)  \over (1+ |y_j |^{2+\alpha} )}  \quad {\mbox {if}} \quad n=5,6 \\
\left(  \| \phi \|_{n-2+\sigma , a}^p  + \| \psi \|^p_{* , \beta , \alpha} \right) \sum_{j=1}^k {\mu_j^{-2} \mu_0^{{n-2 \over 2}+\sigma }  (t) \over (1+ |y_j |^{2+\alpha} )}  \quad \quad \quad  {\mbox {if}} \quad n\geq 7.
\end{array}
\right.
\ee

\medskip
\noindent
{\it Proof of \eqref{uno}}. \ \  We recall (see \eqref{we3}) that
$$
S_{\mu,\xi}^{o}  = S_{\mu,\xi}^{(2)}   +  \sum_{j=1}^k (1- \eta_{j,R} ) S_{\mu,\xi,j},
$$
where $S_{\mu,\xi}^{(2)}$ and $S_{\mu,\xi,j} $ are given by \eqref{we1} and \eqref{we2}, while $\eta_{j,R}$ is the cut-off function introduced in \eqref{we4}.

Using estimate \eqref{form4} and the result in Lemma \ref{lemaerror}, we see that, in the region  $|x-q_j | >\delta$, for some $\delta >0$ small, for all $j$, $S_{\mu , \xi}^{o}$ can be estimated as follows
$$
\left| S_{\mu , \xi}^{o} (x,y) \right| \lesssim \mu_0^{n-2 \over 2} \left[ \mu_0^2 + \mu_0^{n-4} \right] \lesssim \mu_0 (t_0)^{\min (n-4 , 2) -\alpha-\sigma }
\sum_{j=1}^k {\mu_j^{-2} \mu_0^{{n-2 \over 2}  +\sigma } \over (1+ |y_j|^{2+\alpha} )}.
$$
Let us now fix $j \in \{ 1, \ldots , k \}$, and consider the region $|x-q_j | < \delta$, for some $\delta >0$ small. Consequence of Lemma \ref{lemaerror} is that
\begin{align*}
\left| S_{\mu,\xi}^{(2)} (x,t ) \right| &\lesssim \mu_0^{-{n+2 \over 2}} {\mu_0^n \over (1+ |y_j |^2 ) }  \lesssim \mu_0(t_0 )^{2-\alpha-\sigma } \, \sum_{j=1}^k {\mu_j^{-2} \mu_0^{{n-2 \over 2}+\sigma} \over (1+ |y_j|^{2+\alpha} )}.
\end{align*}
Observe now that $(1-\eta_{j,R} ) \not= 0 $ if $|x-\xi_j | > \mu_0 R$. Thus, in $|x-q_j | < \delta$, we see that
\begin{align*}
\left| (1- \eta_{j,R} ) S_{\mu,\xi,j} (x,t) \right| \lesssim
 \left( \mu_0^{n-4-\alpha } + \mu_0^{4-\alpha} \right)
 \, \sum_{j=1}^k {\mu_j^{-2} \mu_0^{{n-2 \over 2}+\sigma } \over (1+ |y_j|^{2+\alpha} )},
\end{align*}
where we use the bounds \equ{rest}-\eqref{rest1} on $\lambda$ and $\xi$. Collecting the above estimates, we get the existence of $\ve$ so that \eqref{uno} is valid.

\medskip
\noindent
{\it Proof of \eqref{uno}}. \ \  Let us fix $j$, and consider first $\nabla \eta_{j,R} \cdot  \nabla \tilde \phi_j$. We recall that $\ttt \phi_j (x,t) := \mu_{0j}^{-\frac{n-2} 2} \phi_j\left (\frac {x- \xi_j} {\mu_{0j}} , t     \right )$, see \eqref{martin1}. Since we are assuming \eqref{rest13}, we have
\begin{align*}
\left| \left( \nabla \eta_{j,R} \cdot  \nabla \tilde \phi_j \right) (x,t) \right|
& \lesssim {\eta' (| {x-\xi_j \over R \mu_{0j} } | ) \over R\mu_{0j} } \, \mu_0^{-{n-2 \over 2}} {|\nabla_y \phi_j | \over \mu_0} \lesssim
  {\eta' (| {x-\xi_j \over R \mu_{0j} } | ) \over R\mu_{0}^2 } \, { \mu_0^{{n-2 \over 2} +\sigma}  \over (1+ |y_j |^{1+ a})} \, \| \phi \|_{n-2+ \sigma , a} \\
&\lesssim  \| \phi \|_{ n-2+ \sigma , a}  \sum_{j=1}^k {\mu_j^{-2} \mu_0^{{n-2 \over 2}+\sigma } \over (1+ |y_j|^{2+a} )}
\end{align*}
where we have used that $(1+ |y_j | ) \sim R$, $y_j = {x-\xi_j \over \mu_{0j}} $, in the region where $\eta' (| {x-\xi_j \over R \mu_{0j} } | )  \not=0$. Since we are assuming $a> \alpha$, we get
$$
\left| \sum_{j=1}^k  \left( \nabla \eta_{j,R} \nabla \ttt \phi_j  \right) (x,t)\right|
\lesssim  \,  \| \phi \|_{n-2+\sigma , a} \,  \sum_{j=1}^k {\mu_j^{-2} \mu_0^{{n-2 \over 2}+\sigma } \over
(1+|y_j|^{2+\alpha} )}.
$$
Let us now consider the term $\ttt \phi_j(\Delta_x-\pp_t)\eta_{j,R} $. A direct computation gives
\begin{align}\label{nina1}
\left| \left( \ttt \phi_j(\Delta_x-\pp_t)\eta_{j,R}  \right) (x,t) \right|& \lesssim
\left|  {\eta'' (| {x-\xi_j \over R \mu_{0j} } | ) \over R^2 \mu_{0}^2 }  \right|  \mu_0^{-{n-2 \over 2}} \left| \phi_j \right| \\
&+ \left|  \eta' (| {x-\xi_j \over R \mu_{0j} } | ) \left( { |x-\xi_j| \over R^2 \mu_0^2 } \dot \mu_0 - \mu_0^{-1} \dot \xi_j \right) \right| \mu_0^{-{n-2 \over 2}} \, \left| \phi_j \right|  .\nonumber
\end{align}
We start with the first term in the right hand side of \eqref{nina1}. Using again the definition of $\ttt \phi_j$ and the assumption \eqref{rest13}, we get
\begin{align*}
\left|  {\eta'' (| {x-\xi_j \over R \mu_{0j} } | ) \over R^2 \mu_{0}^2 }  \right|  \mu_0^{-{n-2 \over 2}} \left| \phi_j \right| & \lesssim
\left|  {\eta'' (| {x-\xi_j \over R \mu_{0j} } | ) \over R^2 \mu_{0}^2 }  \right|  { \mu_0^{{n-2 \over 2}+\sigma} \over (1+ |y_j|^a )} \| \phi \|_{ \sigma , a} \\
&\lesssim   \| \phi \|_{n-2+ \sigma , a}  \sum_{j=1}^k {\mu_j^{-2} \mu_0^{{n-2 \over 2}+\sigma } \over (1+ |y_j|^{2+a} )},
\end{align*}
where we have used that $(1+ |y_j | ) \sim R$, $y_j = {x-\xi_j \over \mu_{0j}} $, in the region where $\eta'' (| {x-\xi_j \over R \mu_{0j} } | )  \not=0$. The second term in the right hand side in \eqref{nina1} can be estimated as follows
\begin{align*}
\left|  \eta' (| {x-\xi_j \over R \mu_{0j} } | ) \left( { |x-\xi_j| \over R^2 \mu_0^2 } \dot \mu_0 - \mu_0^{-1} \dot \xi_j \right) \right| & \mu_0^{-{n-2 \over 2}} \, \left| \phi_j \right| \lesssim{  |  \eta' (| {x-\xi_j \over R \mu_{0j} } | ) | \over \mu_0^2 R^2}  \,
\left( \mu_0^{n-2} R + \mu_0^{n-2+\sigma} R^2 \right)  \mu_0^{-{n-2 \over 2}} \, \left| \phi_j \right| \\
&\lesssim  \, \| \phi \|_{n-2+ \sigma , a}  \sum_{j=1}^k {\mu_j^{-2} \mu_0^{{n-2 \over 2}+\sigma } \over (1+ |y_j|^{2+a} )}.
\end{align*}
Collecting the above estimates, we get the validity of \eqref{due}.

\medskip
\noindent
{\it Proof of \eqref{tre}}. \ \  Since $p-2 \geq 0$ only for dimensions $5$ and $6$, we get
$$
\ttt N_{\mu,\xi}\Big(  \psi + \psi_1+  \sum_{j=1}^k \eta_{j,R} \ttt\phi_j
  \Big) \lesssim \left\{
\begin{array}{r}
(u_{\mu , \xi}^*  )^{p-2} \left[ |\psi  |^2 +|\psi_1 |^2  + \sum |  \eta_{j,R} \ttt \phi_j |^2   \right] \quad {\mbox {if}} \quad n=5,6 \\
\sum_{j} |  \eta_{j,R}\ttt \phi_j |^p  + |\psi |^p + |\psi_1|^p   \quad \quad \quad  {\mbox {if}} \quad n\geq 7.
\end{array}
\right.
$$
Consider then $n=5$, $6$. Thus, we have
\begin{align*}
\left| (u_{\mu , \xi}^*  )^{p-2} ( \eta_{j,R} \ttt \phi_j)^2 \right| &\lesssim
{\mu_0^{n-2 +2\sigma} \over (1+ |y_j|^{2a} )} \, \| \phi \|_{n-2+\sigma , a}^2
 \lesssim \mu_0^{{n-2 \over 2} + \sigma} R^2 \mu_0^2 \| \phi \|_{n-2+ \sigma , a}^2
\sum_{j=1}^k  {\mu_j^{-2} \mu_0^{{n-2 \over 2}+\sigma } \over (1+ |y_j|^{2+a} )},
\end{align*}
and also
\begin{align*}
\left| (u_{\mu , \xi}^*  )^{p-2} \psi^2 \right| &\lesssim
{t^{-2 \beta} \over (1+ |y_j|^{2\alpha } )} \, \| \psi \|_{* , \beta , 2+\alpha}^2
 \lesssim t^{-\beta}  \| \psi \|_{* , \beta , \alpha}^2
\sum_{j=1}^k  {\mu_j^{-2} t^{-\beta} \over (1+ |y_j|^{2+\alpha} )}.
\end{align*}
Consider now $n\geq 7$.
\begin{align*}
\left|   \eta_{j,R} \ttt \phi_j \right|^p &\lesssim
{\mu_0^{( {n-2 \over 2}  +\sigma) p } \over (1+ |y_j|^{ap } )} \, \| \phi \|_{n-2+ \sigma , a}^p
 \lesssim \mu_0^{2 + (p-1) \sigma} \| \phi \|_{ \sigma , a}^p  \, R^2 \mu_0^2
\sum_{j=1}^k  {\mu_j^{-2} \mu_0^{{n-2 \over 2}+\sigma } \over (1+ |y_j|^{2+a} )},
\end{align*}
and also
\begin{align*}
\left|  \psi \right|^p  &\lesssim
{t^{-p \beta} \over (1+ |y_j|^{p\alpha } )} \, \| \psi \|_{* , \beta , 2+\alpha}^p
 \lesssim t^{- (p-1) \beta}  \| \psi \|_{* , \beta , \alpha}^p
\sum_{j=1}^k  {\mu_j^{-2} t^{-\beta} \over (1+ |y_j|^{2+\alpha} )},
\end{align*}
and an analogous estimation holds for $\psi_1$. We thus get the validity of \eqref{tre}.

\medskip
\noindent
\medskip
\noindent
Define
$$
{\mathcal B} = \{ \psi  \, : \, \| \psi \|_{*, \beta , \alpha} \leq  M \,   t_0^{-\ve} \},
$$
where $\beta = {n-2 \over 2(n-4)}+{\sigma \over n-4}$, and $\alpha $, $\ve$ are fixed above. Moreover, $M$ is a positive large constant, independent of $t$ and $t_0$.

For any $\psi \in {\mathcal B}$, we have that ${\mathcal A} (\psi ) \in {\mathcal B}$, as follows directly from \equ{uno}, \eqref{due} and \equ{tre}, provided $M$ is chosen large enough.  Furthermore, for any $\psi_1 $, $\psi_2 \in {\mathcal B}$, we get
$$
\| {\mathcal A} (\psi^{(1)} ) - {\mathcal A} (\psi^{(2)} ) \|_{*, \beta , \alpha} \leq C
\| \psi^{(1)}  -  \psi^{(2)}  \|_{* , \beta , \alpha} ,
$$
where $C$ is a constant, whose definition depends on $t_0$,
and it is less than $1$, if $t_0 $ is chosen large.
Indeed, observe that
$$
{\mathcal A} (\psi^{(1)} ) - {\mathcal A} (\psi^{(2)} ) =
 T \left(\ttt N_{\mu , \xi} (\psi^{(1)} + \psi_1 + \phi^{in} ) - \ttt N_{\mu , \xi} (\psi^{(2)} + \psi_1 + \phi^{in} )  , 0 , 0 \right),
$$
see \eqref{primacontra}, where
$$
\ttt N_{\mu , \xi} (\psi^{(1)} + \psi_1 + \phi^{in} ) - \ttt N_{\mu , \xi} (\psi^{(2)} + \psi_1 + \phi^{in} ) =
\left( u_{\mu , \xi}^* + \psi^{(1)} + \psi_1 + \phi^{in} \right)^p - \left( u_{\mu , \xi}^* + \psi^{(2)} + \psi_1  + \phi^{in} \right)^p
$$
$$
- p ( u_{\mu , \xi}^* )^{p-1} \left[
  \psi^{(1)} - \psi^{(2)} \right].
$$
Thus we get
\begin{align*}
&\left| \ttt N_{\mu , \xi} (\psi^{(1)} + \psi_1  + \phi^{in} )  - \ttt N_{\mu , \xi} (
\psi^{(2)} + \psi_1  + \phi^{in} )
\right|
&\lesssim \left\{
\begin{array}{r}
(u_{\mu , \xi}^*  )^{p-2}   |\phi^{in} | \, |\psi^{(1)} - \psi^{(2)}|  \quad {\mbox {if}} \quad n=5,6 \\
\sum_{j} |\phi^{in} |^{p-1}  |\psi^{(1)} - \psi^{(2)} |   \quad \quad \quad  {\mbox {if}} \quad n\geq 7.
\end{array}
\right.
\end{align*}
In dimensions $n=5,6$, we get
\begin{align*}
& \left| \ttt N_{\mu , \xi} (\psi^{(1)} + \psi_1  + \phi^{in} )  - \ttt N_{\mu , \xi} (\psi^{(2)} + \psi_1  + \phi^{in} )
\right| \\
&\lesssim \| \phi \|_{n-2+ \sigma , a} \, \| \psi^{(1)} - \psi^{(2)} \|_{ \beta , \alpha} \, R^{2-a} \mu_0(t_0)^{{n+2 \over 2} + \sigma} \sum_{j=1}^k {\mu_j^{-2} t^{-\beta} \over (1+ |y_j|^{2+\alpha})} ,
\end{align*}
while in dimensions $n\geq 7$ we have
\begin{align*}
 \left| \ttt N_{\mu , \xi} (\psi^{(1)} + \psi_1+ \phi^{in} )  - \ttt N_{\mu , \xi} (
\psi^{(2)} + \psi_1 + \phi^{in} )
\right|  &
\lesssim \| \phi \|_{ n-2+ \sigma , a}^{p-1} \, \| \psi^{(1)} - \psi^{(2)} \|_{ \beta , \alpha} \, \times \\
& \times R^{2-a (p-1)} \mu_0(t_0)^{4 + (p-1)\sigma} \sum_{j=1}^k {\mu_j^{-2} t^{-\beta} \over (1+ |y_j|^{2+\alpha})} ,
\end{align*}
There exists a choice of  $R$ of the form \eqref{defR} for which we get that
$$
\| {\mathcal A} (\psi^{(1)} ) - {\mathcal A} (\psi^{(2)} ) \|_{*,\beta, \alpha} \leq C
\| \psi^{(1)}  -  \psi^{(2)}  \|_{*, \beta, \alpha}
$$
where $C <1$, provided $t_0 $ is large enough.
Using the above estimates, we readily see that if $t_0$  is fixed sufficiently large,  then the operator ${\mathcal A}$  defines a contraction map in the set ${\mathcal B}$.
The existence result of the Proposition \ref{probesterno} thus follows, as well as the validity of \eqref{leuco1}. Estimate \eqref{leuco11} follows directly from \eqref{potoo1}.

\end{proof}

\begin{remark}\label{rmk1}
Proposition \ref{probesterno} defines the solution to Problem \eqref{equpsi1} as a function of the initial condition $\psi_0$, in the form of a linear operator $ \psi = \bar \Psi [\psi_0]$, from a small neighborhood of $0$ in the Banach space $L^\infty (\Omega)$ equipped with the $C^1$ norm
$\| \psi_0 \|_{L^\infty (\Omega) }+ \| \nabla \psi_0 \|_{L^\infty (\Omega )}$ into the Banach space of functions $\psi \in L^\infty (\Omega)$ equipped with the norm
$ \| \psi \|_{* , \beta , \alpha}$ , defined in \eqref{gg}, with $\beta= {n-2 +2\sigma\over 2 (n-4)   }$, and $0<\alpha <1$.

A closer look to the proof of Proposition \ref{probesterno}, and the Implicit function Theorem give that $\bar \Psi [\psi_0]$ is a diffeomorphism, and that
$$
\| \bar \Psi [\psi_0^1 ] - \bar \Psi [\psi_0^2] \|_{* , \beta , \alpha} \leq c \left[
\| \psi_0^1 - \psi_0^2\|_{L^\infty (\Omega )} +
\| \nabla \psi_0^1 - \nabla \psi_0^2\|_{L^\infty (\Omega )} \right]
$$

\end{remark}

\medskip
The function $\psi = \Psi (\la , \xi , \dot \la, \dot \xi , \phi )$ solution to Problem \eqref{equpsi1} depends also on the parameter functions $\la $, $\xi $, $\dot \la $, $\dot \xi $, and $\psi$. Next Proposition
clarifies this dependence. This is done by  estimating, for instance,
$$\pp_\phi \Psi  (\la , \xi , \dot \la , \dot \xi , \phi) [\bar \phi] = \pp_s \Psi [\phi + s\bar \phi,\la,\xi ]\Big|_{s=0} $$
as linear operator between Banach spaces.

We have the validity of the following

\begin{prop}\label{probesterno1}
Assume the validity of the hypothesis in Proposition \ref{probesterno},
Then,  $\Psi$ depends smoothly on $\la$, $\xi$, $\dot \la$ , $\dot \xi $, $\phi$, and we have, for $y_j = {x-\xi_j \over \mu_{0j} }$,
\be \label{leuco2}
|\pp_\la \Psi (\la , \xi , \dot \la , \dot \xi , \phi) [\bar \la ] (x,t)| \ \lesssim \    \, \mu_0^{1+ \sigma} (t) \, \| \bar \la (t) \|_{1+\sigma}  \, \left(  \sum_{j=1}^k \frac{\mu_0^{\frac{n-4}2   } (t) }{|y_j|^\alpha + 1 } \right) ,
\ee
\be  \label{leuco2new}
|\pp_\xi  \Psi  (\la , \xi , \dot \la , \dot \xi , \phi) [\bar \xi ] (x,t)| \lesssim \      \, \, \mu_0^{1+ \sigma} (t) \,  \| \bar \xi (t) \|_{1+\sigma}   \,  \left( \sum_{j=1}^k \frac{\mu_0^{\frac{n-4}2   } (t) }{|y_j|^\alpha + 1 }  \right), \quad
\ee
\be \label{leuco22}
 |\pp_{\dot \la}  \Psi (\la , \xi , \dot \la , \dot \xi , \phi) [\dot {\bar \la } ] (x,t)| \ \lesssim \,   t_0^{-\ve}  \, \mu_0^{n-3+ \sigma} (t)  \,  \|\dot {\bar \la } (t) \|_{n-3+\sigma} \, \left( \sum_{j=1}^k \frac{  \,  \mu_0^{-{n-4 \over 2} +\sigma } (t) }{|y_j|^\alpha + 1 }  \right), \quad
\ee
\be \label{leuco22new}
|\pp_{\dot  \xi } \Psi (\la , \xi , \dot \la , \dot \xi , \phi) [ \dot {\bar \xi } ] (x,t)|  \lesssim  t_0^{-\ve}  \, \mu_0^{n-3+ \sigma} (t) \,  \|\dot {\bar \xi } (t) \|_{n-3+\sigma}  \, \left( \sum_{j=1}^k \frac{  \,  \mu_0^{-{n-4 \over 2} +\sigma} (t) }{|y_j|^\alpha + 1 } \right) ,
\ee
and
\be \label{leuco3}
|\pp_\phi \Psi (\la , \xi , \dot \la , \dot \xi , \phi) [\bar \phi](x,t) | \ \lesssim \ t_0^{-\ve} \,  \|\bar \phi\|_{n-2+\sigma , a}  \,  \left( \sum_{j=1}^k \frac{\mu_0^{\frac{n-2}2  +\sigma} }{|y_j|^\alpha + 1 } \right) .
\ee
We refer to \eqref{normfinal}, \eqref{rest}, \eqref{rest1}, \eqref{starstar}, \eqref{rest13}
for the definitions of the norms.
\end{prop}

\begin{proof}

 For notational convenience we denote the  operator $\pp_\phi \Psi  (\la , \xi , \dot \la , \dot \xi , \phi) [\bar \phi]$ by $\partial_\phi \Psi [\bar \phi]$. Similarly we denote by  $\partial_\la \Psi [\bar \la ] $, $\partial_\xi \Psi [\bar \xi ] $, $\partial_{\dot \la} \Psi [\dot {\bar \la} ] $ , $\partial_{\dot \xi } \Psi [\dot {\bar \xi} ] $ the other linear operators.

We divide the proof on three steps, for the proof of \eqref{leuco2}-\eqref{leuco2new},
for the proof of \eqref{leuco22}=\eqref{leuco22new}, and finally for the proof of \eqref{leuco3}.

\bigskip
\noindent
{\bf Step 1.} \ \ Proof of estimate \eqref{leuco2}.

Let us fix $j=1$. For any $\la_1$ satisfying \eqref{rest}, the function $\Psi [\la_1] $ is a solution to  Problem \eqref{equpsi1}. Differentiating Problem \eqref{equpsi1} with respect to $\la_1$ we get a non linear problem, and the Implicit Function Theorems ensures that  the solution is given by
$\partial_{\la_1} \Psi [\bar \la_1] (x,t)$. If we decompose  $ \partial_{\la_1} \Psi [\bar \la_1] (x,t)= Z_1 + Z$, with $Z_1= T(0,- (\partial_{\lambda_1 } u_{\mu,\xi}^*) [\bar \la_1]  , 0) $, where $T$ is the operator defined in
Lemma \ref{psistar1}, then  $Z$ is the solution to
\begin{align}\label{huff}
\pp_t Z =&  \Delta Z+  V_{\mu,\xi} Z  \ + (\partial_{\lambda_1 }  V_{\mu,\xi} )   [\bar \la_1]  \psi
 + \partial_{\lambda_1}  \left[ \ttt N_{\mu,\xi}\Big(  \psi + \phi^{in}
  \Big)  \right]  [\bar \la_1] +  \partial_{\lambda_1} \left( S_{\mu,\xi}^{o} \right)  [\bar \la_1]   \inn \Omega \times [t_0,\infty), \nonumber \\
\psi =& 0 \onn  \pp\Omega \times [t_0,\infty), \quad \psi(t_0,\cdot) = 0 \inn \Omega.
\end{align}
Observe that $\sum_{j=1}^k  [2\nn\eta_{j,R} \nn_x \ttt \phi_j +  \ttt \phi_j(\Delta_x-\pp_t)\eta_{j,R} ]$ is independent of $\lambda_1$, as follows from its very definition.

Using \eqref{bb1}, \eqref{ansatz0} and \eqref{marshall}, we get that, for any $x \in \partial \Omega$
$$
| \partial_{\lambda_1 } u_{\mu , \xi}^* (x,t)  [\bar \la_1]  | \lesssim \mu_0^{n \over 2} (t) |\la_1 (t) |
$$
Thus Lemma \ref{psistar1} gives that
\begin{equation}\label{huff4}
|Z_1 (x,t) |  \lesssim \mu_0(t)^{{n-4 \over 2}+\sigma } \,
  \mu_0 (t_0)^{2-\sigma} \, |\la_1 (t) |  .
\end{equation}
To treat Problem \eqref{huff}, we start with the observation that
\begin{align*}
\partial_{\lambda_1}  \left[ \ttt N_{\mu,\xi}\Big(  \psi + \phi^{in}
  \Big)  \right]   [\bar \la_1]  &= p \left[ (u_{\mu , \xi}^* + \phi^{in}  )^{p-1} -  (u_{\mu , \xi}^*  )^{p-1} \right] \, \left( Z + Z_1 \right) \\
&+ \left[ (u_{\mu , \xi}^* + \phi^{in}  )^{p-1} -  (u_{\mu , \xi}^*  )^{p-1}
- (p-1) (u_{\mu,\xi}^*)^{p-2} \phi^{in}    \right] \, \partial_{\lambda_1} u_{\mu , \xi}^*  [\bar \la_1] .
\end{align*}
Thus, $Z$ is a fixed point for the operator
\begin{equation}\label{huff5}
{\mathcal A}_1 (Z) = T (f+ p \left[ (u_{\mu , \xi}^* +  \phi^{in} )^{p-1} -  (u_{\mu , \xi}^*  )^{p-1} \right] Z, 0 , 0)
\end{equation}
where
\begin{align*}
f&=\partial_{\lambda_1} \left( S_{\mu,\xi}^{o} \right)  [\bar \la_1]
+   (\partial_{\lambda_1 }  V_{\mu,\xi} )  [\bar \la_1]  \psi
+p \left[ (u_{\mu , \xi}^* + \phi^{in} )^{p-1} -  (u_{\mu , \xi}^*  )^{p-1} \right] \, Z_1  \\
&+ \left[ (u_{\mu , \xi}^* + \phi^{in} )^{p-1} -  (u_{\mu , \xi}^*  )^{p-1}
- (p-1) (u_{\mu,\xi}^*)^{p-2} \phi^{in}   \right] \, \partial_{\lambda_1} u_{\mu , \xi}^*  [\bar \la_1] .
\end{align*}
We claim that, there exists $\ve >0$ so that
\begin{equation}
\label{huff1}
\left| f (x,t) \right| \lesssim \left( t_0^{-\ve} +\| \phi \|_{ n-2+ \sigma , a} \right)\, \mu_0^{1+\sigma } (t) \,  \|\la_1 \|_{1+\sigma}  \sum_{j=1}^k { \mu_j^{-2} \mu_0^{{n-4 \over 2}+\sigma }  \over (1+ |y_j|^{2+\alpha} )}  .
\end{equation}
We start with the estimate of $ \partial_{\lambda_1} \left( S_{\mu,\xi}^{o} \right)  [\bar \la_1]  $. In order to estimate this term, we shall understand $\partial_{\lambda_1} S(u_{\mu , \xi}^* )  [\bar \la_1]  $
A direct differentiation in $\lambda_1$ of $S(u_{\mu , \xi}^* )$ given in \eqref{SSnew},
implies that
in the region $|x-q_i | > \delta$ for any $i=1, \ldots , k$, we can describe  $\partial_{\lambda_1} S (u_{\mu , \xi }^* )$ as follows
\be \label{form444}
\partial_{\lambda_1} S (u_{\mu , \xi }^* )  [\bar \la_1]   (x,t) =
 \mu_0^{n \over 2} \, f ( x, \mu_0^{-1} \mu, \xi )\,  \la_1 (t)
\ee
where  $f$ is a smooth and bounded function of $( x, \mu_0^{-1} \mu, \xi )$. This is consequence of \eqref{marshall}  and the assumptions on the parameters $\la$, $\xi$ given in \eqref{rest}-\eqref{rest1}.

Let us now fix $j$, and consider the region $|x-q_j| \leq \delta$. Using again \eqref{SSnew}, and a direct differentiation, we get that
$$
\partial_{\lambda_1} S (u_{\mu , \xi }^* ) [\bar \la_1] (x,t)  =\partial_{\lambda_1} S (u_{\mu , \xi })   [\bar \la_1] (x,t) \, \left( 1+ \mu_0 f ( x, t, \mu_0^{-1} \mu, \xi ) \right)
 $$
where $f$ is a smooth and bounded function. Now, a direct and explicit differentiation with respect to $\lambda_1$ in expressions \eqref{form1} and \eqref{form2} gives that,
\begin{align*}
\partial_{\lambda_1} S (u_{\mu , \xi })& [\bar \la_1] (x,t) =
-{n\over 2} \mu_1^{-{n+2 \over 2}} \left[ \dot \mu_1 Z_{n+1} (y_1 ) + \dot \xi_1 \nabla U(y_1 ) - {2\over n} {n-2 \over 2} {n-4 \over 2} \mu_1^{n-2 } \dot \mu_1 H(x, q_1)\right] \,  \bar \la_1  (t) \\
&- \mu_1^{-{n\over 2} } \left( \dot \mu_1 \nabla Z_{n+1} (y_1) + D \nabla U(y_1 ) \right) \cdot {\xi_1 \over \mu_1^2} \bar \la_1 (t) \\
&+ p \left( \sum_{i=1}^k  \mu_i^{-\frac {n-2}2}   U(y_i)  -  \mu_i^{\frac{n-2}2} H(x,q_i)  \right)^{p-1}
\partial_{\lambda_1} [ \mu_1^{-\frac {n-2}2}   U(y_1)  -  \mu_1^{\frac{n-2}2} H(x,q_1) ] \bar \la_1 (t) \\
&- p \left(  \mu_1^{-\frac {n-2}2}   U(y_i) \right)^{p-1} \partial_{\lambda_1} [ \mu_1^{-\frac {n-2}2}   U(y_1) ] \bar \la_1 (t) .
\end{align*}
A carefull analysis of the above terms, and taking into account the restrictions \eqref{rest}- \eqref{rest1} on the parameters $\la$, $\xi$, give
\begin{align*}
\left| \partial_{\lambda_1} S (u_{\mu , \xi })  [\bar \la_1] (x,t)  \right|
\lesssim   \mu_0^{1+\sigma } (t) \,  \|\la_1 \|_{1+\sigma}  \sum_{j=1}^k { \mu_j^{-2} \mu_0^{n-4 \over 2}  \over (1+ |y_j|^{2+\alpha} )}
\end{align*}
This fact, together with the previous arguments gives the validity of
\begin{equation}\label{huff2}
\left| \partial_{\lambda_1 }  S^o_{\mu , \xi} (x,t) [\bar \la_1]  \right| \lesssim  \mu_0^{1+\sigma } (t) \,  \|\la_1 \|_{1+\sigma}  \sum_{j=1}^k { \mu_j^{-2} \mu_0^{n-4 \over 2}  \over (1+ |y_j|^{2+\alpha} )}.
\end{equation}
We shall next estimate the remaining terms in $f$. A direct computation gives that
$$
(\partial_{\lambda_1} V_{\mu , \xi }) [\bar \la_1]  (x,t) = p (p-1) \Bigl[ \left( (u_{\mu , \xi}^* )^{p-2}
- (\mu_1^{-{n-2 \over 2}} U(y_1) )^{p-2} \right) \partial_{\lambda_1}  (\mu_1^{-{n-2 \over 2}} U(y_1) ) \eta_{1,R}
$$
$$
+
\left( 1- \sum_{j=1}^k \eta_{j,R} \right) \, (u_{\mu , \xi}^* )^{p-2} \partial_{\lambda_1}  (\mu_1^{-{n-2 \over 2}} U(y_1) ) \Bigl] \, \bar \la_1 (t) .
$$
Since $| \partial_{\lambda_1}  (\mu_1^{-{n-2 \over 2}} U(y_1) ) | \lesssim
\mu_0^{-1}  |\mu_1^{-{n-2 \over 2}} U(y_1) |$, we get that
\begin{align*}
\left| ( \partial_{\lambda_1} V_{\mu , \xi } ) [\bar \la_1]  \psi  (x,t) \right| &
\lesssim \| \psi \|_{* , \beta , \alpha} \, \mu_0^{1+\sigma } (t) \,  \|\la_1 \|_{1+\sigma}  \, \sum_{j=1}^k {\mu_j^{-2} \mu_0^{{n-4 \over 2} + \sigma } (t) \over (1+ |y_j |^{2+ \alpha } ) },
\end{align*}
using that $\beta = {n-2 \over 2(n-4)}+ {\sigma \over n-4} $. Thus,  we get the expected estimate.  In a very analogous way, we can treat the term $\left[ (u_{\mu , \xi}^* + \phi^{in} )^{p-1} -  (u_{\mu , \xi}^*  )^{p-1}
- (p-1) (u_{\mu,\xi}^*)^{p-2} \phi^{in}   \right] \, (\partial_{\lambda_1} u_{\mu , \xi}^*) \, |\bar \la_1 (t) | $, and get
\begin{align*}
& \left| \left[ (u_{\mu , \xi}^* + \phi^{in} )^{p-1} -  (u_{\mu , \xi}^*  )^{p-1}
- (p-1) (u_{\mu,\xi}^*)^{p-2} \phi^{in}   \right] \, \partial_{\lambda_1} u_{\mu , \xi}^*\right| \\
& \lesssim t_0^{-\ve} \mu_0^{1+\sigma } (t) \,  \|\la_1 \|_{1+\sigma} \, \sum_{j=1}^k { \mu_j^{-2} \mu_0^{{n-4 \over 2}+\sigma } (t) \over (1+ |y_j|^{2+\alpha} )}.
\end{align*}
Similarly one can also treat the last term
$
p \left[ (u_{\mu , \xi}^* + \phi^{in} )^{p-1} -  (u_{\mu , \xi}^*  )^{p-1} \right] \, Z_1
$, and \eqref{huff1} follows from \eqref{huff4}.

\medskip
We now go back to the fixed point problem \eqref{huff5}. Arguing as in the argument of Step 1, we can show that  ${\mathcal A}_1$, defined in \eqref{huff5},  has a fixed point for functions in the set
$ |Z(x,t)|  \leq M  \mu_0^{1+\sigma} (t)  \left( \sum_{j=1}^k {  \mu_0^{{n-4 \over 2} } (t) \over (1+ |y_j|^{\alpha} )} \right)$, for some constant $M$ large and fixed. Indeed, ${\mathcal A}_1$
 is Lipschtz, with Lipschitz constant less than $1$, provided $R$ is chosen properly in terms of $t_0$.
The validity of \eqref{leuco2} for $\partial_{\la_1}  \Psi  [\bar \la_1] $ thus follows. The estimate in \eqref{leuco2new} on $\partial_\xi \Psi [\bar \xi] $ can be obtained arguing as before, with some small modifications.

\bigskip
\noindent
{\bf Step 2.} \ \ Proof of estimate \eqref{leuco2}.

Let us fix $j=1$. For any ${\dot \la}_1$ satisfying \eqref{rest}, the function $\Psi [{\dot \la}_1] $ is a solution to  Problem \eqref{equpsi1}. Differentiating Problem \eqref{equpsi1} with respect to ${\dot \la}_1$ we get a non linear problem, and the Implicit Function Theorems ensures that  the solution is given by
$\partial_{{\dot \la}_1} \Psi [{\dot {\bar \la}}_1] (x,t)$.
Let   $Z(x,t) =  \partial_{\la_1} \Psi [\dot{ \bar \la}_1] (x,t)$. Then $Z$ is the solution to
\begin{align}\label{hufff}
\pp_t Z =&  \Delta Z+  V_{\mu,\xi} Z  \
 + \partial_{\dot \lambda_1}  \left[ \ttt N_{\mu,\xi}\Big(  \psi + \phi^{in}
  \Big)  \right]  [\dot { \bar \la}_1 ]+  \partial_{\dot \lambda_1} \left( S_{\mu,\xi}^{o} \right)  [\dot { \bar \la}_1 ]  \inn \Omega \times [t_0,\infty), \\
\psi =& 0 \onn  \pp\Omega \times [t_0,\infty), \quad \psi(t_0,\cdot) = 0 \inn \Omega. \nonumber
\end{align}
Observe that
\begin{align*}
\partial_{\dot \lambda_1}  \left[ \ttt N_{\mu,\xi}\Big(  \psi + \phi^{in}
  \Big)  \right]  [\dot { \bar \la}_1 ] &= p \left[ (u_{\mu , \xi}^* + \phi^{in}  )^{p-1} -  (u_{\mu , \xi}^*  )^{p-1} \right] \, Z (x,t)  \,
\end{align*}
Thus, $Z$ is a fixed point for the operator
\begin{equation}\label{hufff5}
{\mathcal A}_1 (Z) := T ( \partial_{\dot \lambda_1} \left( S_{\mu,\xi}^{o} \right)  [\dot { \bar \la}_1 ] + p \left[ (u_{\mu , \xi}^* +  \phi^{in} )^{p-1} -  (u_{\mu , \xi}^*  )^{p-1} \right] Z, 0 , 0),
\end{equation}
where  $T$ is the linear operator defined by  Lemma \ref{psistar1}  that to any set of functions $f ( x,t) $, $g (x,t)$, and $h  (x)$ associates the solution $\psi = T(f,g,h)$ to  Problem \eqref{hh1}.
Now, a direct and explicit differentiation with respect to $\dot \lambda_1$ in expressions \eqref{form1}, \eqref{form2} and \eqref{nonsaprei} gives that,
\begin{align*}
\partial_{\dot \lambda_1} S (u_{\mu , \xi }^*)  [\dot { \bar \la}_1 ] (x,t) &= \mu_1^{-{n \over 2}}  \left[ Z_{n+1} (y_1) +{n-2 \over 2} \mu_1^{n-2} H(x, q_1) \right]  \, \dot { \bar \la}_1 (t)  \\
&-\mu_1^{-{n \over 2}}  [ {n-2 \over 2} \Phi_{j} (y_j , t ) +\nabla \Phi_{j} (y_j , t ) \cdot y_j ]  \dot { \bar \la}_1 (t)
\end{align*}
Thus we get
\begin{align*}
\left| \partial_{\dot \lambda_1} S (u_{\mu , \xi }^*)   [\dot { \bar \la}_1 ]  (x,t)  \right|
\lesssim {\mu_0^{-{ n-4 \over 2}} \over R^{n-4-\alpha}} \, \mu_0^{n-3 +\sigma} (t ) \|  \dot { \bar \la}_1 \|_{n-3+\sigma} \sum_{j=1}^k { \mu_j^{-2}   \over (1+ |y_j|^{2+\alpha} )}.
\end{align*}

\medskip
We now go back to the fixed point problem \eqref{hufff5}. Arguing as in Step 2, we can show that  ${\mathcal A}_1$ has a fixed point in the set of functions
$$
\left| Z (x,t)   \right| \, \mu_0^{-(n-3+\sigma)}
\lesssim {\mu_0^{-{ n-4 \over 2}} \over R^{n-4-\alpha}} \sum_{j=1}^k { 1 \over (1+ |y_j|^{\alpha} )}.
$$
This concludes the proof of the first estimate in  \eqref{leuco22}. The  estimate in \eqref{leuco22new} follows after we observe that
$$
\partial_{\dot \xi_1} S (u_{\mu , \xi}^*   [\dot { \bar \xi}_1 ] (x,t) =
\mu_1^{-{n \over 2}} \left[ \nabla U(y_1) + \nabla \Phi_{1} (y_1 , t) \right]\cdot  [\dot { \bar \xi}_1 ],
$$
from which we readily get
\begin{align*}
\left| \partial_{\dot \xi_1} S (u_{\mu , \xi }^*)   [\dot { \bar \xi}_1 ] (x,t)  \right|
\lesssim {\mu_0^{-{ n-4 \over 2}} \over R^{n-3-\alpha}} \sum_{j=1}^k { \mu_j^{-2}   \over (1+ |y_j|^{2+\alpha} )}.
\end{align*}

\bigskip
\noindent
{\bf Step 3.} \ \ Proof of estimate \eqref{leuco3}.
Let us define $Z (x,t) = \partial_{\phi } \psi [ \bar \phi] (x,t) $, for functions $\bar \phi$ satisfying \eqref{rest13}. Thus $Z$ solves
\begin{align}\label{huff6}
\pp_t Z =&  \Delta \psi +  V_{\mu,\xi} Z  \ + \sum_{j=1}^k  [2\nn\eta_{j,R} \nn_x \hat \phi_j + \hat   \phi_j(\Delta_x-\pp_t)\eta_{j,R} ]
     \\
&+  p [ (u_{\mu , \xi}^* + \psi + \phi^{in} )^{p-1} - (u_{\mu , \xi}^* )^{p-1}] \bar \phi   \inn \Omega \times [t_0,\infty), \nonumber \\
\psi =& 0 \onn  \pp\Omega \times [t_0,\infty), \quad \psi(t_0,\cdot) = 0 \inn \Omega. \nonumber
\end{align}
where $\hat \phi_j (x,t) := \mu_{0j}^{-\frac{n-2} 2} \bar \phi_j\left (\frac {x- \xi_j} {\mu_{0j}} , t     \right )$.

Arguing as in the proof of \eqref{leuco2}, we can show that
$$
\left| \sum_{j=1}^k  [2\nn\eta_{j,R} \nn_x \hat \phi_j + \hat   \phi_j(\Delta_x-\pp_t)\eta_{j,R} ] \right| \lesssim t_0^{-\ve} \| \bar \phi \|_{ \sigma , a} \left( \sum_{j=1}^k {\mu_j^{-2} \mu_0^{n-2 \over 2} (t) \over (1+ |y_j|^{2+ \alpha} ) } \right),
$$
and also that
\begin{align*}
&\left| p \left[ (u_{\mu , \xi}^* + \psi +  \phi^{in} )^{p-1} -  (u_{\mu , \xi}^*  )^{p-1} \right]  \bar \phi   \right|  \lesssim  t_0^{-\ve} \| \bar \phi \|_{ \sigma , a}
\, \left[ \| \psi \|_{* , \beta , \alpha} ^{p-1} + \| \phi^{in} \|_{\sigma , a}^{p-1} \right] \,  \left( \sum_{j=1}^k {\mu_j^{-2} \mu_0^{n-2 \over 2} (t) \over (1+ |y_j|^{2+ \alpha} ) } \right) .
\end{align*}
A direct application of Lemma \ref{psistar1} gives \eqref{leuco3}.

This concludes the proof of Proposition \ref{probesterno}.
\end{proof}

\setcounter{equation}{0}
\section{ Choice of the parameters  $\la $ and $\xi$ in the inner problem }\label{sec5}

Let $\psi = \Psi[\phi,\la,\xi , \dot \la , \dot \xi ] $ be the function in Proposition \ref{probesterno}. After substituting $\psi = \Psi[\phi,\la,\xi , \dot \la , \dot \xi ] $ into the inner problem \equ{equ3}, the full problem gets reduced to solving the following system of equations
 for any $j=1,\ldots,k$,
\begin{equation} \label{equ333}
\mu_{0j}^2 \pp_t  \phi_j  =  \Delta_y \phi_j +  pU(y)^{p-1} \phi_j
+ H_j [\phi , \la , \xi , \dot \la , \dot \xi] (y,t) , \quad y \in B_{2R} (0), \quad t \geq t_0
\end{equation}
for $j=1,\ldots, k$,
where
\begin{align}\label{defHj}
&H_j  [\phi , \la , \xi , \dot \la , \dot \xi] (y,t)=
\mu_{0j}^{\frac{n+2}2}  S_{\mu,\xi,j}  (\xi_j +\mu_{0j} y,t) \\
&+  p \mu_{0j}^{\frac{n-2}2} {\mu_{0j}^2 \over \mu_{j}^2 } U^{p-1} ( {\mu_{0j} \over \mu_j}  y) \psi(\xi_j +\mu_{0j} y,t)
 +  B_j[\phi_j] + B_j^0 [\phi_j]   \nonumber
\end{align}
with
$B_j [\phi_j] $ and $B_j^0 [\phi_j] $ defined respectively in \eqref{gajardo1} and \eqref{gajardo2}.

\medskip
We next describe precisely our strategy to solve \eqref{equ333}.
Consider the  change of variable,
$$
t = t (\tau ) , \quad
{dt \over d\tau } = \mu_{0j}^2  (t) ,
$$
that reduces \eqref{equ333} to
\begin{equation} \label{equ3331}
 \pp_\tau  \phi_j  =  \Delta_y \phi_j +  pU(y)^{p-1} \phi_j
+ H_j [\phi , \la , \xi , \dot \la , \dot \xi] (y,t (\tau ) ) , \quad y \in B_{2R} (0), \quad \tau \geq \tau_0
\end{equation}
where $\tau_0 $ is such that $t (\tau_0 ) = t_0$. This is to say
\be \label{ttau}
t^{n-2 \over n-4} = {n-2 \over n-4}\,  \tau.
\ee
We shall construct a solution  $\phi=(\phi_1,\ldots, \phi_k)$ of the system
\begin{align} \label{equ3332}
 \pp_\tau  \phi_j  &=  \Delta_y \phi_j +  pU(y)^{p-1} \phi_j
+ H_j [\phi , \la , \xi , \dot \la , \dot \xi] (y,t (\tau ) ) , \quad y \in B_{2R} (0), \quad \tau \geq \tau_0\nonumber \\
& \phi_j (y, \tau_0 ) = e_{0j} Z_0 (y) , \quad y \in B_{2R} (0),
\end{align}
for some constant $e_{0j}$, and all $j=1,\ldots, k$.

\medskip
We will prove that the system of equations \eqref{equ3332} is solvable in the class of functions $\phi_j$ that satisfy \eqref{rest13}, provided that in addition
 the parameters  $\xi$ and $\la$ are chosen so that the functions  $H_j [\phi , \la , \xi , \dot \la , \dot \xi] (y,t (\tau ) ) $ satisfy the orthogonality conditions
\begin{align}\label{gajardo3}
\int_{B_{2R} } H_j [\phi , \la , \xi , \dot \la , \dot \xi] (y,t (\tau ) )  Z_\ell (y)  dy &= 0, \quad {\mbox {for all}} \quad  t>t_0 \\
&{\mbox {for all}}
\quad j=1, \ldots , k , \, \ell=1, 2, \ldots , n+1. \nonumber
\end{align}
We recall that $Z_\ell (y) $, for $\ell =1 , 2, \ldots , n+1$ are the only bounded elements
in the kernel of the linear elliptic operator $L_0$.
A central point of the proof is to design a linear theory that allows us to solve system \equ{equ3332} by means of a contraction mapping argument. Thus
for a large number $R>0$
we shall construct a solution to an initial value problem of the form
\be \label{p110nuovo}
\phi_\tau  =
\Delta \phi + pU(y)^{p-1} \phi + h(y,\tau )  \inn B_{2R} \times (\tau_0, \infty )
\ee
$$
\phi(y,\tau_0) = e_0Z_0(y)  \inn B_{2R}.
$$
 Let
\be\label{ttau2}
\nu = 1+ {\sigma \over n-2}  \quad {\mbox {so that}} \quad\mu_0^{n-2 +\sigma} (t) \sim \tau^\nu,
\ee
thanks to \eqref{ttau}, and define
\be \label{minchia}
\|h\|_{\nu, a} := \sup_{\tau > \tau_0 }  \sup _{y\in B_{2R}}   \tau^{\nu} (1+ |y|^a ) \, |h( y ,\tau ) | .
\ee

The following is a central step in the proof.
\begin{prop} \label{prop000}
Let $\nu,a$ be given positive numbers with $0<a <1$, $\nu >0$. Then, for all sufficiently large $R>0$ and any  $h=h(y,\tau)$  with  $\|h\|_{\nu, 2+a} <+\infty$
that satisfies for all $j=1,\ldots, n+1$
\be
 \int_{B_{2R}} h(y ,\tau)\, Z_{j} (y) \, dy\ =\ 0  \foral \tau\in (\tau_0, \infty)
\label{ortio}\ee
there exist  $\phi = \phi [h]$  and $e_0 = e_0 [h]$ which solve Problem $\equ{p110nuovo}$. They define linear operators of $h$
that satisfy the estimates
\be
  |\phi(y,\tau) |  \ \lesssim   \  \tau^{-\nu}  \frac {R^{n+1-a}} { 1+ |y|^{n+1}} \,   \|h\|_{\nu, 2+a}.
\label{ctta1}\ee
and
\be
 | e_0[h]| \, \lesssim \,  \|h\|_{ \nu, 2+a}.
\label{cott2}\ee
\end{prop}
This result is a consequence of the key Proposition \ref{prop0} whose proof we postpone to Section \ref{seclineartheory}, so that not interrupting the main thread of the proof.
In order to apply this result  we need the  orthogonality conditions  \equ{gajardo3} satisfied. In what remains of this section we shall choose $\la,\xi$ as functions of a given $\phi$
in such a way that
\eqref{gajardo3} holds. This is  a  system of coupled, nonlocal, nonlinear  ordinary differential equations. Next we show that the system   is solvable and admits a solution $\la = \la [\phi] (t) $, $\xi = \xi [\phi ] (t)$, which satisfy the restrictions \eqref{rest}-\eqref{rest1}, for any given $\phi$ satisfying \eqref{rest13}.
We shall see that the solution $\xi = \xi [\phi]$ is unique, while
$\la = \la [\phi]$ has $k$ degrees of freedom in a small neighborhood of a specific solution. This degree of freedom is due to the presence of elements in the kernel of a linear operator in $\la$,  that also satisfy the decaying conditions \eqref{rest} and \eqref{rest1}. At last, we show the Lipschitz dependence of $\la = \la[\phi]$, $\xi = \xi[\phi]$ on $\phi$, which is a crucial property to ensure the existence of $\phi$, solution to \eqref{equ3332}, and thus to complete our construction.

\medskip
We use the notations $$\la (t) =
\begin{pmatrix} \la_1 (t) \\ \la_2 (t)  \\ \vdots  \\ \la_k (t)  \end{pmatrix}, \,
\dot \la (t) =
\begin{pmatrix} \dot \la_1 (t) \\ \dot \la_2 (t)  \\ \vdots  \\ \dot \la_k (t)  \end{pmatrix}, \,
\xi (t) =
\begin{pmatrix} \xi_1 (t) \\ \xi_2 (t)  \\ \vdots  \\ \xi_k (t)  \end{pmatrix},
\, \dot \xi (t) =
\begin{pmatrix} \dot \xi_1 (t) \\ \dot \xi_2 (t)  \\ \vdots  \\ \dot \xi_k (t)  \end{pmatrix},
\,
q =
\begin{pmatrix} q_1  \\ q_2   \\ \cdots  \\ q_k   \end{pmatrix}.
$$
Let us first describe \eqref{gajardo3} when $\ell = n+1$.

\medskip

\begin{lema}\label{calcolo1}
There exists a positive constant $\ve>0$ so that \eqref{gajardo3} with
$\ell = n+1$  is equivalent to
\begin{align}
\label{gajardo7}
\dot \la +  {1\over t} P^T {\rm diag}\, \left ( {1+ \bar \sigma_j \over n-4} \right )  P \la  &= \Pi_1 [\la, \xi , \dot \la , \dot \xi,  \phi ] (t)
\end{align}
where $P$ is the $k \times k$ invertible matrix defined in \eqref{D2Ib} and \eqref{gajardo6}, and the numbers $\bar \sigma_j >0$ are defined in \eqref{D2Ib}.
Moreover,
\begin{align}
\label{gajardo77}
\Pi_1 [\la, \xi , \dot \la , \dot \xi,  \phi ] (t)  &=
\mu_0^{n-3+\sigma } (t)  f(t) + \, t_0^{-\ve} \,   \Theta  [\dot \la , \dot \xi , \mu_0^{n-4} \la , \mu_0^{n-3} (\xi-q) , \mu_0^{n-3+\sigma} \phi ] (t)
\end{align}
where
$f = f(t)$ is an explicit vector function,   smooth and bounded for $t \in [t_0 , \infty)$, and
$\Theta [ \cdots ] (t)$ is a smooth and bounded function of $t \in [t_0 , \infty)$, and it has Lipschtz dependence with respect to its parameters, in the sense that, for any $t\geq t_0$,
\begin{equation}\label{gajardo111}
\left| \Theta [\dot{ \la} _1  ] (t)  - \Theta [\dot{\la}_2  ] (t) \right|   \lesssim t_0^{-\ve}  \,   \left|\dot{\la}_1 (t) - \dot{\la}_2  (t) \right|
\end{equation}
\begin{equation}\label{gajardo112}
\left| \Theta [ \dot{\xi}_1 ] (t)  - \Theta [ \dot{\xi}_2 ] (t) \right|  \lesssim t_0^{-\ve}  \,   \left|  \dot{\xi}_1 (t) - \dot{\xi}_2  (t)   \right| ,
\end{equation}
and also
\begin{equation}\label{gajardo113}
\Bigl| \Theta [ \mu_0^{n-4} \la_1   ] (t)  - \Theta [\mu_0^{n-4} \la_2 ] (t) \Bigl|  \lesssim t_0^{-\ve}  \,   \Bigl| {\la}_1 (t) - {\la}_2  (t)   \Bigl|
\end{equation}
\begin{equation}\label{gajardo114}
\Bigl| \Theta [\mu_0^{n-3}(  {\xi}_1 -q ) ] (t)  - \Theta [\mu_0^{n-3} (\xi_2 - q)  ] (t) \Bigl|   \lesssim t_0^{-\ve}  \,   \Bigl|{\xi}_1 (t) - {\xi}_2  (t) \Bigl|,
\end{equation}
and
\begin{equation}
\label{gajardo11}
\Bigl| \Theta [\mu_0^{n-3+\sigma } \phi_1] (t)  - \Theta [\mu_0^{n-3+\sigma } \phi_2 ] (t) \Bigl|   \lesssim t_0^{-\ve}  \,   \| \phi_1 - \phi_2 \|_{\sigma , a}.
\end{equation}
\end{lema}

\medskip
\begin{proof}
Let $\sigma $ be the positive number fixed in \eqref{defsigma}.
Let $\phi$ satisfy \eqref{rest13}.
Fix $j \in \{ 1, \ldots , k\}$. We want to compute
$$
\int_{B_{2R} } H_j [\phi , \la , \xi , \dot \la , \dot \xi] (y,t (\tau ) )  Z_{n+1} (y)  dy,
$$
where $H_j $ is given by \eqref{defHj}. We start with $\int_{B_{2R}} \mu_{0j}^{\frac{n+2}2}  S_{\mu,\xi,j}  (\xi_j +\mu_{0j} y,t)  Z_{n+1} (y) \, dy$. We write
\begin{align*}
\mu_{0j}^{\frac{n+2}2}  & S_{\mu,\xi,j}  (\xi_j +\mu_{0j} y,t)  = \left( {\mu_{0j} \over \mu_j } \right)^{n+2 \over 2} \left[ \mu_{0j} S_1 (z , t) +
\lambda_j b_j S_2 (z,t) + \mu_j S_3   (z,t)  \right]_{z= \xi_j + \mu_{j} y}\\
&+ \left( {\mu_{0j} \over \mu_j } \right)^{n+2 \over 2}\mu_{0j}  \left[ S_1 ( \xi_j + \mu_{0j} y , t)  - S_1 (\xi_j + \mu_{j} y, t) \right] \\
&+ \left( {\mu_{0j} \over \mu_j } \right)^{n+2 \over 2} \la_j b_j   \left[ S_2 ( \xi_j + \mu_{0j} y , t)  - S_2 (\xi_j + \mu_{j} y, t) \right] \\
& +\left( {\mu_{0j} \over \mu_j } \right)^{n+2 \over 2} \mu_{j} \left[  S_3 ( \xi_j + \mu_{0j} y , t)  - S_3 (\xi_j + \mu_{j} y, t) \right] \\
\end{align*}
where
$$
S_1 (z) = \dot \la _j\, Z_{n+1}({z-\xi_j \over \mu_j} ) \, - \,   \mu_0^{n-4}  pU({z-\xi_j \over \mu_j})^{p-1} \,   \sum_{i=1}^k M_{ij}\la_i\,
$$
$$
S_2 (z ) = \dot \mu_0 Z_{n+1}  ({z-\xi_j \over \mu_j} )+ p U^{p-1} ({z-\xi_j \over \mu_j})  \mu_0^{n-3} \left( - b_j^{n-2} H(q_j , q_j ) + \sum_{i\not= j} (b_i b_j)^{n-2 \over 2} G(q_i , q_j ) \right)
$$
and
$$
S_3 (z) = \dot \xi_j \cdot \nn U({z-\xi_j \over \mu_j})  +  pU({z-\xi_j \over \mu_j})^{p-1} \big [-\mu_j^{n-2} \nn_x H(q_j ,q_j)   \,  + \,  \sum_{i\ne j}  \mu_j^{\frac {n-2}2} \mu_i^{\frac {n-2}2}  \nn_x G( q_j,q_i )\, \big ] \cdot  {z-\xi_j \over \mu_j}.
$$
A direct computation gives
$$
\int_{B_{2R}} S_1 (\xi_j + \mu_j y ) Z_{n+1} (y) \, dy =  c_1 (1+ O(R^{4-n} ) ) \dot \la_j + c_2
(1+ O(R^{-2}) )  \mu_0^{n-4} \sum_{i\not=j} M_{ij} \la_i ,
$$
with $c_1$ and $c_2$ defined in \eqref{defc1c2},
$$
\int_{B_{2R}} S_2 (\xi_j + \mu_j y ) Z_{n+1} (y) \, dy = O(R^{-2} + R^{-4+n} ) \mu_0^{n-3} (t)
$$
since by construction $\int_{\R^n } S_2 (\xi_j + \mu_j y) Z_{n+1} (y) \, dy = 0$, and
$$
\int_{B_{2R}} S_3 (\xi_j + \mu_j y ) Z_{n+1} (y) \, dy = 0
$$
by symmetry.  Since ${\mu_{0j} \over \mu_j} = ( 1+ {\la_j \over \mu_{0j}} )^{-1}$, we get, for any $\ell = 1 , 2 , 3 $
\begin{align*}
\int_{B_{2R} } & \left[ S_\ell ( \xi_j + \mu_{0j} y , t)  - S_\ell (\xi_j + \mu_{j} y, t) \right]  Z_{n+1} (y) \, dy =g(t, {\la \over \mu_0} )  \dot \la_j + g(t, {\la \over \mu_0} )  \dot \xi \\
&+ g(t, {\la \over \mu_0} )  \sum_{i} \mu_0^{n-4} \la_i + \mu_0^{n-3+\sigma} f(t)
\end{align*}
where $g$ denotes a smooth and bounded function, such that $g(\cdot , s ) \sim s$ as $s \to 0$, and $f$ is smooth and bounded.
Thus we conclude that
\begin{align}\label{vec1}
c \left( {\mu_{j} \over \mu_{0j} } \right)^{n+2 \over 2}  \mu_{0j}^{-1} \int_{B_{2R}} & \mu_{0j}^{\frac{n+2}2}  S_{\mu,\xi,j}  (\xi_j +\mu_{0j} y,t)  Z_{n+1} (y) \, dy =  \left[ \dot \la_j + {1\over t} \left( P^T \hbox{diag}\, \big ({1+ \sigma_\ell \over n-2} \big ) P \la \right)_j \right] \nonumber \\
& + t_0^{-\ve} g(t, {\la \over \mu_0} ) ( \dot \la +\dot \xi )+ \mu_0^{n-4}  t_0^{-\ve} g(t, {\la \over \mu_0} ) \la,
\end{align}
where $c$ is an explicit positive number,
for functions $g$ which are smooth in its argumets, bounded and $g(\cdot, s ) \sim s$, as $s\to 0$.

\medskip
Nest we consider $ p \mu_{0j}^{\frac{n-2}2} (1+ {\lambda_j \over \mu_{0j} } )^{-2}\int_{B_{2R} }  U^{p-1} ( {\mu_{0j} \over \mu_j}  y) \psi(\xi_j +\mu_{0j} y,t)  Z_{n+1} (y) \, dy.
$ The principal part to describe is $I:= \int_{B_{2R} }  U^{p-1} (  y) \psi(\xi_j +\mu_{0j} y,t)  Z_{n+1} (y) \, dy.$ Recall now that $\psi = \psi [\la , \xi , \dot \la , \dot \xi , \phi] (y,t)$. Thus, we write
\begin{align*}
I&=\psi [0,q,0,0,0] (q_j,t) \,  \int_{B_{2R} }  U^{p-1} (  y)   Z_{n+1} (y) \, dy\\
&+  \int_{B_{2R} }  U^{p-1} (  y) Z_{n+1} (y)   \left( \psi [0,q,0,0,0] (\xi_j + \mu_{0j} y,t ) - \psi [0,q,0,0,0] (q_j,t) \right) \, dy\\
&+  \int_{B_{2R} }  U^{p-1} (  y) Z_{n+1} (y)   \left( \psi [\la,\xi,\dot \la ,\dot \xi ,\phi]  - \psi [0,q,0,0,0] \right) (\xi_j + \mu_{0j} y ,t)  \, dy \\
&= I_1 + I_2 + I_3.
\end{align*}
Thanks to \eqref{leuco1}, $I_1 = \mu_0^{{n-2 \over 2} +\sigma} \, t_0^{-\ve} \,  f(t)$, for some smooth and bounded $f$. Applying the mean value theorem to $I_2$ and using \eqref{leuco11}, we get that $I_2 = \mu_0^{{n-2 \over 2}+\sigma} \, t_0^{-\ve} \, g(t, {\la \over \mu_0} , \xi - q )$,
where $g$ is a smooth function with $g(\cdot , s, \cdot ) \sim s$, and $g(\cdot , \cdot , s) \sim s$, as $s \to 0$. Again the mean value Theorem gives that, for some $s \in (0,1)$,
\begin{align*}
I_3& =
   \int_{B_{2R} }  U^{p-1} (  y) \,  Z_{n+1} (y) \, \Biggl[ \pp_\la \psi [0,q,0,0,0] [s \la ]  (\xi_j + \mu_{0j} y,t)  \\
&+  \pp_\xi \psi [0,q,0,0,0][s(\xi_j - q_j) ]  (\xi_j + \mu_{0j} y ,t) +
 \pp_{\dot \la }\psi [0,q,0,0,0] [s \dot \la ] (\xi_j + \mu_{0j} y ,t)    \\
&+ \pp_{\dot \xi} \psi [0,q,0,0,0][ s \dot \xi ]  ((\xi_j + \mu_{0j} y ,t)
+  \pp_{\phi} \psi [0,q,0,0,0][s \phi]  ((\xi_j + \mu_{0j} y ,t) \Biggl] \, dy.
\end{align*}
Using \eqref{leuco2}--\eqref{leuco3}, we can describe the above function as sum of terms of the form
$$
\mu_0^{-{n-4 \over 2} +\sigma} \, t_0^{-\ve } f(t) (\dot \la + \dot \xi ) F (\la , \xi ,\dot \la , \dot \xi , \phi) (t) , \quad \mu_0^{n-4 \over 2} \, t_0^{-\ve} \,  f(t) \, (\la + \xi )
F (\la , \xi ,\dot \la , \dot \xi , \phi) (t) ,
$$
where $f$ is smooth and bounded, while $F$ denotes a non local, non linear smooth
operator in its parameters, with $F(0,q,0,0,0) (t) $ bounded.

\medskip
Let us consider now the terms $B_j [\phi_j ]$ and $B_j^0 [\phi_j]$ defined respectively by \eqref{gajardo1} and \eqref{gajardo2}. We have that
$$
\int_{B_{2R}} B_j [\phi_j ] (y,t) \, Z_{n+1} (y) \, dy = t_0^{-\ve } \Bigl[ \mu_0^{n-3 +\sigma} (t) \ell [\phi] (t) + \dot \xi_j \ell [\phi] (t)\Bigl]
$$
and
$$
\int_{B_{2R}} B_j^0 [\phi_j ] (y,t) \, Z_{n+1} (y) \, dy =t_0^{-\ve} \mu_0^{n-2 +\sigma} (t)
\, g ({\la  \over \mu_0} ) \, \ell [\phi] (t)
$$
where $\ell [\phi] (t)$ is a smooth and bounded function in $t$, and is  a linear, and non local smooth operator in $\phi$, while $g (s)$ is a smooth function with $g(s) \sim s$, as $s\to 0$.

Collecting the above information, we get the validity of \eqref{gajardo7}.

\end{proof}

\medskip
\medskip
Next we shall compute, for any $j=1, \ldots , k$,
$$
\int_{B_{2R} } H_j [\phi , \la , \xi , \dot \la , \dot \xi] (y,t (\tau ) )  Z_{\ell} (y)  dy,
$$
this time for $\ell =1 , \ldots , n$. This is the content of next Lemma, whose proof we leave to the interested reader, being similar to the one performed in Lemma \ref{calcolo1}.

\medskip

\begin{lema}\label{calcolo2}
Let us now fix $j=1, \ldots ,k$. There exists a positive constant $\ve>0$, such that the relation \eqref{gajardo3} for $\ell = 1, \ldots , n$ is equivalent to
\begin{align}
\label{gajardo8}
\dot \xi_j &=\Pi_{2,j} [\la, \xi , \dot \la , \dot \xi,  \phi ] (t)
\end{align}
\begin{align}
\label{gajardo88}
\Pi_{2,j} [\la, \xi , \dot \la , \dot \xi,  \phi ] (t)  &=
\mu_0^{n-2} (t)  \,   c \, \big [b_j^{n-2} \nn_x H(q_j ,q_j)   \,  - \,  \sum_{i\ne j}  b_j^{\frac {n-2}2} b_i^{\frac {n-2}2}  \nn_x G( q_j,q_i )\, \big ]   \\
&+\mu_0^{n-2+\sigma } (t)  f_j(t) + \, t_0^{-\ve} \,   \Theta  [\dot \la , \dot \xi , \mu_0^{n-4} \la , \mu_0^{n-3} (\xi-q) , \mu_0^{n-3+\sigma} \phi ] (t)  \nonumber
\end{align}
where $c= {p \int_{\R^n} U^{p-1} {\partial U \over \partial y_1} y_1 \, dy \over \int_{\R^n} \left(   {\partial U \over \partial y_1}  \right)^2 \, dy }$ and $b_j$ are the
(positive) numbers defined in \eqref{ee}.  The function
$f_j= f_j (t)$ is an explicit $n$ dimensional vector function,  it is  smooth and bounded for $t \in [t_0 , \infty)$.
Furthermore, $\Theta$ has the same properties
as described in Lemma \ref{calcolo1}.
\end{lema}

\bigskip
In summary: given $\sigma $ the positive number fixed in Proposition \ref{probesterno} and
 $\phi$ satisfying \eqref{rest13}, we have proved so far that solving
$$
\int_{B_{2R} } H_j [\phi , \la , \xi , \dot \la , \dot \xi] (y,t (\tau ) )  Z_{\ell} (y)  dy = 0, \quad {\mbox {for}} \quad t \geq t_0
$$
for all $j=1, \ldots , k$, and $\ell = 1, \ldots , n, n+1$, is equivalent to solve in $\la $ and $\xi$ the system of ODEs
\begin{equation} \label{gajardoo}
\left\{ \begin{array}{r}
\dot \la +  {1\over t} P^T {\rm diag}\, \left ( {1+ \bar \sigma_j \over n-4} \right ) P \la  = \Pi_1 [\la, \xi , \dot \la , \dot \xi,  \phi ] (t)  \\
\\
\dot \xi_j =\Pi_{2,j} [\la, \xi , \dot \la , \dot \xi,  \phi ] (t), \quad j=1, \ldots , k
\end{array} \right.
\end{equation}
where $\Pi_1$, and $\Pi_{2,j} $ are defined respectively in \eqref{gajardo77} and \eqref{gajardo88}.

\medskip
We next prove that Problem \eqref{gajardoo} is solvable for
functions  $\la $ and $\xi $  satisfying \eqref{rest} and \eqref{rest1}.

\begin{prop}\label{sabato}
Let $\phi$ be such that \eqref{rest13} is satisfied. Then there exists a solution
$\la = \la [\phi ] (t) $, $\xi = \xi[\phi ] (t ) $ to the nonlinear system of ODEs \eqref{gajardoo}, which satisfies the bounds \eqref{rest}-\eqref{rest1}. Furthermore, for $t \in (t_0 , \infty)$,
\begin{equation}\label{cacca1}
\mu_0^{-(1+\sigma) }  (t) \Bigl| \la [\phi_1 ] (t) -\la [\phi_2] (t) \Bigl| \lesssim
t_0^{-\ve } \, \| \phi_1 -\phi_2 \|_{ \sigma , a}
\end{equation}
and
\begin{equation}\label{cacca2}
\mu_0^{-1}  (t) \Bigl| \xi [\phi_1 ] (t) -\xi [\phi_2] (t) \Bigl| \lesssim t_0^{-\ve } \,
\| \phi_1 -\phi_2 \|_{\sigma , a}
\end{equation}

\end{prop}

\begin{proof}
Let  $h$ be a vector function with $h: [t_0 , \infty) \to \R^k $ with  $\| h \|_{n-3+\sigma} $   bounded. See \eqref{normfinal} for the definition of the norm.  We select the  solution of the linear system of ODEs associated to \eqref{gajardo7}
\be \label{syst1}
\dot \la +  {1\over t} P^T {\rm diag}\, \left ( {1+ \bar \sigma_j \over n-4} \right ) P \la  = h(t)
\ee
explicitly given by
\be\label{solsyst1}
\la (t) = P^T \nu (t) , \quad   \nu (t) =
\begin{pmatrix} \nu_1 (t) \\ \vdots  \\  \nu_k (t)  \end{pmatrix}, \quad  \nu_j (t) =  t^{-{1+ \bar \sigma_j \over n-4}}  \left[ d_j +
  \int_{t_0}^t s^{1+ \bar \sigma_j \over n-4}
\, \left( P h \right)_j (s) \, ds \right] ,
\ee
where the  constants $d_j$, $j=1, \ldots , k$ are arbitrary. Observe that
$$
\| t^{1+\sigma \over n-4} \la (t)  \|_{L^\infty (t_0 , \infty)}  \lesssim t_0^{-{\bar \sigma -\sigma \over n-4}} d
+\| h \|_{n-3+\sigma}
\quad
{\mbox {and}} \quad
\| \dot \la \|_{n-3+\sigma} \lesssim t_0^{-{\bar \sigma_j -\sigma \over n-4}} d
+\| h \|_{n-3+\sigma},
$$
where $0\leq d= \max_{i=1, \ldots , k} |d_i|$.

Let $\Lambda (t) = \dot \la (t)$. Thus \eqref{solsyst1} defines a linear operator ${\mathcal L}_1 : h \to \Lambda$, which associates to any $h$ with $\| h \|_{n-3+\sigma}$ -bounded the solutions   \eqref{solsyst1} to \eqref{syst1}, that now takes the form
$$
\Lambda +  {1\over t} P^T {\rm diag}\, \left ( {1+ \bar \sigma_j \over n-4} \right ) P \int_t^\infty \Lambda (s) \, ds  = h(t).
$$
 This operator ${\mathcal L}_1$ is continuous between the Banach spaces $\left(  L^\infty (t_0  \infty) \right)^k$ equipped with the $\| \cdot \|_{n-3+\sigma}$-topology.

For any $h$, this time $h :[t_0 , \infty) \to \R^n$,  with $\| h \|_{n-3+\sigma} $   bounded, we now select the  solution  of the linear system of ODEs associated to \eqref{gajardo8}, for $j\in \{ 1, \ldots , k\}$,
\be\label{syst2}
\dot \xi_j =
\mu_0^{n-2} (t)  \,   c \, \big [b_j^{n-2} \nn_x H(q_j ,q_j)   \,  - \,  \sum_{i\ne j}  b_j^{\frac {n-2}2} b_i^{\frac {n-2}2}  \nn_x G( q_j,q_i )\, \big ]  +h(t)
\ee
explicitly given by
\be\label{solsyst2}
\xi_j (t) =  \xi_j^0 (t) +  \int_{t}^\infty h(s) \, ds,
\ee
where
$$  \xi_j^0 (t) = q_j + c \, \big [b_j^{n-2} \nn_x H(q_j ,q_j)   \,  - \,  \sum_{i\ne j}  b_j^{\frac {n-2}2} b_i^{\frac {n-2}2}  \nn_x G( q_j,q_i )\, \big ] \int_t^\infty \mu_0^{n-2} (s)  \,   ds.
$$
Observe that
$$
| \xi_j (t) -q_j   | \lesssim  t^{-{1+\sigma \over n-4}}
+t^{-{1+\sigma \over n-4}} \, \| h \|_{n-3+\sigma}
\quad
{\mbox {and}} \quad
\| \dot \xi_j -\dot \xi^0_j \|_{\sigma} \lesssim
\| h \|_{n-3+\sigma}.
$$
Let $\Xi (t) = \dot \xi (t) -\dot \xi^0 (t) $, which is a $n\, k $-dimensional vector function. Thus \eqref{solsyst2} defines a linear operator ${\mathcal L}_2 : h \to \Xi$, which associates to any $n\times k$-dimensional vector function $h$ with $\| h \|_{n-3+\sigma}$-bounded  the solutions   \eqref{solsyst2} to \eqref{syst2}.  This operator is continuous in the $\| \cdot \|_{n-3+\sigma}$-topology.

\medskip
Having introduced the linear operators ${\mathcal L}_i$, $i=1,2$, we observe that
$\la $, $\xi$ is a solution to \eqref{gajardo7}-\eqref{gajardo8} if $\Lambda =\dot \la$, $\Xi = \dot \xi - \dot \xi^0$ is a fixed point for the operator
\be\label{h2o}
\left( \Lambda , \Xi \right) = {\mathcal A} \left(\Lambda , \Xi  \right)
\ee
where
$$
 {\mathcal A} \left(\Lambda , \Xi  \right) := \left({\mathcal L}_1 (\hat \Pi_1
[ \Lambda ,  \Xi,  \phi ]  ) , {\mathcal L}_2 (\hat \Pi_2 [ \Lambda ,  \Xi,  \phi ]  ) \right)
=: \left(\bar A_1 (\Lambda , \Xi ) , \bar A_2 (\Lambda , \Xi)  \right)
$$
with
$$
\hat \Pi_1 (\Lambda , \Xi , \phi  ) = \Pi_1
[\int_t^\infty \Lambda , q+ \int_t^\infty \Xi  ,  \Lambda ,  \Xi,  \phi ]  )
, \quad
\hat \Pi_2 (\Lambda , \Xi , \phi ) = \Pi_2
[\int_t^\infty \Lambda , q+ \int_t^\infty \Xi  ,  \Lambda ,  \Xi,  \phi ]  )
$$
and
 $\ \Pi_1$, $ \Pi_2$ defined respectively in \eqref{gajardo77} and \eqref{gajardo88}.
Let
$$
K = \max \{ \|  f \|_{n-3+\sigma} ,   \| f_1 \|_{n-3+\sigma} , \ldots ,  \| f_k \|_{n-3+\sigma }\}
$$
where the functions $f$, $f_1, \ldots , f_k$ are the ones in \eqref{gajardo77} and \eqref{gajardo88}.
We show that ${\mathcal A}$ \eqref{h2o}  has a fixed point $(\Lambda , \Xi)$
in the set
$$
{\mathcal B} = \{ (\Lambda , \Xi ) \in L^{\infty} (t_0 , \infty ) \times L^{\infty} (t_0 , \infty )\, : \, \| \Lambda \|_{n-3+\sigma} + \| \Xi \|_{n-3+\sigma }\leq  cK \},
$$
for some $c>0$.
Indeed, we observe directly from \eqref{gajardo77} that
\begin{align*}
\left| t^{n-3 +\sigma \over n-4} \bar A_1 (\Lambda , \Xi ) \right| & \lesssim
t_0^{-{\bar \sigma - \sigma \over n-4} } d + \| \phi \|_{n-2+ \sigma , a} +K
+ t_0^{-\ve} \| \Lambda \|_{n-3+\sigma} + t_0^{-\ve} \| \Xi \|_{n-3+\sigma} ,
\end{align*}
and from \eqref{gajardo88}
\begin{align*}
\left| t^{n-3 +\sigma \over n-4} \bar A_2 (\Lambda , \Xi ) \right| & \lesssim
 \| \phi \|_{n-2+ \sigma , a} + K
+ t_0^{-\ve} \| \Lambda \|_{n-3+ \sigma} + t_0^{-\ve} \| \Xi \|_{n-3+\sigma} .
\end{align*}
Thus, for $d$ with $t_0^{-{\bar \sigma - \sigma \over n-4} } d <K$, and choosing possibly $c$ large, we have that ${\mathcal A} \left( {\mathcal B} \right) \subset {\mathcal B}
$. Let us now check the Lipschitz condition for ${\mathcal A}$. For instance, we have
\begin{align*}
& t^{n-3 +\sigma \over n-4} \left| \bar A_1 (\Lambda_1 , \Xi ) -
\bar A_1 (\Lambda_2 , \Xi ) \right| = t^{n-3 +\sigma \over n-4} \left| {\mathcal L}_1 (\hat \Pi_1
[ \Lambda_1 ,  \Xi,  \phi ]  -  \hat \Pi_1
[ \Lambda_2 ,  \Xi,  \phi ]  )  \right| \\
&\leq  t^{n-3 +\sigma \over n-4}  t_0^{-\ve} \left| {\mathcal L}_1 (  \Theta_2 (\Lambda_1 , \Xi ) - \Theta_2 (\Lambda_2 , \Xi ) ) \right| \\
& +  t^{n-3 +\sigma \over n-4}   t_0^{-\ve}[
\left| {\mathcal L}_1 (  \mu_0^{n-4}\Theta_3 (\int_t^\infty \Lambda_1 , \Xi ) - \Theta_3 (\int_t^\infty \Lambda_2 , \Xi ) )
  \right|\\
& \leq t_0^{-\ve } \| \Lambda_1 - \Lambda_2 \|_{n-3+\sigma},
\end{align*}
as consequence of \eqref{gajardo111} and \eqref{gajardo113}.
Arguing with the same logic, one gets a similar estimation for $ \bar A_1 (\Lambda_1 , \Xi ) -
\bar A_1 (\Lambda_2 , \Xi )$.
Making use of \eqref{gajardo111}--\eqref{gajardo114}, we get
$$
\| {\mathcal A} (\Lambda_1 , \Xi_1 ) - {\mathcal A} (\Lambda_2 , \Xi_2 ) \|_{n-3+\sigma} \leq t_0^{-\ve} \|\Lambda_1 - \Lambda_2 \|_{n-3+\sigma}.
$$
Since $t_0^{-\ve}$, ${\mathcal A}$ is a contraction, and a direct application of Contraction Mapping Theorem, we get the existence of a solution for System \eqref{gajardoo}, in the class of functions $\la $ and $\xi$ satisfying \eqref{rest} and \eqref{rest1}.

\medskip
Let us now fix $\phi_1$ and $\phi_2$ in the class of functions satisfying \eqref{rest13}.
The functions $\bar \la  = \lambda [\phi_1 ] - \lambda [\phi_2]$, $\bar \xi =
\xi [\phi_1 ] -\xi [\phi_2]$ solve the system of ODEs
$$
\dot \la +  {1\over t} P^T {\rm diag}\, \left ( {1+ \bar \sigma_j \over n-4} \right )  P \la  = \bar \Pi_1  (t), \quad
\dot \xi_i =\bar \Pi_{2,i}  (t), \quad i=1, \ldots , k
$$
where
\begin{align*}
\left( \bar \Pi_1 (t) \right)_j &= c  p \mu_j^{n-2 \over 2} \mu_{0j}^{-1} \int_{B_{2R}}
   U^{p-1} ( {\mu_{0j} \over \mu_j}  y) \left[  \psi [\phi_1]  -
 \psi [\phi_2] \right] (\xi_j +\mu_{0j} y,t) Z_{n+1} (y) \, dy\\
 & + c \left( {\mu_j \over \mu_{0j} } \right)^{n+2 \over 2} \, \mu_{0j}^{-1} \int_{B_{2R}}  \left[ B_j[(\phi_1)_j] -  B_j[(\phi_2)_j] \right] Z_{n+1} (y) \, dy \\
 & + c \left( {\mu_j \over \mu_{0j} } \right)^{n+2 \over 2} \, \mu_{0j}^{-1} \int_{B_{2R}}  \left[ B_j^0 [(\phi_1)_j] -  B_j^0 [(\phi_2)_j] \right] Z_{n+1} (y) \, dy
\end{align*}
and
\begin{align*}
\left( \bar \Pi_{2i} (t) \right)_j &= c  p \mu_j^{n-2 \over 2} \mu_{0j}^{-1} \int_{B_{2R}}
   U^{p-1} ( {\mu_{0j} \over \mu_j}  y) \left[  \psi [\phi_1]  -
 \psi [\phi_2] \right] (\xi_j +\mu_{0j} y,t) {\partial U \over \partial y_i} (y) \, dy\\
 & + c \left( {\mu_j \over \mu_{0j} } \right)^{n+2 \over 2} \, \mu_{0j}^{-1} \int_{B_{2R}}  \left[ B_j[(\phi_1)_j] -  B_j[(\phi_2)_j] \right] {\partial U \over \partial y_i}  (y) \, dy \\
 & + c \left( {\mu_j \over \mu_{0j} } \right)^{n+2 \over 2} \, \mu_{0j}^{-1} \int_{B_{2R}}  \left[ B_j^0 [(\phi_1)_j] -  B_j^0 [(\phi_2)_j] \right] {\partial U \over \partial y_i} (y) \, dy
\end{align*}
Thus, the validity of \eqref{cacca1} and \eqref{cacca2} follows from  \eqref{gajardo11}.

\end{proof}

\begin{remark}\label{rmk2}
Recall that the function $ \psi = \bar \Psi [\psi_0]$ solution to Problem \eqref{equpsi1}
depends smoothly on the initial condition $\psi_0$, provided $\psi_0$ belongs to a small
neighborhood of $0$ in the Banach space $L^\infty (\Omega)$ equipped with the $C^1$ norm, as observed in Remark \ref{rmk1}. This fact implies that also $\la = \la[\psi_0] $,
$\xi = \xi [\psi_0]$ solutions to \eqref{gajardoo} depends smoothly on $\psi_0$. A closer look at the definitions of $\la = \la[\psi_0] $,
$\xi = \xi [\psi_0]$ gives that
$$
\| \la [\psi_0^{(1)} ] - \la [\psi_0^{(2)} ] \|_{1+\sigma} \lesssim
\| \psi_0^{(1)}  - \psi_0^{(1)}  \|_{L^\infty (\Omega) }
$$
and
$$
\| \xi [\psi_0^{(1)} ] - \xi [\psi_0^{(2)} ] \|_{1+\sigma} \lesssim
\| \psi_0^{(1)}  - \psi_0^{(1)}  \|_{L^\infty (\Omega) }.
$$

\end{remark}

\setcounter{equation}{0}
\section{Final argument: solving \eqref{equ3332}} \label{final}

Let $\psi = \Psi [ \la , \xi , \dot \la , \dot \xi , \phi ] $ be the solution to the outer Problem \eqref{equpsi1} (or equivalently \eqref{equpsi}) as described in Propositions \ref{probesterno} and \ref{probesterno1}, and let $\la = \la [\phi]$, and $\xi = \xi[\phi]$ be the solution to \eqref{gajardoo} whose existence and properties are established in Proposition \ref{sabato}. Then
$u_{\mu , \xi } + \tilde \phi$, with $\tilde \phi$ as in \eqref{deftildephi}-\eqref{martin1}, is the expected solution to \eqref{equ1}, as described in Section \ref{sec3new},  if there exists  $\phi_j$ solution to Problem \eqref{equ3332} in the class of functions
with $\| \phi \|_{\nu , a}$ (see \eqref{minchia}), or equivalently $\| \phi \|_{ n-2+ \sigma , a}$ (see \eqref{starstar}), bounded.

Proposition \ref{prop000} states the existence of a linear operator ${\mathcal T}$ which to any function $h(y, \tau )$, with $\| h \|_{\nu , a} $-bounded, associate the solution to \eqref{p110nuovo}. Furthermore, it states that this operator ${\mathcal T}$ is continuous between $L^\infty$ spaces equipped with the topologies described by \eqref{cta1}. Thus, the existence and properties of $\phi_j$, solution to \eqref{equ3332}  are reduced to find a fixed point for
$$
\phi = (\phi_1 , \ldots , \phi_k ) = {\mathcal A}_3 (\phi ) := \left( {\mathcal T}  (H_1 [ \la , \xi , \dot \la , \dot \xi , \phi ] ), \ldots , {\mathcal T}  (H_k [ \la , \xi , \dot \la , \dot \xi , \phi ] ) \right)
$$
in a proper set of functions.
We claim that ${\mathcal A}_3$ admits indeed a fixed point in the set of functions
$
\| \phi \|_{ n-2+ \sigma , a} < c t_0^{-\ve}.
$
To prove this, we claim that, for each $j=1, \ldots , k$,
\begin{align}\label{ultimo1}
\left| H_j [  \la , \xi , \dot \la , \dot \xi , \phi ] (t,y) \right| \lesssim t_0^{-\ve } \,  {\mu_0^{n-2+\sigma} \over 1+ |y|^{2+a}}
\end{align}
for some $a\in (0,1)$, and
\begin{equation}
\label{ultimo2}
\| H_j [\phi^{(1)}] - H_j [\phi^{(2)}] \|_{n-2+ \sigma , 2+a} \leq t_0^{-\ve } \| \phi^{(1)} - \phi^{(2)} \|_{n-2+\sigma , a}.
\end{equation}
Estimate \eqref{ultimo1} is consequence of Lemma \ref{lemaerror}, \eqref{leuco1}, \eqref{gajardo1} and \eqref{gajardo2}.

To get  \eqref{ultimo2} we observe that
\be\label{ultimo3}
\mu_{0j}^{\frac{n+2}2}  \left| S_{\mu_1 ,\xi_1 ,j}  (\xi_{j,1} +\mu_{0j} y,t)
- S_{\mu_2 ,\xi_2 ,j}  (\xi_{j,2} +\mu_{0j} y,t) \right| \lesssim t_0^{-\ve} {\mu_0^{n-2+\sigma} (t) \over 1+ |y|^{2+a} } \, \| \phi^{(1)} - \phi^{(2)} \|_{n-2+\sigma , a}
\ee
as consequence of \eqref{we2}, \eqref{cacca1} and \eqref{cacca2}, where we have use the notation
$$
\mu_i = \mu [\phi^{(i)} ] , \quad \xi_i = \xi [\phi^{(i)} ] , \quad \xi_{j,i} = \xi_j [\phi^{(i)}  ],  \quad i=1, 2.
$$
Furthermore, we have
\begin{align} \label{ultimo4}
p \mu_{0j}^{\frac{n+2}2} & \left|  \mu_{j, 1}^2  U^{p-1} ( {\mu_{0j} \over \mu_{j,1}}  y) \psi [\phi^{(1)} ] (\xi_{j,1} +\mu_{0j} y,t) -
 \mu_{j, 2}^2  U^{p-1} ( {\mu_{0j} \over \mu_{j,2}}  y) \psi [\phi^{(2)} ] (\xi_{j,2} +\mu_{0j} y,t) \right| \nonumber \\
& \lesssim t_0^{-\ve} {\mu_0^{n-2+\sigma} (t) \over 1+ |y|^{2+a} } \, \| \phi^{(1)} - \phi^{(2)} \|_{\sigma , a}
\end{align}
as consequence of \eqref{leuco2}--\eqref{leuco3}, where we use the notation
$$
\mu_{j,i} = \mu_j [\phi^{(i)}], \quad \psi [\phi^{(i)} ] =
\Psi [\la_i , \xi_i , \dot \la_i , \dot \xi_i , \phi^{(i)}] , \quad i=1,2.
$$
Finally, directly from the definitions \eqref{gajardo1} and \eqref{gajardo2} respectively, we have
\be \label{ultimo5}
\left| B_j [\phi^{(1)} ] -B_j [\phi^{(2)}] \right| \lesssim t_0^{-\ve} {\mu_0^{n-2+\sigma} (t) \over 1+ |y|^{2+a} } \, \| \phi^{(1)} - \phi^{(2)} \|_{n-2+\sigma , a},
\ee
and
\be \label{ultimo6}
\left| B_j^0 [\phi^{(1)} ] -B_j^0  [\phi^{(2)}] \right| \lesssim t_0^{-\ve} {\mu_0^{n-2+\sigma} (t) \over 1+ |y|^{2+a} } \, \| \phi^{(1)} - \phi^{(2)} \|_{n-2+\sigma , a}.
\ee
Estimates \eqref{ultimo3}--\eqref{ultimo6} give \eqref{ultimo2}.

Now using \eqref{ultimo1} and \eqref{ultimo2}, one gets that ${\mathcal A}_3$ has a fixed point $\phi$, with $
\| \phi \|_{ n-2+\sigma , a} < c t_0^{-\ve},
$
for some positive constant $c$, provided the number $\rho$ in \eqref{defR} is chosen small enough. Authomatically, this defines $e_0 = (e_{01} , \ldots , e_{0k})$,  where $e_{0,j}$ are the constants in the initial conditions in Problem \eqref{equ3332}. Each constant $e_{0j}$ is a linear functions of $H_j$, $e_{0j} = e_{0j} [H_j [\phi , \la , \xi , \dot \la , \dot \xi] ]$.  One has
that $|e_{0j} | \lesssim t_0^{-\ve}$.
This concludes the proof of the existence of the solution to Problem \eqref{P}, as predicted by Theorem \ref{teo1}. The proof is concluded. \qed

\bigskip
\subsection*{Proof of Corollary \ref{corol1}}
We begin by some comments that carry interest in their own sake and that are needed for the proof.

\medskip
Let us  observe that the construction of $\phi = (\phi_1 , \ldots , \phi_k )$, and
$e_0 = (e_{01} , \ldots ,e_{0k} )$ solution to \eqref{equ3332} is possible for any initial condition $\psi_0$ to the outer Problem \eqref{equpsi1}. We have the validity of Lipschitz dependence of  $\phi = \phi [\psi_0 ] $, and $e_0 = e_0 [\psi_0 ]$ in the $C^1(\bar\Omega)$-topology. Indeed, Remarks \ref{rmk1} and \ref{rmk2} give that
$$
|e_0 [\psi_0^{(1)} ] - e_0 [\psi_0^{(1)} ] | \leq c \left[
\| \psi_0^{(1)}  - \psi_0^{(1)}  \|_{L^\infty (\Omega)} +
\| \nabla \psi_0^{(1)}  - \nabla \psi_0^{(1)}  \|_{L^\infty (\Omega)}\right]
$$
for some fixed constant $c$. Moreover, as a 
consequence of the Implicit Function Theorem the maps $\phi [\psi_0 ] $, and $e_0 [\psi_0 ]$ depends in $C^1$-sense on $\psi_0$ in the $C^1(\bar\Omega)$-topology, thanks to the corresponding dependence for  $\psi$, $\la$ and $\xi$.

\medskip 
A second observation we make is that a slight change in the above proof allows to find an interesting variation of our main result: There is a manifold of initial conditions indexed by initial values of the parameters $\mu$ and by the centers $\xi_j(t_0)$, $j=1,\ldots, k$. In the proof we can prescribe their initial values as the corresponding $q_j$'s, letting freedom in the final points
$\xi_j(+\infty)$ which in any case is close to the $q_j$. We can achieve this by simply replacing the operator \equ{solsyst2} by the formula 

\be\label{solsyst22}
\xi_j (t) =  q_j + c \, \big [b_j^{n-2} \nn_x H(q_j ,q_j)   \,  - \,  \sum_{i\ne j}  b_j^{\frac {n-2}2} b_i^{\frac {n-2}2}  \nn_x G( q_j,q_i )\, \big ] \int_{t_0}^t \mu_0^{n-2} (s)  \,   ds.+  \int_{t_0}^t h(s) \, ds.
\ee
The consequence of this observation is the following.
If we consider the initial datum in \equ{P} of the form 
$$
u(\cdot,0)   =   u^*_{\ve , q}(\cdot ,0)  +   \psi_0 -  \sum_{j=1}^k  e_{0j} [\psi_0]  Z_{0j} 
$$
where $ u^*_{\mu, \xi }(x,t)$ is the first approximation 
given by \equ{bb1}  with fixed $\xi(0) =q$ and $\mu(0) = \ve$,  $Z_{0j}(x) =  \ve_j^{-\frac{n-2}2} \eta_{j,R} \left (\frac { x-q_j}{\ve_j} \right) $,
 and $\psi_0$  is sufficiently small, then $u(x,t)$ blows-up at $k$ points $\ttt q_j$  in the form \equ{forma1}.

Let us consider the following map defined  in a small neighborhood of 0 in  $X= C^1(\bar\Omega)$.
$$
F (\psi_0) =    \psi_0 -  \sum_{j=1}^k  (e_{0j} [\psi_0]- e_{0j}[0])  Z_{0j}
$$
so that $F[0]=0$, $F$ is differentiable and
$$
D_{\psi_0} F (0) [h]  =  h -  \sum_{j=1}^k  \langle  D_{\psi_0} e_{0j} [0], h \rangle  Z_{0j}, \quad h\in X. 
$$
We have a $k$-bubble blow-up as $t\to +\infty$  provided that 
\be\label{pq}
u(\cdot,0)   =   u^*_{\ve , q}(\cdot ,0) -  \sum_{j=1}^k  e_{0j} [0]  Z_{0j}    + g 
\ee
where $g= F[\psi_0] $ for any small $\psi_0$.

Let us assume  that the vector space of the functionals in $X$  $D_{\psi_0} e_{0j} [0]$ has dimension $\ttt k \le k$.  
We write $W= \bigcap_{j=1}^k {\rm  Ker}\, (D_{\psi_0} e_{0j} [0])$, so that $Z$ is a space with codimension $\ttt k$. Indeed,  
we can find $\ttt k$ linearly independent functions $e_j$ such that 
$$
X =     W\oplus  < \{e_1,\ldots,  e_{\ttt k}\}>. 
$$
We consider the operator in a neighborhood of $0$ in  $X$ given  by 
$$
G\big( w+  \sum_{j=1}^{\ttt k}  \alpha_j e_j \big ) \ =\   \sum_{j=1}^{\ttt k}  \alpha_j e_j +  F(w) , \quad \alpha_j\in \R,\quad w\in W.  
$$ 
Then $G$ is of class $C^1$ near the origin, $G(0)=0$  and  $D_{\psi_0} G(0) [h] = h$. By the local inverse theorem, $G$ defines a local $C^1$ diffeormorphism onto a neighborhood of the origin. For all small $g$ we can find smooth functions $\alpha(g)$, $w(g)$ with 
$$\sum_{j=1}^{\ttt k}  \alpha_j(g) e_j +  F(w(g)) = g. $$ Thus the set  $\mathcal M$ of functions $F[w]$, $w\in W$ can be described in a neighborhood of $0$ exactly as those $g\in X$
such that 
$$
\alpha_j (g) = 0 \foral  j=1,\ldots, \ttt k,  
$$  
in other words the space of $g$'s that satisfy $\ttt k$ $C^1$-constraints which in addition have linearly independent derivatives at 0. This says precisely that $\mathcal M$  is locally  
a codimension $\ttt k$ $C^1$-manifold, such that if $g$ in \equ{pq} is selected there, then the desired phenomenon takes place. Selecting a $k$-codimensional submanifold  of $\mathcal M$,
the result of the corollary follows. The proof is concluded. \qed

\bigskip

\setcounter{equation}{0}
\section{Linear theory for the inner problem} \label{seclineartheory}

The key linear ingredient in the proof above is the presence of the linear operator in Proposition \ref{prop000}. This section will be devoted to prove that result, in its more precise form of Proposition \ref{prop0} below. 
  
\medskip
Given a sufficiently large number $R>0$
we shall construct a solution to an initial value problem of the form
\be \label{p110}
\phi_\tau  =
\Delta \phi + pU(y)^{p-1} \phi + h(y,\tau )  \inn B_{2R} \times (\tau_0, \infty )
\ee
$$
\phi(y,\tau_0) = e_0Z_0(y)  \inn B_{2R}
$$
which defines a linear operator of $h$,
where $e_0= e_0[h]$ is a constant that defines a linear functional of $h$.
No boundary conditions are specified, while suitable time-space decay rates and orthogonality conditions are imposed on $h$. Let $\nu >0$ be such that
$$
\mu_0^{n-2 +\sigma} (t) = \tau^\nu,
$$
and $a$ the positive number, less than 1, which was introduced in \eqref{rest13}.
The solution we build has $R$-dependent uniform bounds in  $L^\infty$-weighted norms of the type
$$
\|h\|_{\nu, a} := \sup_{\tau > \tau_0 }  \sup _{y\in B_{2R}}   \tau^{\nu} (1+ |y|^a ) \, |h( y ,\tau ) | .
$$
Also, for a function $p=p(\tau)$ we denote
$$
\|p\|_{\nu}:= \sup_{\tau > \tau_0 } \tau^{\nu} |p(\tau ) | .
$$

\medskip
As we shall see, our construction involves solving the equation in different spherical harmonic modes.
Let us consider an orthonormal basis $\Theta_m$, $m=0,1,\ldots,$ of  $L^2(S^{n-1})$ made up of spherical harmonics, namely eigenfunctions of the problem
$$
\Delta_{S^{n-1}} \Theta_m + \la_m \Theta_m = 0 \inn S^{n-1}
$$
so that $$0=\la_0 < \la_1 =\ldots= \la_n = (n-1) < \la_{n+1} \le \ldots $$
We have $\Theta_0(y) = \alpha_0 $ and $\Theta_j(y) = \alpha_1{y_j}$, $j=1,\ldots, n$, for constant numbers $\alpha_0$ and $\alpha_1$. In general, all eigenvalues $\la_m$ are of the form $\ell (n-2+\ell)$ for a nonnegative integer $\ell$.

Let $h\in L^2(B_{2R})$. We decompose it into the form
$$
h(y,\tau) =     \sum_{j=0}^\infty h_j(r,\tau)\Theta_j (y/r), \quad r=|y|, \quad h_j(r,\tau) = \int_{S^{n-1}} h (r\theta , \tau ) \Theta_j(\theta) \, d\theta.
$$
In addition, we write $h = h^0 + h^1 + h^\perp$ where
$$
h^0 =  h_0(r,\tau), \quad h^1 = \sum_{j=1}^n h_j(r,\tau)\Theta_j, \quad h^\perp = \sum_{j=n+1}^\infty h_j(r,\tau)\Theta_j.
$$
Consider also the analogous decomposition for $\phi$ into $\phi = \phi^0 + \phi^1 + \phi^\perp$. It clearly suffices to build the solution $\phi$ of Problem \equ{p110} by doing so separately for the pairs
$(\phi^0,h^0)$, $(\phi^1, h^1)$ and $(\phi^\perp, h^\perp)$.

\medskip
Our main result in this section is the following proposition.

\begin{prop} \label{prop0}
Let $\nu,a$ be given positive numbers with $0<a <1$. Then, for all sufficiently large $R>0$ and any  $h=h(y,\tau)$  with  $\|h\|_{\nu, 2+a} <+\infty$
that satisfies for all $j=1,\ldots, n+1$
\be
 \int_{B_{2R}} h(y ,\tau)\, Z_{j} (y) \, dy\ =\ 0  \foral \tau\in (\tau_0, \infty)
\label{ortio}\ee
there exist  $\phi = \phi [h]$  and $e_0 = e_0 [h]$ which solve Problem $\equ{p110}$. They define linear operators of $h$
that satisfy the estimates
\be
  |\phi(y,\tau) |  \ \lesssim   \  \tau^{-\nu} \Big [\, \frac {R^{n-a}} { 1+ |y|^{n}} \,   \|h^0\|_{\nu, 2+a}   +  \frac {R^{n+1-a}} { 1+ |y|^{n+1}} \,   \|h^1\|_{\nu, 2+a} +
  \frac 1{1+ |y|^a} \|h \|_{\nu, 2+a}\Big ] ,
\label{cta1}\ee
and
\be
 | e_0[h]| \, \lesssim \,  \|h\|_{ \nu, 2+a}.
\label{cot2}\ee
\end{prop}

\medskip
Proposition \ref{prop0} is a direct consequence of Proposition \ref{prop1} below, to whose proof we will devote most of this section.
It refers to the following problem:

\be \label{p11}
\phi_\tau  =
\Delta \phi + pU(y)^{p-1} \phi + h(y,\tau )-c(\tau) Z_0 \inn B_{2R} \times (\tau_0, \infty )
\ee
$$
\phi(y,0) = 0 \inn B_{2R}
$$

\begin{prop} \label{prop1}
Let $\nu,a $ be given positive numbers with $0< a < 1 $. Then, for all sufficiently large $R>0$ and any $h$ with  $\|h\|_{\nu, 2+a} <+\infty$ and satisfying the orthogonality conditions \eqref{ortio},
there exist  $\phi = \phi[h]$ and $c=  c[h]$ which solve Problem $\equ{p11}$,  and define linear operators of $h$. The function
$\phi[h]$ satisfies estimate $\equ{cta1}$,
and for some $\gamma >0$
\be
\left |  c(\tau) -  \int_{B_{2R}} hZ_0 \right | \, \lesssim \, \tau^{-\nu}\, \Big [\,    R^{2-a} \left \| h -  Z_0 \int_{B_{2R}} hZ_0\, \right \|_{ \nu, 2+a } +     e^{-\gamma R} \|h\|_{ \nu, 2+\alpha } \Big ].
\label{cta2}\ee
\end{prop}

\medskip
\subsubsection*{\bf Proof of Proposition \ref{prop0}.}
Let us derive Proposition \ref{prop0} from Proposition \ref{prop1}. Let us write
\be
\phi(y,\tau) = \phi_1(y,\tau) +  e(\tau) Z_0(y)
\label{vv}\ee
where $\phi_1$ is the solution of Problem \eqref{p11} predicted by Proposition \ref{prop1}.
Assuming that $e\in C^1 \left( [\tau_0,\infty) \right)$ we find
$$
\pp_\tau \phi   =   \Delta \phi  + pU^{p-1} \phi   + h(y,\tau ) +   \left[ \, e'(\tau)-\la_0 e(\tau) - c(\tau) \, \right] \, Z_0(y).
$$
At this point we make the natural choice of $e(\tau)$ as the unique bounded solution of the equation
$$
e'(\tau)\,-\, \la_0 e(\tau)\, =\,  c(\tau), \quad \tau \in (\tau_0,\infty)
$$
which is explicitly given by
$$
e(\tau )  =   \int_\tau^\infty \exp({ \sqrt{\la_0} (\tau - s)})\, c(s)\, ds  \, .
$$
The function $e$ depends linearly on $h$. Besides, we clearly have from \equ{cta2},
$$
|e(\tau)| \ \lesssim \ \|c\|_\nu  \ \lesssim \  \tau^{-\nu} \|h\|_{ \nu, 2+a}.
$$
and thus, from the fact that $\phi_1$ satisfies estimate \equ{cta1},  so does $\phi$ given by \equ{vv}. Thus
$\phi$ satisfies Problem \equ{p110} with initial condition $\phi(y,\tau_0) =  e(\tau_0 ) Z_0(y)$.
The proof is concluded. \qed


\bigskip
We devote the rest of the Section to prove  Proposition \ref{prop1}  which is at the core of the proof of our main results.


\subsection{The slow-decay  solution}
As an intermediate result to establish Proposition \ref{prop1}, we shall consider Problem \equ{p11} in which the decay in space variable is relaxed to any positive power and no orthogonality assumptions on $h$ are made.  Instead we shall impose a zero Dirichlet boundary condition. For $h= h(y,\tau)$ we consider
the initial-boundary value problem

\be\label{modo0}
 \phi_\tau = \Delta \phi   + pU(y)^{p-1} \phi  + h(y,\tau) - c(\tau) Z_0(y)\inn B_{2R}\times (\tau_0,\infty)
 \ee
$$
\phi = 0 \onn  \pp B_{2R} \times (\tau_0,\infty)  ,\quad \phi(\cdot, \tau_0 ) = 0 \inn B_{2R}.
$$

\medskip
 Our principal result in this subsection is the following

\begin{lemma} \label{slow}

Let $\nu >0$ and $0<a<3$.
Then, for all sufficiently large $R>0$ and any $h$ with  $\|h\|_{\nu, a} <+\infty$
there exists  $\phi$ and $c$ which solve Problem $\equ{modo0}$ which define linear operators of $h$
and satisfy the estimates
$$
  |\phi(y,\tau)|
   \ \lesssim   \
   $$
   \be
   \tau^{-\nu} \Big [
     \frac {R^{n-2} \theta_R^0  \| h^0\|_{\nu,a} }{ 1+ |y|^{n-2} }     +   \frac {R^{n} \theta_R^1\| h^1\|_{\nu,a} }{ 1+ |y|^{n-1} }    +   \frac{\theta_R^0 \| h \|_{\nu,a}}{1+|y|^{n-2}} +   \frac{ \| h \|_{\nu,a}}{(1+ |y|)^{a-2}}\Big ] \label{cota11}\ee
where
\be\label{theta}
\theta_R^0 = \left \{ \begin{matrix}  1   & \hbox{\rm  if } a> 2 \\   \log R  & \hbox{\rm if } a= 2 \\ R^{2-a}   & \hbox{ \rm if } a < 2, \end{matrix}\right. , \quad \theta_R^1 = \left \{ \begin{matrix}  1   & \hbox{\rm  if } a> 1 \\   \log R  & \hbox{\rm if } a= 1 \\ R^{1-a}   & \hbox{ \rm if } a < 1, \end{matrix}\right.
\ee
and for some $\gamma >0$
\be
\left |  c(\tau) -  \int_{B_{2R}} hZ_0 \right | \, \lesssim \, \tau^{-\nu}\, \Big [\,  \theta_R^0 \left \| h -  Z_0 \int_{B_{2R}} hZ_0\, \right \|_{ \nu, a} +     e^{-\gamma R} \|h\|_{ \nu, a} \Big ].
\label{cota22}\ee
\end{lemma}

\medskip
The proof requires technical ingredients which we will establish in Lemmas \ref{lema01} and  \ref{lemaphi0}  below.
We will make use of the following basic, key lemma regarding the quadratic form associated to the linear operator $L_0 = \Delta + pU^{p-1}$,
\be\label{Q}
Q(\phi,\phi) := \int \left[  |\nn \phi|^2 - pU^{p-1} |\phi|^2 \right]  .
\ee
The next result provides an estimate of the associated second $L^2$-eigenvalue in a ball $B_{2R}$ with large radius under zero boundary conditions.

\begin{lemma} \label{lema01}
There exists a constant $\gamma >0$ such that for all sufficiently large $R$ and all radially symmetric function $\phi\in H_0^1(B_{2R})$ with
 $ \int_{B_{2R}} \phi Z_0 = 0 $
 we have
\be \label{ineq0}
\frac {\gamma} {R^{n-2}} \int_{B_{2R}} |\phi|^2 \ \le \ Q(\phi,\phi) .
\ee
\end{lemma}

\proof
Let $H_R$ be the linear space of all radial functions $\phi\in H_0^1(B_{2R})$ that satisfy the
orthogonality condition $ \int_{B_{2R}} \phi Z_0 = 0 $, and
$$
\la_R := \inf  \left \{ \, {Q(\phi, \phi)}\ /\   {\phi\in H_R}, \  \int_{B_{2R}} |\phi|^2 = 1 \, \right \}   .
$$
A standard compactness argument yields that $\la_R$ is achieved by a radial function $\phi_R (x) = \psi_R(r)\in H_R$ with $\int_{B_{2R}}  \phi_R^2 = 1 $, which satisfies the equation
\be
L_0[\phi_R] + \la_R  \phi_R =  c_R Z_0 \inn B_{2R}, \quad \phi_R = 0\onn \pp B_{2R},
\label{ggg}\ee
for a suitable Lagrange multiplier $c_R$.
We have that $\la_R\ge 0$. Indeed, the radial eigenvalue problem in $\R^n$
\be
\mathcal L_0[\psi] + \la \psi   = 0 , \quad \psi'(0) = \psi(+\infty) =0
\label{l0}\ee
where
$$
\mathcal L_0[\psi] := \psi'' + \frac{n-1}r  \psi' +p U(r)^{p-1}\psi
$$
has just one negative eigenvalue, as it follows from  maximum principle, using the fact that
$
\mathcal L_0[Z]= 0
$
with $Z=Z_{n+1}$, and the fact that this function changes sign just once. It follows that the associated quadratic form must be positive in $H^1(\R^n)$-radial, subject to the $L^2$-orthogonality condition
with respect to $Z_0$. This implies $\la_R\ge 0$.
Thus, to establish \equ{ineq0},  we assume by contradiction that
\be \label {cont}
\la_R = o( R^{2-n} ) \ass R\to +\infty .
\ee
Let $\chi$ be a smooth cut-off function with
\be  \hbox{ $\chi(s)=1$ for $s<1$ and $\chi(s)=0$ for $s>2$.}  \label{chi}\ee
Testing \equ{ggg} against $Z_0 (y)\eta_R  (|y|) $  where $\eta_R (s) = \chi( s- \frac R2)$, we get
$$
c_R\int_{B_{2R}} Z_0^2\eta_R  =   \int_{B_{2R}} \phi_R  [Z_0 \Delta\eta_R + 2\nn\eta_R \nn Z_0 ].
$$
Since $\| \phi_R\|_{L^2(B_{2R})}=1$,
 it follows that, for some $\sigma >0$,
$
c_R =O( e^{-\sigma R}) .
$
On the other hand, again using that $\| \phi_R\|_{L^2(B_{2R})}=1$,  standard elliptic estimates yield that
$\| \phi_R\|_{L^\infty(B_{2R})}\lesssim  1$.

\medskip
Let us represent $\phi_R(x) = \psi_R(r)$ using the variation of parameters formula.
The function $\psi_R$ satisfies the ODE
\be \label{plrk}
\mathcal L_0[\psi_R]   = h_R(r) ,
 \quad
r\in (0,R),\quad \psi_R'(0) = \psi_R(R) = 0 \ee
where $h_R(r) = -\la_R \psi_R (r) +  c_RZ_0(r). $ Furthermore,  it is uniformly bounded in $R$.

For
$Z= Z_{n+1}$, we  consider a second, linearly independent solution $\ttt Z(r)$ of this problem,
namely ${\mathcal L}_0 [ \tilde Z ] = 0$,  normalized
in such a way that their Wronskian satisfies
$$
\ttt Z' Z(r)  - \ttt Z Z'(r) =  \frac 1{r^{n-1}} .
$$
Since $Z(r) \sim 1  $ near $r=0$ and $ Z(r) \sim r^{2-n}$ as $r\to \infty$, we see
that $\ttt Z(r) \sim  r^{2-n}$ near $r=0$ and $\ttt Z(r)\sim 1$ as $r\to \infty$.
The formula of variation of parameters then yields the representation
\be \psi_R(r) =  \ttt Z(r)\int_0^r h_R(s)\,Z(s)\, s^{n-1}\,ds  +   Z(r) \int_r^{2R}   h_R(s)\,\ttt Z(s)\, s^{n-1}\,ds  -  A_RZ(r)
\label{ll0}\ee
where $ A_R $ is such that  $\psi_R (2R) =0$, namely
$$
A_R =   Z(2R)^{-1} \ttt Z(2R)\int_0^{2R} h_R(s)\,Z(s)\, s^{n-1}\,ds.
$$
We observe that
$ \|h_R\|_{L^2(B_{2R})} \lesssim \la_R + e^{-\sigma R} .$
Then we estimate
$$
\left |  \int_0^r h_R(s)\,Z(s)\, s^{n-1}\,ds \right | \ \le \ \|Z\|_{L^2(B_{2R})} \|h_R\|_{L^2(B_{2R})} \, \lesssim \,  (\la_R  +  e^{-\sigma R})  \|Z\|_{L^2(B_{2R})}
$$
and
$$
\left |  \int_r^{2R}   h_R(s)\,\ttt Z(s)\, s^{n-1}\,ds  \right | \ \lesssim \  R^{\frac n2} ( \la_R  +  e^{-\sigma R}).
$$
Hence we have for instance,
$$
\|A_RZ \|_{L^2(B_{2R})}  \, \lesssim \,   R^{n-2} (\la_R  +  e^{-\sigma R} ) \|Z \|_{L^2(B_{2R})},
$$
and estimating the other two terms we obtain at last, thanks to \equ{cont},
\be\label{sa}
\|\phi_R\|_{L^2(B_{2R})} \le       R^{n-2} (\la_R  +  e^{-\sigma R} )  \|Z\|_{L^2(B_{2R})} .\ee
At this point we notice
$\|Z\|_{L^2(\R^n)} < +\infty$.
This, the fact that $\la_R = o( R^{2-n})$ and relation \equ{sa}  yield that $\|\phi_R\|_{L^2(B_{2R})} \to 0$. This is a contradiction and estimate \equ{ineq0}
is thus proven.
\qed

 \medskip
  We let $\chi (s)$ be a smooth cut-off function as in \equ{chi}.
 We consider a large but fixed number $M$ which we will fix later independently of $R$ and define $\chi_M (y) = \chi ( |y| -M)$.
 Let us consider the auxiliary problem, for a general $h(y,\tau)$,

  \be\label{a1}
 \phi_\tau = \Delta \phi   + pU(r)^{p-1}(1-\chi_M)\phi  + h(y,\tau)  \inn B_{2R}\times (\tau_0,\infty)
\ee
 $$
\phi = 0 \onn  \pp B_{2R} \times (\tau_0,\infty)  ,\quad \phi(\cdot, 0) = 0 \inn B_{2R},
$$
By standard parabolic theory this problem has a unique solution, which defines a linear operator of $h$.

\begin{lemma} \label{lemaphi0}
Let $\phi_*[h]$ denote the unique solution of Problem $\equ{a1}$.
If $M$ is fixed sufficiently large, then the  following estimate holds, provided that $\tau_0$ was chosen fixed and sufficiently large,
\be \label{esphi0}
| \phi_*[h]|\ \lesssim\      \tau^{-\nu} \| h\|_{\nu,a}  \left \{ \begin{matrix}  (1+ r)^{2-a}    & \hbox{ if } a> 2 \\ \log R  & \hbox{ if } a= 2 \\ R^{2-a}   & \hbox{ if } a < 2.
 \end{matrix} \right.
\ee
\end{lemma} \proof
Let us consider the problem
\be\label{ii}
 \mathcal L _M [g]  +   \frac 1{1+r^a} = 0, \quad g'(0) = 0 = g(R) ,
 \ee
where
$$
 \mathcal L _M [g] =  g'' + \frac{n-1}r g'    +  pU^{p-1}(1-\chi_M) g .
$$
 We observe that the function $g_1(r) = Z_{n+1}(r) \sim r^{2-n} $ is an exact solution of $\mathcal L(g_1) = 0$ for $r>M+2 $. Let us assume that $M$ is chosen so large that $g_1'(M+2) <0$, $g_1(M+2 )>0$. Since, for $r<M+2$ $g_1$, must extend as a linear combination of a constant and the  fundamental solution of the Laplacian, we conclude that also $g_1(r) \sim r^{2-n}$ as $r\to 0$. The second linearly independent solution
 $g_2$ with $g_2(0)=1$ is constant up to $M$ and then continues, converging to a positive constant as $r\to \infty$. So $g_2(r) \sim 1$ for all $r$. The variation of parameters formula expresses then the unique solution of Problem \equ{ii}
as
\be
g(r) = g_2(r) \int_r^{2R} \frac {d\rho}{g_2 ^2\rho^{n-1}} \int_0^\rho  \frac { g_2 (s) s^{n-1} ds} {1+s^a} .
\label{jj}\ee
Then letting
$$
\bar \phi(r,\tau ) =  2  \| h\|_{\nu,a} \tau^{-\nu}g(r)
$$
 and choosing  $\tau_0$ sufficiently large in terms of $R$, we see that $\bar\phi $ is a positive  supersolution of Problem \equ{a1}. By parabolic comparison
 we then have $ |\phi_*[h]| \le \bar \phi $. Finally,
 directly checking the asymptotic behavior of formula \equ{jj}, estimate \equ{esphi0} readily follows. \qed

\bigskip
 \noindent
{\bf The Proof of Lemma \ref{slow}}.

\medskip
\noindent
{\bf Construction in the radial case.}
We shall first  solve Problem \equ{modo0} in the radial case, $h=h_0(r)$,
where  $\| h_0\|_{\nu,a}< +\infty$.
Let us consider the function
$$
\bar h_0 = h_0 -  c_0(\tau) Z_0 , \quad      c_0(\tau) =  \int_{B_{2R}} hZ_0
$$
and let $\phi_*[\bar h_0]$ be the solution of \equ{a1} as in Lemma \ref{lemaphi0}, which is of course radial.  Setting $\phi =  \phi_*[\bar h_0] + \ttt \phi $ and $c(\tau) = c_0(\tau) + \ttt c(\tau)$,
Problem \equ{modo0} gets reduced to solving

 \be\label{modo01}
 \ttt \phi_\tau = \Delta \ttt \phi   + pU(r)^{p-1} \ttt \phi  + \ttt h_0(r,\tau) - \ttt c(\tau) Z_0\inn B_{2R}\times (\tau_0,\infty)
 \ee
$$
\ttt \phi = 0 \onn  \pp B_{2R} \times (\tau_0,\infty)  ,\quad \ttt \phi(\cdot, \tau_0) = 0 \inn B_{2R} .
$$
where
$$
\ttt h_0 = p U^{p-1}\chi_M  \phi_*[\bar h_0 ] .
$$
The most important feature of $\ttt h_0$ is that it is radial, compactly supported with size controlled by that of $\bar h_0$. In particular we have that for any $m>0$,
\be \label{tete}
\|  \ttt h_0  \|_{m,\nu}\ \lesssim\  \| \phi_*[\bar h_0]  \|_{0,\nu}\ 
\ee
We shall next solve Problem \equ{modo01} under the additional orthogonality constraint
\be \label{modo02}
\int_{B_{2R}}  \ttt \phi (\cdot , \tau )\, Z_0   \ =\ 0 \foral \tau\in (\tau_0, \infty ) .
\ee
Problem
\equ{modo01}-\equ{modo02} is equivalent to solving just \equ{modo01} for $\ttt c$ given by the explicit linear functional
$\ttt c := \ttt c[\ttt \phi, \ttt h_0] $  determined  by the relation
\be \label{cc}
\ttt c(\tau) \int_{B_{2R}} Z_0^2    =  \int_{B_{2R}} \ttt  h_0 (\cdot, \tau)  Z_0     +     \int_{\pp B_{2R}} \pp_r \ttt \phi(\cdot, \tau)  Z_0 .
\ee
Since $Z_0( R) = O(e^{-\gamma R})$ for some $\gamma>0$, the dependence in $\phi$ of this functional  is small for instance in an $L^\infty$-$C^{1+\alpha, \frac {1+\alpha} 2} $ setting. This implies that standard linear parabolic theory together with a contraction argument apply to yield existence of a unique solution to \equ{modo01}-\equ{modo02} defined at all times. Next we shall elaborate further  on explicit estimates on $\ttt \phi$. Crucial in this endeavor is the mild coercivity estimate for the quadratic form $Q$ in Lemma \ref{lema01}. To get the estimates, smoothness of the data can be assumed so that integrations by parts and differentiations can be carried over, and then arguing by approximations.
Testing \equ{modo01}-\equ{modo02} against $\ttt \phi$ and integrating in space, we obtain the relation
$$
\pp_\tau \int_{B_{2R}} \ttt \phi^2     + Q(\ttt \phi,\ttt \phi)  =  \int_{B_{2R}} g\ttt \phi, \quad g= \ttt h_0 - \ttt c(\tau) Z_0 .
$$
Using the estimate in Lemma \ref{lema01} we get that for some $ \gamma >0$,
\be\label{en}
\pp_\tau \int_{B_{2R}} \ttt \phi^2     + \frac{\gamma }{R^{n-2}} \int_{B_{2R}} \ttt \phi^2  \lesssim   R^{n-2}\int_{B_{2R}} g^2 .
\ee
We observe that from \equ{cc} and \equ{tete} for $m=0$  we get that
$$
|\ttt c(\tau)| \le \tau^{-\nu } \left [ \| \phi_*(\bar h_0)  \|_{0,\nu}   + e^{-\gamma R}  \|\nn_y \ttt \phi \|_{0,\nu} \right ] .
$$
Besides,  using again estimate \equ{tete} for a sufficiently large $m$, we get
$$
\int_{B_{2R}} g^2 \ \lesssim\     \tau^{-2\nu } \left [ \| \phi_*(\bar h_0)  \|_{0,\nu}  + e^{-\gamma R}  \|\nn_y \ttt \phi \|_{0,\nu} \right ]^2  .
$$
Using that $\ttt\phi(\cdot, \tau_0) =0$ and Gronwall's inequality, we readily get from \equ{en} the $L^2$-estimate
\be\label{K}
\|\ttt \phi (\cdot , \tau ) \|_{L^2(B_{2R})} \  \lesssim \ \tau^{-\nu } R^{n-2}  K, \quad K:=  \left [ \| \hat h  \|_{0,\nu}  + e^{-\gamma R}  \|\nn_y \ttt \phi \|_{0,\nu} \right ] .
\ee
for all $\tau > \tau_0$. Now, using standard parabolic estimates in the equation satisfied by $\ttt \phi$ we obtain then that on any large fixed radius $M>0$,
$$
\|\ttt \phi (\cdot , \tau ) \|_{L^\infty(B_M)} \  \lesssim \ \tau^{-\nu } R^{n-2} K \foral \tau > \tau_0.
$$
Since the data in the equation has arbitrarily fast space decay, we can dominate the solution  outside $B_M$ by a barrier of the order $\tau^{-\nu} |y|^{2-n} $.
As a conclusion, also using local parabolic estimates for the gradient, we find that
$$
 |\nn_y \ttt \phi (y , \tau )| +  |\ttt \phi (y , \tau )|   \ \lesssim\    \tau^{-\nu } R^{n-2} K \, |y|^{2-n} ,
 $$
thus from the definition of $K$ we finally get
\be \label{i1}
|\nn_y \ttt \phi (y , \tau )| + |\ttt \phi (y , \tau )| \ \lesssim  \tau^{-\nu } \frac{ R^{n-2}} { |y|^{n-2}} \| \phi_*[\bar h_0]  \|_{0,\nu}   .
\ee
It clearly follows from this estimate, inequality \equ{tete} and  Lemma \ref{lemaphi0} that
the function
\be\label{p0} \phi^0 [h^0] :=  \ttt \phi + \phi_*[\bar h_0] \ee
solves Problem \equ{modo0}  for $h=h_0$ and satisfies
$$
  |\phi_0(y,\tau)|
   \ \lesssim   \    \tau^{-\nu} \frac {R^{n-2}} { 1+ |y|^{n-2} }   \theta_R^0\| h^0\|_{\nu,a}  \quad
$$
with $\theta_R^0$ given in \equ{theta}.
Finally, from \equ{cc} we see that
we have that  $$ c(\tau) =  \int_{B_{2R}}  h Z_0   +  \int_{B_{2R}}  pU^{p-1} \chi_M \phi_*[\bar h_0 ] \,Z_0   + O  (e^{-\gamma R} ) \|h\|_{a,\nu}. $$
From here and applying again Lemma \ref{lemaphi0} we find the validity of estimate
$$
\left |  c(\tau) -  \int_{B_{2R}} h_0Z_0 \right | \, \lesssim \, \tau^{-\nu}\, \Big [\,  \theta_R^0 \left \| h^0 -  Z_0 \int_{B_{2R}} h_0Z_0\, \right \|_{ \nu, a} +     e^{-\gamma R} \|h_0\|_{ \nu, a} \Big ].
$$
Hence estimates \equ{cota11} and \equ{cota22} hold. The construction of the solution at mode 0 is concluded.

\bigskip
\subsubsection*{\bf Construction at modes $1$ to $n$.}
Here we consider the case $h=h^1$ where
$$
h^1(y,\tau)  =  \sum_{j=1}^nh_j(r,\tau) \Theta_j.
$$
The function
\be\label{p1}
\phi^1[h^1] :=  \sum_{j=1}^n\phi_j(r,\tau) \Theta_j.
\ee
solves the initial-boundary value problem
\be\label{modo1}
 \phi_\tau = \Delta \phi   + pU(y)^{p-1} \phi  + h^1(y,\tau)  \inn B_{2R}\times (\tau_0,\infty)
 \ee
$$
\phi = 0 \onn  \pp B_{2R} \times (\tau_0,\infty)  ,\quad \phi(\cdot, \tau_0 ) = 0 \inn B_{2R}.
$$
We shall estimate this solution.  The functions $\phi_j(r,\tau)$ solve

\be\label{modo11}
 \pp_\tau \phi_j  = \mathcal L_1 [\phi_j]  + h_j(r,\tau)   \inn (0,2R)\times (\tau_0,\infty)
 \ee
$$
\partial_r \phi_j(0,\tau) = 0 = \phi_j(R,\tau)  \foral \tau\in (\tau_0,\infty)  ,\quad \phi_j(r, \tau_0 ) = 0 \foral r\in (0,R),
$$
where
\be\label{mathcalL1}
\mathcal L_1 [\phi_j] := \pp_{rr}\phi_j  + (n-1)\frac{ \pp_{r}\phi_j } r -(n-1)\frac{\phi_j } {r^2}  + pU(r)^{p-1} \phi_j.
\ee
Let us assume that
$
\| h_j\|_{a,\nu} < +\infty,
$
so that
$$ |h^1(r,\tau ) | \ \le \    \tau^{-\nu } \| h^1\|_{a,\nu} (1+r)^{-a}. $$
Let us consider
the solution of the stationary problem
$$
  \mathcal L_1 [\phi]  +  (1+r)^{-a} =0
$$
given by the variation of parameters formula
$$
\bar \phi(r) =  Z(r) \int_r^{2R} \frac 1{\rho^{n-1} Z(\rho)^2} \int_0^\rho (1+s)^{-a} Z(s)s^{n-1}\, ds
$$
where $Z(r) = w_r(r).$ Since $w_r(r)\sim r^{1-n}$ for large $r$, we find the estimate
$$
|\bar \phi(r) | \ \lesssim \ \frac { R^{n} \theta_R^1  } {1 + r^{n-1}}  , \quad
$$
where $\theta_R^1$ is precisely the number defined in \equ{theta}.
Then, provided that $\tau_0$ was  chosen sufficiently large, the function  $2\|h_j\|_{a,\nu} \tau^{-\nu} \bar \phi (r) $ is a positive supersolution of Problem  \equ{modo11}
 and thus  we find
$$
|\phi_j(r,\tau) |   \ \lesssim \ \tau^{-\nu}  \frac { R^{n} \theta_R^1   } {1 + r^{n-1}}\| h_j\|_{a,\nu} .
$$
Hence $\phi^1[h^1]$ given by \equ{p1} satisfies
$$
|\phi^1[h^1] (y,\tau) |   \ \lesssim \ \frac { R^{n} \theta_R^1   } {1 + |y|^{n-1}}\| h^1\|_{a,\nu} .
$$
so that estimate \equ{cota22} is satisfied by this function

\medskip
\subsubsection*{\bf The case of higher modes. }
We consider now the case

$$h= h^\perp = \sum_{j=n+1}^\infty h_j(r) \Theta_j$$
and the problem
\be\label{modok}
 \phi_\tau = \Delta \phi   + pU(y)^{p-1} \phi  + h^\perp  \inn B_{2R}\times (\tau_0,\infty)
 \ee
$$
\phi = 0 \onn  \pp B_{2R} \times (\tau_0,\infty)  ,\quad \phi(\cdot, \tau_0 ) = 0 \inn B_{2R},
$$
whose solution has the form
$$\phi^\perp = \sum_{j=n+1}^\infty \phi_j(r, \tau) \Theta_j$$
Our first claim is that for $\phi^\perp\in H_0^1(B_{2R})$  we have the coercivity estimate for the quadratic form \equ{Q}
\be\label{co}
\int_{B_{2R}} \frac {|\phi^\perp|^2} {r^2} \lesssim Q(\phi^\perp ,\phi^\perp )
\ee
Indeed, we can decompose
$$
Q(\phi^\perp ,\phi^\perp ) =   c\sum_{j=n+1}^\infty Q_j(\phi_j,\phi_j)
$$
where
$$
Q_j(\psi,\psi) = \int_0^{2R} \left [\, |\psi'|^2 - pw^{p-1}\psi^2 + \mu_j\frac{\psi^2}{r^2}\,\right ]\, r^{n-1}dr\, ,
$$
and $\mu_j = \ell(n-2+ \ell)$ for some  $\ell \ge 2$.
It follows that
\be\label{qq}
Q_j(\phi_j,\phi_j) \ge    Q_1(\phi_j,\phi_j) + (n+1) \int_0^{2R} \frac{\phi_j^2}{r^2}r^{n-1}dr.
\ee
But we have,  $Q_1(\phi_j,\phi_j) \ge 0$.
In fact, writing $\phi_j = w_r\psi_j$ we  get the identity
$$
Q_1(\phi_j,\phi_j) =   \int_0^{2R}   |\psi_j'|^2\, w_r^2r^{n-1}dr\, \ge\,  0
$$
Hence, after addition of the terms in \equ{qq} we get
$$
Q(\phi^\perp ,\phi^\perp ) \ge    c\sum_{j=n+1}^\infty  \int_0^{2R} \frac{|\phi_j|^2}{r^2}r^{n-1}dr = c \int_{B_{2R}} \frac {|\phi^\perp|^2} {r^2} .
$$
and thus estimate \equ{co} holds.
Let $\phi_*[h^\perp] $ be the solution in Lemma \ref{lemaphi0}. Then $\phi_*[h^\perp]$ only has components in spherical harmonics $\Theta_j$ with $j\ge 2$.
By writing $\phi =\phi_*[h^\perp] +\ttt\phi$, Problem \equ{modok}   reduces to solving

\be
 \ttt\phi_\tau = \Delta \ttt\phi   + pU(y)^{p-1} \ttt \phi  + \ttt h  \inn B_{2R}\times (\tau_0,\infty)
 \ee
$$
\ttt \phi = 0 \onn  \pp B_{2R} \times (\tau_0,\infty)  ,\quad \ttt \phi(\cdot, \tau_0 ) = 0 \inn B_{2R},
$$
where
$$
\ttt h = pU(y)^{p-1}\chi_M \phi_*[h^\perp] ,
$$
for a sufficiently large $M$.
Arguing as in \equ{en} we now get
\be\label{en1}
\pp_\tau \int_{B_{2R}} \ttt \phi^2     + c\int_{B_{2R}} \frac {|\ttt \phi|^2}{|y|^2}   \lesssim   \int_{B_{2R}} |y|^2 |\ttt h|^2  .
\ee
Similarly to \equ{K}, using the estimates for in Lemma \ref{lemaphi0} we get

\be\label{K1}
\|\, |y|^{-1} \ttt \phi (\cdot , \tau ) \|_{L^2(B_{2R})} \  \lesssim \ \tau^{-\nu } \theta^0_R  \| h  \|_{a,\nu}
\ee
where $\theta^0_R$ is defined in \eqref{theta}.
From elliptic estimates we then get that
$$
\|\ttt \phi (\cdot , \tau ) \|_{L^\infty(B_{2R})} \  \lesssim \ \tau^{-\nu }  \theta^0_R \| h^\perp  \|_{a,\nu} .  \foral \tau > \tau_0,
$$
so that with the aid of a barrier we obtain

$$
 |\ttt \phi (y , \tau )|   \ \lesssim\    \tau^{-\nu } \theta^0_R \| h^\perp  \|_{a,\nu} \,(1+ |y|)^{2-n} .
 $$
  It follows that the function
\be \label{pperp}
 \phi^\perp[h^\perp]   : =  \ttt \phi + \phi_*[ h^\perp ] \ee  satisfies
$$
|\phi^\perp [h^\perp] (y , \tau )|    \ \lesssim\   \tau^{-\nu } \left [ \theta^0_R \,(1+ |y|)^{2-n}    +  (1+|y|)^{a-2}  \right ]\, \| h^\perp  \|_{a,\nu}\inn B_{2R}.
$$

\subsubsection*{\bf Conclusion.}
We simply let
$$
\phi[h] := \phi^0[h^0] +  \phi^1[h^1] + \phi^\perp[h^\perp]
$$
for the functions defined in \equ{p0}, \equ{p1}, \equ{pperp}. By construction, $\phi[h]$
solves Equation \equ{modo0}. It defines a linear operator of $h$ and satisfies the conclusions of Lemma \ref{slow}.
The proof is concluded. \qed

\bigskip
\noindent
{\bf Proof of Proposition \ref{prop1}}. \ \
As in the proof of Lemma \ref{slow} we decompose $h$ in modes,
$h= h^0+ h^1 +h^\perp$,
$$
h^0 = h_0(r , \tau ), \quad h^1= \sum_{j=1}^n h_j(r , \tau )\Theta_j  ,  \quad h^\perp= \sum_{j=n+1}^\infty h_j(r, \tau )\Theta_j
$$
and define separately associated solutions of \equ{p11} in a decomposition $\phi= \phi^0 + \phi^1+ \phi^\perp$.

\medskip
For a bounded radial $h=h(|y|)$  defined in $B_{2R}$ with $\int_{B_{2R}} hZ_{n+1}=0$
the equation
$$
\Delta H + pU^{p-1}H + \ttt h_0(|y|) = 0 \inn \R^n , \quad H(y) \to  0 \ass |y|\to \infty
$$
where $\ttt h$ designates the extension of $h_0$ as zero outside $B_{2R}$,
has a  solution $H =: L_0^{-1}[h] $ represented by the variation of parameters formula

\be  H(r) =  \ttt Z(r)\int_r^{\infty}  \ttt h(s)\,Z(s)\, s^{n-1}\,ds  +   Z(r) \int_r^{\infty}   \ttt h(s)\,\ttt Z(s)\, s^{n-1}\,ds
\label{ll0}\ee
where $Z= Z_{n+1}(r)$, and $\ttt Z(r)$ is a suitable second linearly independent radial solution of $ L_0[\ttt Z] = 0 $.

If we consider a function $h_0 = h_0(|y|,\tau)$ defined in $B_{2R}$ with $\|h_0\|_{2+a, \nu} <+\infty $  and $\int_{B_{2R}} h_0 Z_{n+1}=0$ for all $\tau$, then
$H_0= L_0^{-1}[h_0(\cdot, \tau) ]$ satisfies the estimate
$$
\| H_0\|_{a, \nu} \lesssim \| h_0\|_{2+ a, \nu} .
$$
Let us consider the boundary value problem in $B_{3R}$
\be\label{modo00}
 \Phi_\tau = \Delta \Phi   + pU(y)^{p-1} \Phi  + H_0(|y|,\tau) - c_0(\tau) Z_0\inn B_{3R}\times (\tau_0,\infty)
 \ee
$$
\Phi = 0 \onn  \pp B_{3R} \times (\tau_0,\infty)  ,\quad \Phi(\cdot, \tau_0 ) = 0 \inn B_{3R}.
$$
Invoking Lemma \ref{slow} we find a radial solution $\Phi_0[h_0 ]$ to this problem, which defines a linear operator of $h_0$ and satisfies the estimates

\be
  |\Phi_0(y,\tau) |  \ \lesssim   \  \frac{\tau^{-\nu} R^{n-2}} { 1+ |y|^{n-2}}  R^{2-a}  \| H_0 \|_{\nu,a} , \quad
\label{cota111}\ee
where for some $\gamma >0$
\be
\left |  c_0(\tau) -  \int_{B_{2R}} H_0Z_0 \right | \, \lesssim \, \tau^{-\nu}\, \Big [\,  R^{2-a} \left \| H_0 -  Z_0 \int_{B_{2R}} H_0Z_0\, \right \|_{ \nu, a} +     e^{-\gamma R} \|h_0\|_{ \nu, 2+a} \Big ].
\label{cota222}\ee
At this point we observe that since $L_0[Z_0] = \la_0 Z_0$ then
$$
\la_0 \int_{B_{2R}} H_0Z_0 = \int_{B_{2R}} H_0 L_0[ Z_0]  =    \int_{B_{2R}} L_0[H_0]\, Z_0  + \int_{\pp B_{2R}}  (Z_0 \pp_\nu H_0 -  H_0 \pp_\nu Z_0),
$$
and hence
$$
 \int_{B_{2R}} H_0Z_0    =   \la_0^{-1} \int_{B_{2R}} h_0\, Z_0  +  O(e^{-\gamma R}) \tau^{-\nu} \| h_0\|_{2+ a, \nu} .
$$
Also, from the definition of the operator $L_0^{-1}$ we see that $  Z_0  =  \la_0 L_0^{-1}[Z_0] $.
Thus
$$
\left \| H_0 -  Z_0 \int_{B_{2R}} H_0Z_0\, \right \|_{ \nu, a} = \left \|L_0^{-1}\Big [ h_0 -  \la_0 Z_0 \int_{B_{2R}} H_0Z_0  \Big ] \, \right \|_{ \nu, a}
$$
$$
\ \lesssim \
\left \| h_0 -  Z_0 \int_{B_{2R}} h_0Z_0\, \right \|_{ \nu, 2+a } +     e^{-\gamma R} \|h_0\|_{ \nu, 2+a}.
$$
Let us fix now a  vector $e$ with $|e|=1$, a large number $\rho>0$ with $\rho \le 2R$ and a number $\tau_1 \ge \tau_0$. Consider the change of variables
$$
\Phi_\rho (z,t) := \Phi ( \rho e +  \rho z , \tau_1 +  \rho^2 t ), \quad H_\rho (z,t) := \rho^2 [ H_0 ( \rho e +  \rho z , \tau_1 +  \rho^2 t ) - c_0( \tau_1 +  \rho^2 t) Z_0(\rho e +
\rho z)\, ].
$$

Then $\Phi_\rho( z, t) $ satisfies an equation of the form
$$
 \pp_t \Phi_\rho  = \Delta_z \Phi_\rho   +  B_\rho(z,t)  \Phi_\rho   +     H_\rho  (z,t)  \inn B_1(0) \times (0,2). 
$$
where  $B_\rho = O(\rho^{-2})$ uniformly in $B_2(0) \times (0,\infty)$.
Standard parabolic estimates yield that for any $0<\alpha < 1$
$$
 \|\nn_z \Phi_\rho \|_{L^\infty  (B_{\frac 12 }(0) \times (1,2))  } \lesssim  \|\Phi_\rho \|_{L^\infty ( B_{1 }(0) \times (0,2)) } + \| H_\rho \|_{L^\infty  (B_{1 }(0) \times (0,2)) }  .
$$
Moreover
$$
\| H_\rho \|_{L^\infty  (B_{1 }(0) \times (0,2)) }\ \lesssim\  \rho^{2-a}\tau_1^{-\nu}\| H_0 \|_{a,\nu} ,\quad
\|\Phi_\rho \|_{L^\infty  (B_{1 }(0) \times (0,2)) } \lesssim  \tau_1^{-1} K(\rho)
$$
where
\be\label{Knew}
K(\rho) =
\frac{ R^{n-2}} {  \rho^{n-2}}  R^{2-a}  \| h^0\|_{\nu,a}
\ee
This yields in particular that
$$
\rho |\nn_y \Phi(\rho e , \tau_1 + \rho^2 )|  \ = \  |\nn\ttt \phi(0, 1)| \lesssim   \tau_1^{-\nu} K(\rho) .
$$
Hence if we choose $\tau_0 \ge R^2$,
we get that for any $\tau > 2\tau_0$ and $|y|\le  3R$
\be \label{coco}
(1+ |y|)\, |\nn_y \Phi(y , \tau  )|  \ \lesssim  \   \tau^{-\nu} K(|y|)
\ee
We obtain that these bounds are as well valid for $\tau < 2\tau_0$   by the use of similar parabolic estimates up to the initial time (with condition 0).

\medskip
Now, we observe that the function $H_0$ is of class $C^1$ in the variable $y$ and  $\|\nn_y H_0\|_{1+ a, \nu} \le  \|h^0\|_{2+a,\nu} $. It follows that we have the estimate
$$
(1+ |y|^2)\, | D^2_y \Phi(y,\tau) |\ \lesssim \  \tau^{-\nu} K(|y|)
$$
for all $\tau> \tau_0$, $|y|\le  2R  $.
where $K$ is the function in \equ{Knew}.
The proof follows simply by differentiating the equation satisfied by $\Phi$, rescaling in the same way we did to get the gradient estimate, and apply the bound already proven for
 $\nn_y \Phi$.

$$
 (1+|y|^2) |D^2\Phi(y,\tau)| +  (1+|y|) | \nn \Phi(y,\tau)| \  + \  |\Phi(y,\tau)|
 $$
 $$
 \, \lesssim\,     \tau^{-\nu} \|h^0\|_{2+a,\nu} \frac{R^{n-a}} {1+ |y| ^{n-2}}  \inn B_{2R}  .
$$
This yields in particular
$$
| L_0 [\Phi](\cdot, \tau ) |\ \lesssim\  \tau^{-\nu} \|h^0\|_{2+a,\nu} \frac{R^{n-a}} {1+ |y|^n } \inn B_{2R}
$$
We define
$$
\phi^0[h_0] := L_0 [\Phi]\, \Big |_{B_{2R}}.
$$
Then  $\phi^0[h_0]$ solves Problem \equ{p11}
with
\be\label{ccnew}
c(\tau) := \la_0 c_0(\tau).
\ee
$\phi^0[h_0]$ satisfies the estimate
\be\label{pp0}
| \phi^0[h_0] (y, \tau) | \ \lesssim \ \tau^{-\nu} \|h_0\|_{2+a,\nu} \frac{R^{n}} {1+ |y|^n } R^{-a} \inn B_{2R}.
\ee
and from \equ{cota222},
estimate \equ{cta2} holds too.


\subsubsection*{\bf The construction for modes $1$ to $n$.}

We consider now
$$h^1(y,\tau) =  \sum_{j=1}^n h_j (r , \tau ) \Theta_j $$  with  $\|h^1\|_{\nu,2+a } <+\infty$
that satisfies for all $i=1,\ldots, n$
\be
\int_{B_{2R}} h^1 Z_i = 0 \foral \tau\in (\tau_0, \infty).
\label{ortio1}\ee
We will show that
there is  a solution $$ \phi^1[h^1]  = \sum_{j=1}^n \phi_j (r, \tau ) \Theta_j({y \over r}) $$
to Problem $\equ{p11}$ for $h= h^1$, which define a linear operator of $h^1$
and satisfies the estimate
\be
  |\phi^1(y,\tau)|  \ \lesssim   \   \frac {R^{n+1}}{ 1+|y|^{n+1} }R^{-a}     \|h\|_{\nu, 2+a} .
\label{cta4}\ee
The construction is similar to that of $\phi^0$. Let us fix $1\le j\le n$.
For a function  $h = h_j(r  ) \Theta_j  ({y \over r})  $ defined in $B_{2R}$, we
let $H= L_0^{-1}[h] := H_j(r) \Theta_j ({y \over r}) $ be the solution of the equation
$$
\Delta H + pU^{p-1}H + \ttt h_j \Theta_j  = 0 \inn \R^n , \quad H(y) \to  0 \ass |y|\to \infty
$$
where $\ttt h_j$ designates the extension of $h_j$ as zero outside $B_{2R}$, represented by the variation of parameters formula
$$
H_j(r) =  w_r(r) \int_r^{2R} \frac 1{\rho^{n-1} w_r(\rho)^2} \int_\rho^\infty \ttt h_j(s)\, w_r(s)s^{n-1}\, ds
$$
If we consider a function $h^j = h_j(r,\tau)\Theta_j$ defined in $B_{2R}$ with $\|h^j\|_{2+a, \nu} <+\infty $  and $\int_{B_{2R}} h^jZ_{j}=0$ for all $\tau$, then
$H_j= L_0^{-1}[h^j(\cdot, \tau) ]$ satisfies the estimate
$$
\| H_j\|_{a, \nu} \lesssim \| h_j \|_{2+ a, \nu} .
$$
Let us consider the boundary value problem in $B_{3R}$
\be\label{modo001}
 \Phi_\tau = \Delta \Phi   + pU(y)^{p-1} \Phi  + H_j(r)\Theta_j (y) \inn B_{3R}\times (\tau_0,\infty)
 \ee
$$
\Phi = 0 \onn  \pp B_{3R} \times (\tau_0,\infty)  ,\quad \Phi(\cdot, \tau_0 ) = 0 \inn B_{3R}.
$$
Invoking Lemma \ref{slow} we find a solution $\Phi_j[h]$ to this problem, which defines a linear operator of $h_j$ and satisfies the estimates
\be
  |\Phi_j(y,\tau) |  \ \lesssim   \  \frac{\tau^{-\nu} R^{n}} { 1+ |y|^{n-1}}  R^{1-a}  \| h_j\|_{2+ a, \nu} , \quad
\label{cota112}\ee
Arguing by scaling and parabolic estimates, we find as in the construction for mode 0,
$$
| L[\Phi_j](\cdot, \tau ) |\ \lesssim\  \tau^{-\nu} \|h\|_{2+a,\nu} \frac{R^{n+1-a}} {1+ |y|^{n+1}} \inn B_{2R}.
$$
We define
$$
\phi_j[h_j] := L[\Phi_j]\, \Big |_{B_{2R}}.
$$
Then  $\phi_j[h]$ solves the equation \equ{p11}  and satisfies
$$
| \phi_j[h_j] (y, \tau) | \ \lesssim \ \tau^{-\nu} \|h_j\|_{2+a,\nu} \frac{R^{n+1}} {1+ |y|^{n+1} } R^{-a} \inn B_{2R}.
$$
We then define
\be\label{fastmodo1}
\phi^1[h^1] := \sum_{j=1}^n \phi_j[h_j]\Theta_j .
\ee
This function solves \equ{p11} for $h= h^1$ and satisfies
\be\label{pp1}
| \phi^1[h^1] (y, \tau) | \ \lesssim \ \tau^{-\nu} \|h_j\|_{2+a,\nu} \frac{R^{n+1}} {1+ |y|^{n+1} } R^{-a} \inn B_{2R}.
\ee

\subsubsection*{\bf Higher modes.}
Now, for
$$h= h^\perp = \sum_{j=n+1}^\infty h_j(r) \Theta_j$$
we let $\phi^\perp[ h^\perp]$ be just the unique solution of the problem
\be\label{modoknew}
 \phi_\tau = \Delta \phi   + pU(y)^{p-1} \phi  + h^\perp  \inn B_{2R}\times (\tau_0,\infty)
 \ee
$$
\phi = 0 \onn  \pp B_{2R} \times (\tau_0,\infty)  ,\quad \phi(\cdot, \tau_0 ) = 0 \inn B_{2R},
$$
which 
is estimated, according to Lemma \ref{slow} as
\be\label{pport}
|\phi^\perp [h^\perp] (y , \tau )|    \ \lesssim\   \tau^{-\nu }\frac {\| h^\perp  \|_{a,\nu}}{ 1+|y|^{a}}\, \inn B_{2R}.
\ee

\subsubsection*{\bf Conclusion.}
We just let
$$
\phi [h]:= \phi^0[h^0] + \phi^1[h^1]+ \phi^\perp [h^\perp]
$$
be the functions constructed above. According to estimates \equ{pp0}, \equ{pp1} and \equ{pport} we find that
this function solves Problem  \equ{p11} for $c(\tau)$ given by \equ{cc}, with bounds \equ{cta1} and \equ{cta2} as required. The proof is concluded. \qed

\bigskip
{\it  Acknowledgments}.  This work has been supported by grants Fondecyt 1150028,  1160135, 115066,  Millennium Nucleus Center for Analysis of PDE NC13001, and Fondo Basal CMM.

\end{document}